\documentclass[12pt,a4paper]{article}

\usepackage{amsbsy,amsfonts,amsmath,amssymb,amsthm}
\usepackage{array}
\usepackage{bm}
\usepackage[dvipsnames]{xcolor}
\usepackage{enumerate}
\usepackage{mathrsfs}
\usepackage{latexsym}
\usepackage{mathtools}
\mathtoolsset{showmanualtags, showonlyrefs}

\definecolor{wineRed}{rgb}{0.7,0,0.3}
\definecolor{grandBleu}{rgb}{0,0,0.8}
\definecolor{Green}{rgb}{0,0.4,0}
\definecolor{blueViolet}{rgb}{0.4,0,1.0}
\definecolor{bloodOrange}{rgb}{0.85,0.05,0}
\newcommand{\KS}[1]{{\color{black}{#1}}}

\usepackage[colorlinks=true]{hyperref}
\hypersetup{urlcolor=blue, citecolor=blue, linkcolor=blue}

\usepackage{cite}

\usepackage{comment}

\DeclareMathAlphabet{\mathpzc}{OT1}{pzc}{m}{it}

\usepackage[textsize=small]{todonotes}
\numberwithin{equation}{section}

\theoremstyle{plain}
\newtheorem{mTh}{Main Theorem}
\newtheorem{thm}{Theorem}%[section] 
\newtheorem{lem}{Lemma}%[section]
\newtheorem{keyLem}{Key-Lemma}
\newtheorem{cor}{Corollary}%[section]

\theoremstyle{definition}
\newtheorem{defn}{Definition}%[section]
\newtheorem{rem}{Remark} 

\newtheorem{fact}{Fact}

\def\bB{{\mathbb B}}

\def\R{{\mathbb R}}

\def\bB{{\mathbb B}}
\def\N{{\mathbb{N}}}

\def\L{{\mathcal L}}

\def\ds{\displaystyle}
\def\ts{\textstyle}

\def\ds{\displaystyle}
\def\ts{\textstyle}

\def\Sgn{\mathop{\mathrm{Sgn}}\nolimits}

\newcommand{\trace}[1]{{_{|_{#1}}}}

\begin{document}
\vspace*{0.5cm}
\begin{center}
    \textbf{\LARGE{K}obayashi--{W}arren--{C}arter System of \\[-0.25ex] Singular Type under \\[0.5ex] Dynamic Boundary Condition}
    \footnote{
This work is supported by Grant-in-Aid for Scientific Research (C) No. 20K03672, JSPS.\\
AMS Subject Classification: 
35K51, %"Initial-boundary value problems for second-order parabolic systems" (MSC2020)
35K55, %"Nonlinear parabolic equations" (MSC2020)
35K61, %"Nonlinear initial, boundary and initial-boundary value problems for nonlinear parabolic equations" (MSC2020)
35K67, %"Singular parabolic equations" (MSC2020)
82C26. %: "Dynamic and nonequilibrium phase transitions (general) in statistical mechanics" (MSC2020)
\\
Keywords: 
KWC system of grain boundary motion, 
singular diffusion equation,
dynamic boundary condition, 
existence of solution, 
mathematical meanings of the non-standard
terms.
}
%\footnotetext[1]{aaa}
%\footnotetext[2]{bbb}
%\footnotemark[1]
    \vspace{1.5cm}

\textit{Dedicated to Professor Pierluigi Colli on the occasion of his 65th birthday}

    \vspace{1.5cm}
{\sc Ryota Nakayashiki}
%\footnotemark[2]
    \\[1ex]
Department of General Education, Salesian Polytechnic\\
4--6--8, Oyamagaoka, Machida, 194--0215, Tokyo, Japan.
{\ttfamily nakayashiki.ryota@salesio-sp.ac.jp}
    \\[1cm]
{\sc Ken Shirakawa}
    \\[1ex]
%\footnotemark[2]
Department of Mathematics, Faculty of Education, Chiba University\\ 
1--33, Yayoi-cho, Inage-ku, 263--8522, Chiba, Japan.\\
{\ttfamily sirakawa@faculty.chiba-u.jp}
\end{center}
\vspace{2cm}

\noindent
{\bf Abstract.} \indent
    In this paper, we consider a coupled system, known as \emph{Kobayashi--Warren--Carter system,}  abbreviated as the \emph{KWC system}.  KWC system consists of an Allen--Cahn type equation and a singular diffusion equation, and it was proposed by [Kobayashi et al, Phys. D, 140, 141--150 (2000)] as a possible mathematical model of grain boundary motion. The focus of this work is on the dynamic boundary condition imposed in our KWC system, and the mathematical interest is in a conflicting situation between: the continuity of the transmission condition included in the dynamic boundary condition; and the discontinuity encouraged by the singular diffusion equation. On this basis, we will prove the Main Theorem concerned with the existence of solution to our KWC system with energy-dissipation. Additionally, as a sub-result, we will prove a key-lemma that is to give a certain mathematical interpretation for the conflicting situation. 
\newpage

\section*{Introduction}
\vspace{-1ex}

Let $\kappa>0$, $ \kappa_\Gamma>0 $, $0<T<\infty$, and $1<N\in \N$ be fixed constants. Let $ (0, T) \subset \R $ be a fixed time-interval. Let $\Omega$ be a bounded spatial domain, with a smooth ($C^\infty$-) boundary $ \Gamma := \partial \Omega$, and let $n_\Gamma \in \mathbb{S}^{N-1}$ be the unit outer normal on $\Gamma$. Besides, we let $Q:=(0,T) \times \Omega$ and $\Sigma:=(0,T) \times \Gamma$.

  In this paper, we take a nonnegative constant $\varepsilon\ge0$ to consider the following coupled system of parabolic type PDEs, denoted by (KWC)$_{\varepsilon}$.
\medskip

\noindent
  (KWC)$_{\varepsilon}$:

    \begin{equation}\label{eta-eq}
\left\{\begin{aligned}
  & \partial_t \eta -\kappa^2\varDelta \eta + g(\eta) + \alpha'(\eta)|D\theta| =0 \mbox{ in $Q$,}%\label{1st.eq}
  \\[.5ex]
    & \partial_t \eta_\Gamma -\varDelta_\Gamma (\varepsilon^2\eta_\Gamma)+\nabla \eta\trace{\Gamma}\cdot n_\Gamma +g_\Gamma(\eta_\Gamma) +\alpha'(\eta_\Gamma)|\theta\trace{\Gamma} -\theta_\Gamma|=0, 
    \\
    & \hspace{4ex}\mbox{and } \eta\trace{\Gamma}= \eta_\Gamma \mbox{ on $\Sigma$,}%\label{Neu.BC}
  \\[.5ex]
  & \eta(0,\cdot)=\eta_0 \mbox{ in $\Omega $, and } \eta_\Gamma(0,\cdot)=\eta_{\Gamma,0},
  \mbox{ on $\Gamma$.}
\end{aligned}
\right.
\end{equation}

\begin{equation}\label{theta-eq}
  \hspace{-.35cm} \left\{\begin{aligned}
& \alpha_0(\eta)\partial_t \theta - \mathrm{div}\left(\alpha(\eta)\dfrac{D \theta}{|D \theta|} \right)=0 \mbox{ in $Q$,}%\label{2nd.eq}
\\[.5ex]
      & \alpha_{\Gamma, 0}(\eta) \partial_t \theta_\Gamma -\kappa_\Gamma^2 \varDelta_\Gamma \theta_\Gamma + \ts{(\alpha(\eta)\frac{D\theta}{|D\theta|})\trace{\Gamma} \cdot n_\Gamma} =0, \mbox{ and } \theta\trace{\Gamma}= \theta_\Gamma \mbox{ on $\Sigma$,}%\label{dyn.BC}
\\[.5ex]
 &\theta(0,\cdot)=\theta_0 \mbox{ in $\Omega $, and }  \theta_\Gamma(0,\cdot) = \theta_{\Gamma,0} \mbox{ on $\Gamma$.}%\label{IC}  
\end{aligned}
\right.
\end{equation}

The above system is based on Kobayashi--Warren--Carter system, which was proposed by Kobayashi et al \cite{MR1752970, MR1794359} as a phase-field model of grain boundary motion in a polycrystal. Our system (KWC)$_{\varepsilon}$ is one of modified versions of the original system, and the principal modification is in the point that the boundary data $\eta_\Gamma$ and $\theta_\Gamma$ are treated as unknown variables, governed by the dynamic boundary conditions on $\Sigma$.
%In the original model \cite{MR1752970, MR1794359}, the spatial domain $\Omega$ is settled as a two-dimensional domain ($N=2$), and the main focus is to reproduce the dynamics of the crystalline orientation by the time and spatial variation of a vector field:
%\begin{equation*}
%(t,x)\in Q \mapsto \varpi(t,x):=\eta(t,x){}^{\mathrm{t}}\bigl(\cos \theta(t,x), \sin \theta(t,x)\bigr),
%\end{equation*}
%consisting of two order parameters $\eta=\eta(t,x)$ and $\theta=\theta(t,x)$. The variation of $\varpi=\varpi(t,x)$ is supposed to be governed by 

Referring to the modelling method of \cite{MR1752970, MR1794359},  the system (KWC)$_\varepsilon$ is derived as a gradient flow of the following energy functional, called \emph{free-energy}:%, and for any $\varepsilon\ge0$, the free-energy for (KWC)$_{\varepsilon}$ is provided as follows.
\begin{equation}\label{Free-Ener}
\begin{array}{ll}
    \multicolumn{2}{l}{\ds [\bm{\eta}, \bm{\theta}] = [\eta,\eta_\Gamma, \theta,\theta_\Gamma]\in D(\mathscr{F}_{\varepsilon}) \mapsto \mathscr{F}_{\varepsilon}(\bm{\eta}, \bm{\theta}) = \mathscr{F}_{\varepsilon}(\eta, \eta_\Gamma,\theta,\theta_\Gamma)}
    \\[2ex]
    & \ds := \frac{\kappa^2}{2}\int_\Omega |\nabla \eta|^2\,dx + \int_\Omega \hat{g}(\eta) \,dx +\int_\Gamma |\nabla_\Gamma (\varepsilon \eta_\Gamma)|^2 \, d\Gamma +\int_\Gamma \hat{g}_\Gamma(\eta_\Gamma) \, d \Gamma
\\[2ex]
    & \qquad \ds+ \int_\Omega \alpha(\eta)|D \theta| +\int_\Omega \alpha(\eta_{\Gamma}) |\theta\trace{\Gamma} -\theta_\Gamma| \, d \Gamma + \frac{\kappa_\Gamma^2}{2}\int_\Gamma |\nabla_\Gamma \theta_\Gamma|^2\,d\Gamma \in[0,\infty],
\end{array}
\end{equation}
with the effective domain:
\begin{equation*}%\label{eff.F_e}
    D(\mathscr{F}_{\varepsilon}):=\left\{\begin{array}{l|l}[\bm{\eta}, \bm{\theta}] & \hspace{-3ex} \parbox{9cm}{
        \vspace{-2.5ex}
        \begin{itemize}
            \item$ \bm{\eta} = [\eta, \eta_{\Gamma}] \in H^1(\Omega) \times H^{\frac{1}{2}}(\Gamma) $, $ \varepsilon \eta_\Gamma \in H^1(\Gamma) $, and $\eta\trace{\Gamma}=\eta_\Gamma$ in $H^{\frac12}(\Gamma)$
                \vspace{-1ex}
            \item $ \bm{\theta} = [\theta, \theta_\Gamma] \in \bigl( BV(\Omega) \cap L^2(\Omega) \bigr) \times H^1(\Gamma)$
        \end{itemize}
    \vspace{-2.5ex}
}\end{array}\right\}.
\end{equation*}
In the context, ``$\trace{\Gamma}$" denotes the trace on $\Gamma$ for a Sobolev / BV function, $d\Gamma$ denotes the area element on $\Gamma$, $\nabla_\Gamma$ denotes the surface gradient on $\Gamma$, and $\varDelta_\Gamma$ denotes the Laplacian on on $ \Gamma $, i.e. the so-called Laplace--Beltrami operator (cf. \cite{MR1484545}). On this basis, the unknowns $ \bm{\eta} = [\eta, \eta_\Gamma] $ and $ \bm{\theta}= [\theta, \theta_\Gamma] $ are order parameters of the orientation order and the orientation angle in a  polycrystal, respectively. The components $ \eta $ and $ \eta_\Gamma $ (resp. $ \theta $ and $ \theta_\Gamma $) are to reproduce the dynamic exchanges of physical phases between $ Q $ and $ \Sigma $, via the transmission condition $ \eta \trace{\Gamma} = \eta_\Gamma $ in \eqref{eta-eq} (resp. $ \theta\trace{\Gamma} = \theta_\Gamma $ in \eqref{theta-eq}). In particular, the components $\eta$ and $ \eta_\Gamma $ are supposed to take values on $ [0, 1] $ and the threshold values $1$ and $0$ indicate the completely oriented phase and the disoriented phase of orientation, respectively. $g=g(\eta)$ and $g_\Gamma=g_\Gamma(\eta_\Gamma)$ in \eqref{eta-eq} are given Lipschitz continuous functions on $\R$, and $\hat{g}=\hat{g}(\eta)$ and $\hat{g}=\hat{g}_\Gamma (\eta_\Gamma)$ are the primitives of $ g $ and $ \hat{g} $, respectively. $0<\alpha_0=\alpha_0(\eta)$ and $ 0 < \alpha_{\Gamma, 0} = \alpha_{\Gamma, 0}(\eta_\Gamma) $ in \eqref{theta-eq} are a given locally Lipschitz functions. $0<\alpha=\alpha(\eta)$ in \eqref{eta-eq}--\eqref{theta-eq} is a given $C^2$-convex function, and $\alpha'=\alpha'(\eta)$ is the differential of $\alpha$. $\eta_0$, $\eta_{\Gamma,0}$, $\theta_0$, and $\theta_{\Gamma,0}$ are given initial data for the components $\eta,\eta_\Gamma$, $\theta_\Gamma$ and $\theta_\Gamma$, respectively. 

%The objective of this study is to develop the mathematical analysis for the Kobayashi--Warren--Carter systems, and the main issue of this paper is concerned with the qualitative properties of the systems (KWC)$_{\varepsilon}$, for any $\varepsilon \ge0$. 

In recent years, a vast number of studies have addressed mathematical models under dynamic-boundary conditions (cf. \cite{MR2469586, MR2548486, MR2836555, MR2668289, MR3268865, MR3670006, MR3038131, MR3082861, MR3462536, MR3362773, MR2577805, MR3281854, cherfils2012long, MR2629535, MR3306048, MR2861821, MR3130137, MR2509552, MR3506285, MR2275977, MR3432081, MR3485909, MR4345548, MR4303582, MR4529496}). In the line of these previous works, the mathematical method of Kobayashi--Warren--Carter type systems have been studied by a lot of mathematicians (cf. \cite{MR2469586, MR2548486, MR2836555, MR2668289, MR3268865, MR3670006, MR3038131, MR3082861, MR3462536, MR3362773}), and a number of qualitative results for $L^2$-based solutions have been obtained from various viewpoints. In most of these, the boundary condition is considered only in a few simple cases, such as time-independent Dirichlet / Neumann cases, and the variety of the boundary conditions, such as dynamic cases, were not focused so much.

On the other hand, in the study of single PDE, there have been previous works \cite{MR3661429,MR4150355} which dealt with singular type diffusion equations, under dynamic boundary condition. In the recent works, the qualitative results of the systems, such as the existence, uniqueness, and continuous dependence of solution, e.t.c., were obtained by means of general theory of nonlinear evolution equation. 

Based on the background, we set the objective in this paper to establish the mathematical theory for the parabolic systems (KWC)$_\varepsilon $, for $ \varepsilon \geq 0 $, under dynamic boundary condition, that are governed by singular diffusion $ -\mathrm{div} \bigl( \alpha(\eta) \frac{D \theta}{|D \theta|} \bigr) $ with unknown-dependent weight $ \alpha(\eta) $. To this end, we will prove the following Main Theorem, as the principal result of this paper.
\begin{description}
\item[Main Theorem 1:]The existence of solution to (KWC)$_{\varepsilon}$ with energy-dissipation. %and the uniqueness in the constant case of $\alpha_0$.
%\item[Main Theorem 2:]$\varepsilon$-dependence of solution to  (KWC)$_{\varepsilon}$.
\end{description}
    Furthermore, as a sub-result, the rigorous mathematical meanings of the non-standard terms $ -\mathrm{div} \bigl( \alpha(\eta) \frac{D \theta}{|D \theta|} \bigr) $, $ \ts{(\alpha(\eta)\frac{D\theta}{|D\theta|})\trace{\Gamma} \cdot n_\Gamma} $, and the transmission condition $ \theta\trace{\Gamma} = \theta_\Gamma  $, will be clarified through the proof of a key-lemma. This sub-result will give a  certain mathematical interpretation for the conflicted situation between: the continuity of the transmission condition included in the dynamic boundary condition; and the discontinuity encouraged by the singular diffusion equation. 
\medskip

Here is the content of this paper. In Section 1, we list preliminaries which are used throughout this paper. In Section 2, we state the Main Theorem 1, with the assumptions and the definition of the solutions to (KWC)$_{\varepsilon}$. In Sections 3--4, some key-properties will be verified as groundworks for the proofs of the Main Theorems, and in particular, the sub-result of this paper will be stated as Key-Lemma \ref{KeyLem01} discussed in Section 3. The proof of Main Theorem 1 is discussed in Section 5, and it is discussed on the basis of the key-properties, discussed in Section 3 and 4. The final Section 6 is an appendix that is to supply general information of the specific notions used in this paper. 
\vspace{-1ex}

\section{Preliminaries}
\vspace{-1ex}

In this section, we outline some notations and known facts, as preliminaries of our study. 
For any dimension $ d \in \N $, the $ d $-dimensional Lebesgue measure is denoted by $ \L^d $. Unless otherwise specified, the measure theoretical phrases, such as ``a.e.'', ``$ dt $'', ``$ dx $'', and so on, are with respect to the Lebesgue measure in each corresponding dimension. Also, in the observations on a smooth surface $ S $, the phrase ``a.e.'' is with respect to the Hausdorff measure in each corresponding Hausdorff dimension, and the area element on $ S $ is denoted by $ dS $. %\begin{notn}\label{notn.omega}
\bigskip

    Throughout this paper, let $ 0 < T < \infty $ and $1<N \in \N$ be fixed constants, and let $\Omega \subset \R^N$ be a bounded domain, such that:  
\begin{description}
\item[($\bm{\omega}$1)]$\Omega$ has a $C^\infty$-boundary $\Gamma:=\partial\Omega$;
\item[($\bm{\omega}$2)]the function $\ds \mathrm{dist}_\Gamma: x\in \overline{\Omega} \mapsto \mathrm{dist}_\Gamma(x):=\inf_{y\in \Gamma}|x-y| \in [0,\infty)$ forms a $C^\infty$-function on a neighborhood of $\Gamma$.
\end{description}
\KS{
Also, we write ``\,$\trace{\Gamma}$\,'' to denote the trace of Sobolev / BV function on $ \Omega $ to $ \Gamma $. On this basis, we let:
}
\begin{align}
    & H := L^2(\Omega) \times L^2(\Gamma), ~~ 
     ~~\mbox{and}~~ W := (BV(\Omega) \cap L^2(\Omega)) \times H^1(\Gamma),
\end{align}
and for any $ \varepsilon \geq 0 $, we define:
\begin{align}\label{V_e1}
    & V_\varepsilon:=\left\{\begin{array}{l|l} v=[\xi,\xi_\Gamma] \in H^1(\Omega)\times H^{\frac12}(\Gamma) & \parbox{3.25cm}{
            $\varepsilon \xi_\Gamma \in H^1(\Gamma) $ and $ \xi\trace{\Gamma} =\xi_\Gamma $ in $H^{\frac12}(\Gamma)$}\end{array}\right\}.
\end{align}
Additionally, we let 
    \begin{align}
        & L^\infty(\Omega, \Gamma) := L^\infty(\Omega) \times L^\infty(\Gamma),
    \end{align}
    and  for every $ -\infty \leq a \leq b \leq \infty $, we define
\begin{align}
    & \Lambda_a^b := \left\{ \begin{array}{l|l}
        [z, z_\Gamma] \in H & \parbox{8.25cm}{
            $ a \leq z \leq b $ a.e. in $ \Omega $, and $ a \leq z_\Gamma \leq b $ a.e. on $  \Gamma $
        }
    \end{array} \right\}. 
\end{align}
\begin{rem}\label{Rem.V_e}
If $\varepsilon>0$ (resp. $\varepsilon =0$), then $V_\varepsilon$, given in \eqref{V_e1} is closed linear space in $H^1(\Omega) \times H^1(\Gamma)$ (resp. $H^1(\Omega) \times H^{\frac12}(\Gamma)$) and hence, they are Hilbert space endowed with the standard inner products of $H^1(\Omega) \times H^1(\Gamma)$ (resp. $H^1(\Omega) \times H^{\frac12}(\Gamma)$).
\end{rem}

\subsection{Abstract notations}
For an abstract Banach space $ X $, we denote by $ |{}\cdot{}|_X $ the norm of $ X $, and denote by $ \langle{}\cdot{}, {}\cdot{}\rangle_X $ the duality pairing between $ X $ and the dual space $ X^* $ of $ X $. Let $ \mathcal{I}_X : X \to X $ be the identity map from $ X $ onto $ X $. In particular, when $ X $ is a Hilbert space, we denote by $ ({}\cdot{},{}\cdot{})_X $ the inner product in $ X $.

For Banach spaces $X_1,\cdots,X_d$ with $1<d \in \N$, let $X_1 \times \cdots X_d$ be the product Banach space endowed with the norm $|{}\cdot{}|_{X_1\times \cdots \times X_d}:=|{}\cdot{}|_{X_1}+\cdots +|{}\cdot{}|_{X_d}$. However, when all $X_1,\cdots,X_d$ are Hilbert spaces, $X_1 \times \cdots \times X_d$ denotes the product Hilbert space endowed with the inner product $({}\cdot{},{}\cdot{})_{X_1\times \cdots \times X_d}$ and the norm $|{}\cdot{}|_{X_1\times \cdots \times X_d}:=(|{}\cdot{}|_{X_1}^2+\cdots +|{}\cdot{}|_{X_d}^2)^{\frac12}$.

Also, let $X$ be an abstract Banach space, and let $\mathcal{A}:X\to 2^X$ be a set-valued operator. Then, we define $D(\mathcal{A}):=\{z\in X {}\,|\,{}\mathcal{A}z\neq \emptyset \}$ as the domain of $\mathcal{A}$.

\subsection{Notations in real analysis}%\label{notn.real}
For arbitrary $ a, b \in [-\infty, \infty] $, we define 
%\begin{equation*}
$ a \vee b := \max \{ a, b \} \mbox{ and } a \wedge b := \min \{ a, b \} $, 
%\end{equation*}
and especially, we write $[a]^+:= a \vee 0$ and $[b]^-:=-(b \wedge 0)$, respectively.

Let $ d \in \N $ be any fixed dimension. Then, we simply denote by $ |x| $ and $ x \cdot y $ the Euclidean norm of $ x \in \R^d $ and the standard scalar product of  $ x, y \in \R^d $, respectively. Also, we denote by $ \mathbb{B}^d $ and $ \mathbb{S}^{d -1} $ the $ d $-dimensional unit open ball centered at the origin, and its boundary, respectively, i.e.:
\begin{equation*}
\mathbb{B}^d := \left\{ \begin{array}{l|l}
x \in \R^d & |x| < 1
\end{array} \right\} \mbox{ and } \mathbb{S}^{d -1} := \left\{ \begin{array}{l|l}
x \in \R^d & |x| = 1
\end{array} \right\}.
\end{equation*}
In particular, when $d>1$, we write $x \le y$, if $x_i \le y_i$, for all $i=1,\cdots, d$, and define:
\begin{equation*}
\left\{\parbox{11cm}{
$x \vee y:=[x_1 \vee y_1,\dots, x_d \vee y_d]$, $x \wedge y:=[x_1 \wedge y_1,\dots, x_d \wedge y_d]$,
\\[0.5ex]
$[x]^+:=\bigl[[x_1]^+,\dots, [x_d]^+\bigr]$ and $[y]^-:=\bigl[[y_1]^-,\dots, [y_d]^-\bigr]$,
}\right. \mbox{ for all $x,y \in \R^d$.}
\end{equation*}

Additionally, we note the following elementary fact.
\begin{fact}\label{Fact0}
Let $ m \in \N $ be a fixed finite number. If $ \{ A_1, \dots, A_m \} \subset \R $ and $ \{ a_n^k \}_{n = 1}^{\infty} $, $ k= 1,\dots, m $, fulfill that:
    \vspace{-2ex}
\begin{equation*}
\varliminf_{n \to \infty} a_n^k \geq A_k \mbox{, $ k=1, \dots, m $, and } \varlimsup_{n \to \infty} \sum_{k=1}^{m} a_n^k \leq \sum_{k=1}^m A_k.
\end{equation*}
    \vspace{-2ex}
Then, it holds that ~~$ \ds \lim_{n \to \infty} a_n^k = A_k $, $ k = 1, \dots, m $.
\end{fact}

Let $ d, m \in \N $ be fixed dimensions. For an open set $ U \subset \R^d $, $\sigma$-algebra $ \mathfrak{M}(U) \subset 2^{\R^d} $ on $ U $, and a measure $ \mu : \mathfrak{M}(U) \longrightarrow \R^m $, let $ |\mu| : \mathfrak{M}(U) \longrightarrow [0, \infty] $ be the total variation of $ \mu $. Also, for a positive measure $ \mu_0 : \mathfrak{M}(U) \longrightarrow [0, \infty] $, let $ \frac{\mu}{\mu_0} : U \longrightarrow \mathbb{S}^{d -1} $ be the Radon--Nikod\'{y}m density of the measure $ \mu $ for $ \mu_0 $. 

\subsection{Notations in variational analysis}
For any proper lower semi-continuous (l.s.c. from now on) and convex function $ \Psi:X\to (-\infty,\infty] $ defined on a Hilbert space $ X $, we denote by $ D(\Psi) $ its effective domain, and denote by $ \partial \Psi $ its subdifferential. The subdifferential $ \partial \Psi $ is a set-valued map corresponding to a weak differential of $ \Psi $, and it has a maximal monotone graph in the product space $ X \times X $. More precisely, for each $ z_0 \in X $, the value $ \partial \Psi(z_0) $ is defined as a set of all elements $ z_0^* \in X $ which satisfy the following variational inequality:
\begin{equation*}
(z_0^*, z -z_0)_X \leq \Psi(z) -\Psi(z_0) \mbox{, for any $ z \in D(\Psi) $.}
\end{equation*}
%The set \$ D(\partial \Psi) := \{ z \in X \,|\, \partial \Psi(z) \ne \emptyset \} $ \is called the domain of $ \partial \Psi $. 
We often use the notation ``$ [z_0, z_0^*] \in \partial \Psi $ in $ X \times X $\,'', to mean that ``$ z_0^* \in \partial \Psi(z_0) $ in $ X $ with $ z_0 \in D(\partial \Psi) $'', by identifying the operator $ \partial \Psi $ with its graph in $ X \times X $.

\begin{rem}[Examples of the subdifferentials]\label{exConvex}
    ~~~~
\begin{description}
    \item[(I)]As one of representatives of the subdifferentials, we exemplify the following set-valued function $ \Sgn^d : \R^d \rightarrow 2^{\R^d} $, for each dimension $ d \in \N $, which is given as:
\begin{equation*}
\omega \in \R^N \mapsto \Sgn^d(\omega) := \left\{ \begin{array}{ll}
\ds \frac{\omega}{|\omega|}, & \mbox{if $ \omega \ne 0 $,}
\\[2.5ex]
\overline{\mathbb{B}^d}, & \mbox{otherwise.}
\end{array} \right.
\end{equation*}
It is known that the set-valued function $ \Sgn^d $ coincides with the subdifferential of the Euclidean norm $ |{}\cdot{}| : \omega \in \R^d \mapsto |\omega| = \sqrt{\omega \cdot \omega} \in [0, \infty) $, i.e.:
\begin{equation*}
\partial |{}\cdot{}|(\omega) = \Sgn^d(\omega), \mbox{ for any $ \omega \in D(\partial |{}\cdot{}|) = \R^d$.}
\end{equation*}
\item[(II)]%(cf. \cite{MR2582280,MR0348562}) 
    Let us define proper l.s.c. and convex function $ \Psi_{\Omega, \Gamma} : H \longrightarrow [0, \infty] $ by letting:
        \begin{align}\label{Psi_OmgGm}
            \Psi_{\Omega, \Gamma} ~:& [z, z_\Gamma] \in H \mapsto \Psi_{\Omega, \Gamma}(z, z_\Gamma) := \left\{ \begin{array}{ll}
                \multicolumn{2}{l}{\ds \frac{1}{2} \int_\Omega |\nabla z|^2 \, dx +\frac{1}{2} \int_\Gamma |\nabla z_\Gamma|^2 \, d \Gamma,}
                \\[2ex]
                & \mbox{if $ [z, z_\Gamma] \in V_1 $,}
                \\[2ex]
                \infty, & \mbox{otherwise.}
            \end{array} \right.
        \end{align}
        Then, the subdifferential $ \partial \Psi_{\Omega, \Gamma} \subset H \times H $ is computed as a single valued operator, formulated as follows:
        \begin{align}
            \partial \Psi_{\Omega, \Gamma} : ~& [z, z_\Gamma] \in D(\partial \Psi_{\Omega, \Gamma}) = V_1 \cap \bigl( H^2(\Omega) \times H^2(\Gamma) \bigr) \subset H 
            \\
            &\mapsto \Psi_{\Omega, \Gamma}(z, z_\Gamma)  = \left[ \begin{array}{c}
                -\varDelta z \\ -\varDelta_\Gamma z_\Gamma +\KS{(\nabla z)\trace{\Gamma}} \cdot n_\Gamma
            \end{array} \right] \in H.
        \end{align}
\end{description}
\end{rem}

Finally, we recall a notion of functional-convergence, known as \emph{$ \mathit{\Gamma} $-convergence.}
 
\begin{defn}[$\Gamma$-convergence: cf. \cite{MR1201152}]\label{Def.Gamma}
    Let $ X $ be an abstract Hilbert space. Let $ \Psi : X \rightarrow (-\infty, \infty] $ be a proper functional, and let $ \{ \Psi_n \}_{n = 1}^\infty $ be a sequence of proper functionals $ \Psi_n : X \rightarrow (-\infty, \infty] $, $ n \in \N $.  Then, it is said that $ \{ \Psi_n \}_{n = 1}^\infty $ $ \Gamma $-converges to $ \Psi $ on $ X $, iff. the following two conditions are fulfilled.
\begin{description}
    \item[\hypertarget{G_lb}{$\bm{\Gamma}$1}) Lower-bound condition:]$ \ds \varliminf_{n \to \infty} \Psi_n(\check{z}_n) \geq \Psi(\check{z}) $, if $ \check{z} \in X $, $ \{ \check{z}_n \}_{n = 1}^\infty \subset X $, and $ \check{z}_n \to \check{z} $ in $ X $ as $ n \to \infty $.
    \item[(\hypertarget{G_opt}{$\bm{\Gamma}$2}) Optimality condition:]for any $ \hat{z} \in D(\Psi) $, there exists a sequence $ \{ \hat{z}_n \}_{n = 1}^\infty \subset X $ such that $ \hat{z}_n \to \hat{z} $ in $ X $ and $ \Psi_n(\hat{z}_n) \to \Psi(\hat{z}) $, as $ n \to \infty $.
\end{description} 
\end{defn}
In the $ \Gamma $-convergence of convex functions, we can easily see the following fact. %As a basic matter of the Mosco-convergence, we can see the following fact (see \cite[Theorem 3.66]{MR2192832}, \cite[Chapter 2]{Kenmochi81} or \cite[Remark 1.5 (Fact\,7)]{MR3362773}, for example).
\begin{fact}\label{Rem.Gamma}
    Let $ \Psi : X \longrightarrow (-\infty, \infty] $ and $ \Psi_n : X \longrightarrow (-\infty, \infty] $, $ n \in \N $, be proper l.s.c. and convex functions on a Hilbert space $ X $, such that $ \{ \Psi_n \}_{n \in N} $ $ \Gamma $-converges to $ \Psi $ on $ X $ as $ n \to \infty $. Besides, let us assume that:
\begin{center}
$ \Psi_n \to \Psi $ on $ X $, in the sense of Mosco, as $ n \to \infty $,
\vspace{-1ex}
\end{center}
and
\begin{equation*}
\left\{ ~ \parbox{10cm}{
    $ [z, z^*] \in X \times X $, $ [z_n, z_n^*] \in \partial \Psi_n $ in $ X \times X $, for all $ n \in \N $,
    \\[1ex]
    $ z_n \to z $ in $ X $ and $ z_n^* \to z^* $ weakly in $ X $, as $ n \to \infty $.
} \right.
\end{equation*}
Then, it holds that:
\begin{equation*}
[z, z^*] \in \partial \Psi \mbox{ in $ X \times X $, and } \Psi_n(z_n) \to \Psi(z) \mbox{, as $ n \to \infty $.}
\end{equation*}
\end{fact}

\subsection{Specific notations}
We fix a large open ball $ \bB_\Omega \subset \R^N $ such that $ \bB_\Omega \supset \overline{\Omega}$.
Besides, for any $ u \in BV(\Omega) \cap L^2(\Omega) $ and any $ \gamma \in H^{\frac{1}{2}}(\Gamma) $, we define a function $ [u]_\gamma^{\rm ex} \in BV(\R^N) \cap W^{1, 1}(\R^N \setminus \overline{\Omega}) $, by letting
\begin{equation*}
    \begin{array}{c}
        x \in \R^N \mapsto [u]_\gamma^{\rm ex}(x) := \left\{ \begin{array}{l}
            [u]^*(x), \mbox{ if $ x \in \Omega $,}
            \\[1ex]
            [\gamma]^{\rm ex}(x), \mbox{ if $ x \in \R^N \setminus \overline{\Omega} $,}
        \end{array} \right.
        \\[3ex]
        \mbox{with $ [\gamma]^{\rm ex} := [\mathfrak{ex}_{\Omega} \circ \mathfrak{hm}_{\Gamma}] (\gamma) $ in $ H^1(\mathbb{R}^N) \cap W^{1, 1}(\R^N) $,}
    \end{array}
\end{equation*}
    where $ \mathfrak{ex}_\Omega : BV(\Omega) \longrightarrow BV(\R^N) $ is the extension operator as in Appendix (Fact\,1), and $ \mathfrak{hm}_\Gamma : H^{\frac{1}{2}}(\Gamma) \longrightarrow H^1(\Omega) $ is the harmonic extension as in Remark \ref{Rem.ex_hm} \KS{(ex.0)} in Appendix. 

For any $ \beta \in H^1(\Omega) \cap L^\infty(\Omega) $, any $ \gamma \in L^1(\Gamma) $ and any $ u \in BV(\Omega) \cap L^2(\Omega) $, we define a Radon measure $ [\beta |D u|]_\gamma \in \mathscr{M}_{\rm loc}(\R^N) $ by letting:
\begin{equation}\label{01_[bt|Du|]_gm}
\begin{array}{c}
\displaystyle
[\beta |Du|]_\gamma(B) := \int_{B \cap \Omega} [\beta]^* d |Du| +\int_{B \cap \Gamma} \beta |u -\gamma| \, d \Gamma +\int_{B \setminus \overline{\Omega}} [\beta]^{\rm ex}|D [\gamma]^{\rm ex}| \, dx,
\\[2ex]
\mbox{for any bounded Borel set $ B \subset \R^N $,}
\end{array}
\end{equation}
\KS{where $ [\beta]^\mathrm{ex} \in H^1(\R^N) $ is the extension of $ \beta $ as in Remark \ref{Rem.ex_hm} (ex.2) in Appendix. Here, referring to the previous works, e.g. \cite[Theorem 6.1]{MR2306643}, we can reduce the measure $ [\beta|Du|]_\gamma $ in the following simple form:}
\begin{equation}\label{02_[bt|Du|]_gm}
[\beta |Du|]_\gamma = [\beta]^{\rm ex} \, |D[u]_\gamma^{\rm ex}| \mbox{ \ in $ \mathscr{M}_{\rm loc}(\R^N) $.}
\end{equation}
In addition, we can see the following facts.
\KS{
\begin{fact}\label{Fact2}
    (cf. \cite[Theorem 3.88]{MR1857292})
    If $ \{ u_n \}_{n = 1}^\infty \subset BV(\Omega) \cap L^2(\Omega) $, $ u \in BV(\Omega) \cap L^2(\Omega) $ and $ u_n \to u $ in $ L^2(\Omega) $ and strictly in $ BV(\Omega) $ as $ n \to \infty $, then $ [\beta|Du_n|]_\gamma(\overline{\Omega}) \to [\beta |Du|]_\gamma(\overline{\Omega}) $ as $ n \to \infty $.
\end{fact}
}
\begin{fact}\label{Fact3}
(cf. \cite[Theorem 6.1]{MR2306643} and \cite[Section 2]{MR3268865})
    If $ A \subset \R^N $ is bounded and open, and $ \beta \in H^1(\Omega) \cap L^\infty(\Omega) $ with $ \beta \geq 0 $ a.e. in $ \Omega $, then
\begin{equation}\label{exp01}
[\beta|Du|]_\gamma(A) = \sup \left\{ \begin{array}{l|l}
    \ds \int_A [u]_\gamma^\mathrm{ex} \, {\rm div} \, \varphi \, dx & \parbox{4.75cm}{
$ \varphi \in L^\infty(A)^N $, such that $ {\rm supp} \, \varphi $ is compact and $ |\varphi| \leq [\beta]^{\rm ex} $ a.e. in $ A $
}
\end{array} \right\},
\end{equation}
and in particular, if $ \log \beta \in L^\infty(\Omega) $, then
\begin{equation}\label{exp02}
[\beta|Du|]_\gamma(A) = \sup \left\{ \begin{array}{l|l}
    \ds \int_A [u]_\gamma^\mathrm{ex} \, {\rm div}(\beta \varphi) \, dx & \parbox{4.75cm}{
$ \varphi \in L^\infty(A)^N $, such that $ {\rm supp} \, \varphi $ is compact and $ |\varphi| \leq 1 $ a.e. in $ A $}
\end{array} \right\}.
\end{equation}
\end{fact}
%\end{notn}

Now, we set:
\begin{equation*}
\begin{array}{c}
\left\{ \begin{array}{l}
\ds X_{\rm c}(\Omega) := \left\{ \begin{array}{l|l}
    \varphi \in L^{\infty}(\Omega)^{N} & {\rm div} \, \varphi \in L^{2}(\Omega) \mbox{ and $ {\rm supp} \, \varphi$ is compact in } \Omega
\end{array} \right\},
\\[1ex]
\ds Y_{0}(\Omega) := \left\{ \begin{array}{l|l}
\beta \in H^{1}(\Omega) \cap L^{\infty}(\Omega) &
\parbox{2.75cm}{$ \beta \ge 0 $ a.e. in $ \Omega $}
\end{array} \right\},
\\[1ex]
\ds Y_{\rm c}(\Omega) := \left\{ \begin{array}{l|l}
\beta \in H^{1}(\Omega) \cap L^{\infty}(\Omega) &
\parbox{2.75cm}{$ \log \beta \in L^\infty(\Omega) $}
\end{array} \right\}.
\end{array} \right.
\ \\[-2ex]
\end{array}
\end{equation*}

Besides, for every $ \beta \in Y_0(\Omega) $ and $ \gamma \in H^{\frac{1}{2}}(\Gamma) $, we define a functional {$\mathrm{Var}_{\gamma}^{\beta}$} on $L^{2}(\Omega)$, by letting
\begin{equation}\label{Phi_gm(bt.)}
    \ds u \in L^{2}(\Omega) \mapsto {\mathrm{Var}_{\gamma}^{\beta}(u)} := \left\{ \begin{array}{ll}
\multicolumn{2}{l}{\ds \int_{\overline{\Omega}} d [\beta |Du|]_\gamma = \int_\Omega [\beta]^* \, d |Du| +\int_{\Gamma} \beta |u -\gamma| \, d\Gamma,}
\\[2ex]
& \mbox{if $ u \in BV(\Omega) $,}
\\[1ex]
\infty, & \mbox{otherwise,}
\end{array} \right.
\end{equation}
where $ [{}\cdot{}]^* $ denotes the \emph{precise representative} of $ BV $-function (cf. \cite[Corollary 3.80]{MR1857292}, and Section 6.1 in Appendix).  
Note that the value of {$\mathrm{Var}_{\gamma}^{\beta}$} is determined, independently on the choice of the extensions $ [\beta]^{\rm ex} $ and $ [\gamma]^{\rm ex} $. Therefore, taking into account (\ref{01_[bt|Du|]_gm}) and (\ref{Phi_gm(bt.)}), it is inferred that the functional $\mathrm{Var}_{\gamma}^{\beta}$ is proper l.s.c. and convex on $ L^2(\Omega) $.
In addition, the following holds.

\begin{fact}\label{Fact4}
(cf. \cite{MR0102739,MR2139257,MR3751650}) If $ \beta \in Y_{\rm c}(\Omega) $, $ \gamma \in H^{1}(\Gamma) $ and $ u \in BV(\Omega) \cap L^2(\Omega) $, then there exists a sequence $ \{ u_i \}_{i = 1}^\infty \subset H^1(\Omega) $ such that
\begin{equation}\label{reg01}
\left\{ \hspace{-4ex} \parbox{11.75cm}{
\vspace{-1ex}
\begin{itemize}
\item[]$ u_i = \gamma $, a.e. on $ \Gamma $, for any $ i \in \N $,
\item[]$ u_i \to u $ in $ L^2(\Omega) $ and $ \ds \int_\Omega \beta |\nabla u_i| \, dx \to {\mathrm{Var}_{\gamma}^{\beta}(u)} $, as $ i \to \infty $.
\vspace{-1ex}
\end{itemize}
} \right.
\end{equation}
    Hence, the functional $ {\mathrm{Var}_{\gamma}^{\beta}} $ coincides with the lower semi-continuous envelope of the functional:
$$
\varphi \in L^2(\Omega) \mapsto \left\{ \begin{array}{ll}
        \multicolumn{2}{l}{\ds \int_\Omega \beta |\nabla \varphi| \, dx, \mbox{ \ if {$ \varphi \in H^{1}(\Omega) $} and $ \varphi = \gamma $, a.e. on $ \Gamma $,}}
\\[2ex]
\infty, & \mbox{otherwise,}
\end{array} \right.
$$
i.e. for any $ v \in L^2(\Omega) $, it holds that
$$
    {\mathrm{Var}_{\gamma}^{\beta}(v)} = \inf \left\{ \begin{array}{l|l}
\ds \liminf_{i \to \infty} \int_\Omega \beta |\nabla \varphi_i| \, dx & \parbox{7cm}{
        {$ \{ \varphi_i \}_{i = 1}^\infty \subset H^{1}(\Omega) $,} $ \varphi_i = \gamma $, a.e. on $ \Gamma $, for any $ i \in \N $, and $ \varphi_i \to v $ in $ L^2(\Omega) $ as $ i \to \infty $
}
\end{array} \right\}.
$$
\end{fact}
\begin{rem}\label{Rem.Fact4}
    In the construction of the sequence $ \{ u_i \}_{i = 1}^\infty \subset H^1(\Omega) $ as in \eqref{reg01}, the key point is to prepare an auxiliary sequence $ \{ u_i^\perp \}_{i = 1}^\infty \cap W^{1, 1}(\Omega) \cap L^2(\Omega) $ such that:
\begin{equation*}
\left\{ \hspace{-4ex} \parbox{12.75cm}{
\vspace{-1ex}
\begin{itemize}
\item[]$ u_i^\perp = \gamma $, a.e. on $ \Gamma $, for any $ i \in \N $,
\item[]$ u_i^\perp \to 0 $ in $ L^2(\Omega) $ and $ \ds \int_\Omega \beta |\nabla u_i^\perp| \, dx \to \int_\Gamma \beta|u| \, d \Gamma $, as $ i \to \infty $,
\vspace{-1ex}
\end{itemize}
} \right.
\end{equation*}
    and such a sequence is easily obtained by referring to some appropriate general theories, e.g. \cite{MR0102739,MR2139257}. In addition, if we suppose $ \gamma \in H^1(\Gamma) $, then we can take the auxiliary sequence to satisfy $ \{ u_i^\perp \}_{i = 1}^\infty \subset H^1(\Omega) $ (cf. \cite[Lemma 2]{MR3751650}), and thereby, we can also suppose $ \{ u_i \}_{i = 1}^\infty \subset H^1(\Omega) $ for the sequence as in \eqref{reg01}.
\end{rem}

\begin{fact}\label{Fact_Gm-conv}
    (cf. \cite[Remark 12 (Fact 6)]{MR4352617}) Let $ \beta \in W_c(\Omega) $, $ \{ \beta_n \}_{n = 1}^\infty \subset W_0(\Omega) $, $ \gamma \in H^{\frac{1}{2}}(\Gamma) $, and $ \{ \gamma_n \}_{n = 1}^\infty \subset H^{\frac{1}{2}}(\Gamma) $ be such that:
        \begin{equation}\label{beta^circ}
            \begin{cases}
                \beta_n \to \beta \mbox{ in $ L^2(\Omega) $, weakly in $ H^1(\Omega) $,} 
                \\
                \qquad \mbox{and weakly-$*$ in $ L^\infty(\Omega) $,}
                \\[1ex]
                \gamma_n \to \gamma \mbox{ in $ H^{\frac{1}{2}}(\Gamma) $,}
            \end{cases}
             \mbox{as $ n \to \infty $.}
        \end{equation}
Then the sequence of convex functions $ \{ \mathrm{Var}_{\gamma_n}^{\beta_n}({}\cdot{}) \}_{n = 1}^\infty $ $ \Gamma $-converges to the convex function $ \mathrm{Var}_\gamma^{\beta}({}\cdot{}) $ on $ L^2(\Omega) $, as $ n \to \infty $.
\end{fact}

\section{Statements of Main Theorems}

We start with prescribing the assumptions in this study.
\begin{description}
\item[(A1)]
$ \bm{g} = [g, g_\Gamma] : \R \longrightarrow \R^2 $ is a Lipschitz continuous function, such that:
  \begin{equation*}
    \left\{\begin{aligned}
& g(0)\le 0, \mbox{ and } g(1) \ge 0,
\\
& g_\Gamma(0)\le 0, \mbox{ and } g_\Gamma(1) \ge 0.
\end{aligned}
\right.
\end{equation*}
\KS{Additionally, there exist nonnegative primitives $\hat{g} = [\hat{g}, \hat{g}_\Gamma] : \R \longrightarrow [0,\infty)$ and $\hat{g}_\Gamma:\R \longrightarrow [0,\infty)$ of $ g $ and $ g_\Gamma $, respectively.}

    \item[(A2)]$ A_0 : \R^2 \longrightarrow \R^{2 \times 2} $ is a locally Lipschitz function, such that for any $ \bm{\eta} = [\tilde{\eta}, \tilde{\eta}_\Gamma] \in \R^2 $, the value $ A_0(\bm{\tilde{\eta}}) = A_0(\tilde{\eta}, \tilde{\eta}_\Gamma) \in \R^{2 \times 2} $ is defined by a diagonal matrix:
        \begin{align}
            & A_0(\bm{\tilde{\eta}}) = \left[ \begin{array}{cc} 
                \alpha_{0}(\bm{\tilde{\eta}}) & 0
                \\
                0 & \alpha_{\Gamma, 0}(\bm{\tilde{\eta}})
            \end{array} \right]  = \left[ \begin{array}{cc} 
                \alpha_{0}({\tilde{\eta}}, {\tilde{\eta}_\Gamma}) & 0
                \\
                0 & \alpha_{\Gamma, 0}({\tilde{\eta}}, {\tilde{\eta}_\Gamma})
            \end{array} \right]  \in \R^{2 \times 2},
        \end{align}
        with use of fixed functions $\alpha_0 : \R^2 \longrightarrow (0, \infty) $, and $ \alpha_{\Gamma, 0} : \R^2 \longrightarrow (0, \infty) $.

\item[(A3)]$\alpha :\R \longrightarrow (0,\infty)$ is a $C^2$-function such that $\alpha'(0)=0$,  and $\alpha''\ge0$ on $\R$. 

\item[(A4)]$ \delta_\alpha := \inf \bigl( \alpha(\R) \cup \alpha_0(\R) \cup \alpha_{\Gamma, 0}(\R) \bigr)  > 0 $
\item[(A5)]
    $ \bm{\eta}_0 = [\eta_0,\eta_{\Gamma,0}] \in V_\varepsilon \cap \Lambda_0^1 $, and $ \bm{\theta}_0 = [\theta_0,\theta_{\Gamma,0}] \in  W \cap \Lambda_{\KS{r_0}}^{\KS{r_1}} $ for some fixed constants $ -\infty < \KS{r_0} \leq \KS{r_1} < \infty $.
\end{description}
\begin{rem}\label{KS-Rem.freeEnergy}
For any \KS{$ \beta \in W_0(\Omega)$,} we define a functional $ \Phi_0(\beta;{}\cdot{}) : H \longrightarrow [0, \infty] $, by letting:
\KS{
\begin{align}
    \Phi_0 & : \bm{\theta} = [\theta, \theta_\Gamma] \in H \mapsto  \Phi_0(\beta; \bm{\theta}) = \Phi_0(\beta; \theta, \theta_\Gamma)  := \mathrm{Var}_{\theta_\Gamma}^{\beta}(\theta) +\Psi_\Gamma(\kappa_\Gamma \theta_\Gamma)
\\[1ex]
    &= \left\{\begin{array}{ll}
        \multicolumn{2}{l}{\ds \int_{\overline{\Omega}} d \bigl[ \beta |D \theta| \bigr]_{\theta_\Gamma} +\frac{\kappa_\Gamma^2}{2} \int_\Gamma |\nabla_\Gamma \theta_\Gamma|^2 \,d\Gamma, \mbox{if $ [\theta,\theta_\Gamma] \in W $,}}
\\[2ex]
\infty, & \mbox{ otherwise.}
\end{array}\right.
\label{Phi_ep}
\end{align}
}
Then, on account of \eqref{Phi_gm(bt.)}, (A1)--(A5), and Fact \ref{Fact_Gm-conv}, it is observed that $ \Phi(\beta; \cdot) $, for each $ \beta \in W_0(\Omega) $, is proper l.s.c. and convex on $ H $. Hence, we can provide the rigorous definition of free-energy $ \mathscr{F}_\varepsilon $, as in \eqref{Free-Ener}, as follows:
\KS{
\begin{align*}
        \mathscr{F}_{\varepsilon} & : [\bm{\eta}, \bm{\theta}] = [\eta, \eta_\Gamma, \theta, \theta_\Gamma] \in [H]^2 \mapsto \mathscr{F}_{\varepsilon}(\eta, \eta_\Gamma, \theta,\theta_\Gamma) 
        \\[0.5ex]
        & := \Psi_{\Omega, \Gamma}(\kappa \eta, \varepsilon \eta_\Gamma) +G(\bm{\eta}) +\Phi_0(\alpha(\eta); \bm{\theta}) \in [0, \infty],
\end{align*}
where 
\begin{align*}
    G: \bm{\eta} = [\eta, \eta_\Gamma] \in H ~& \mapsto  G(\bm{\eta}) = G(\eta, \eta_\Gamma) := \int_\Omega \hat{g}(\eta)\,dx + \int_\Gamma \hat{g}_\Gamma(\eta_\Gamma)\,dx \in [0, \infty].
\end{align*}
}
\end{rem}
\vspace{-2ex}

On this basis of assumptions, we define the solution to (KWC)$_{\varepsilon,\delta}$, for any $\varepsilon,\delta\ge0$.
\begin{defn}[Definition of solution]\label{Def.sol} Let us fix any  $ \varepsilon \geq 0 $. Then, a quartet of functions $ [\bm{\eta}, \bm{\theta}] = [\eta, \eta_\Gamma, \theta, \theta_\Gamma] \in L^2(0, T; [H]^2)$, with $ \bm{\eta} = [\eta,\eta_\Gamma] \in L^2(0, T; H)$ and $ \bm{\theta} = [\theta, \theta_\Gamma] \in L^2(0,T;H)$, is called solution to (KWC)$_{\varepsilon}$, iff. the following items hold.
\begin{description}
    \item[(S0)] $\bm{\eta}=[\eta,\eta_\Gamma] \in W^{1, 2}(0, T; H) \cap L^\infty(0, T; V_\varepsilon)$, $\bm{\theta} = [\theta,\theta_\Gamma] \in W^{1, 2}(0, T; H)$, $ |D\theta(\cdot)|(\Omega) \in L^\infty(0, T) $, and $ \theta_\Gamma \in L^\infty(0, T; H^1(\Gamma)) $, and
\begin{align}\label{S1}
    & \begin{cases}
        [\bm{\eta}(t), \bm{\theta}(t)] \in \Lambda_0^1 \times \Lambda_{\KS{r_0}}^{\KS{r_1}}, ~ \mbox{for a.e. $ t \in (0, T) $.}
        \\[1ex]
        [\bm{\eta}(0), \bm{\theta}(0)] = [\bm{\eta}_0, \bm{\theta}_0] \mbox{ in $ [H]^2 $. }
    \end{cases}
\end{align}
        \item[(S1)]$\bm{\eta} = [\eta,\eta_\Gamma]$ solves the following variational equality:
\begin{align}
    & \bigl( \partial_t \bm{\eta}(t) + \bm{g}(\bm{\eta}(t)), \bm{\varphi} \bigr)_H +\kappa^2 \bigl( \nabla \eta(t), \nabla \varphi \bigr)_{[L^2(\Omega)]^N} +\bigl( \KS{\nabla_\Gamma (\varepsilon\eta_\Gamma)(t),} \nabla_\Gamma(\varepsilon \varphi_\Gamma) \bigr)_{[L^2(\Gamma)]^N}
    \\[1ex]
    & \hspace{18ex}+\int_{\overline{\Omega}} d\bigl[ \varphi \alpha'(\eta(t)) |D \theta(t)| \bigr]_{\theta_\Gamma(t)} = 0,
    \\[1ex]
    & \hspace{5ex} \mbox{for any $ \bm{\varphi} = [\varphi,\varphi_\Gamma] \in V_\varepsilon \cap \KS{L^\infty(\Omega, \Gamma)} $, a.e. $ t \in (0,T)$.}
\end{align}
        \item[(S2)]$ \bm{\theta} = [\theta,\theta_\Gamma] $ solves the following variational inequality:
\begin{align}
    & \bigl( A_0(\bm{\eta}(t)) \partial_t \bm{\theta}(t), \bm{\psi} \bigr)_H +\kappa_\Gamma^2 \big( \nabla_\Gamma \theta_\Gamma(t), \nabla_\Gamma (\theta_\Gamma(t) - \psi_\Gamma) \bigr)_{[L^2(\Gamma)]^N} 
    \\[1ex]
    & \qquad + \int_{\overline{\Omega}} d \bigl[ \alpha(\eta(t))|D \theta(t)| \bigr]_{\theta_\Gamma(t)} \le \int_{\overline{\Omega}} d \bigl[ \alpha(\eta(t))|D \psi| \bigr]_{\psi_\Gamma}, 
    \\[1ex]
    & \hspace{9ex}\mbox{ for any $\bm{\psi} = [\psi,\psi_\Gamma] \in W $, a.e. $t \in (0,T)$.}
\end{align}
\end{description}
\end{defn}

Based on these, our Main Theorems are stated as follows.
    \begin{mTh}[Existence of solution with energy-disspation]\label{mTh.Sol} Under (A0)--(A5), the system (KWC)$_{\varepsilon}$ admits at least one solution $[\bm{\eta}, \bm{\theta}] = [\eta,\eta_\Gamma,\theta,\theta_\Gamma] \in L^2(0,T;[H]^2)$, such that:
        \begin{description}
            \item[(A)]the function $ t \in [0, T] \longrightarrow \mathscr{F}_\varepsilon(\bm{\eta}(t), \bm{\theta}(t)) = \mathscr{F}_\varepsilon(\eta(t), \eta_\Gamma(t), \theta(t), \theta_\Gamma(t)) \in [0, \infty) $ is right continuous,
            \item[(B)]$ [\bm{\eta}, \bm{\theta}] =  [\eta, \eta_\Gamma, \theta, \theta_\Gamma] $ reproduces the following \emph{energy-dissipation in time}
    \begin{align}
        & \int_s^t \bigl| \partial_t \bm{\eta}(t) \bigr|_H^2 \, d \varsigma +\int_s^t \bigl| A_0(\bm{\eta}(\varsigma))^{\frac{1}{2}} \partial_t \bm{\theta}(\varsigma) \bigr|_H^2 \, d \varsigma +\mathscr{F}_\varepsilon(\bm{\eta}(t),  \bm{\theta}(t)) \leq \mathscr{F}_\varepsilon(\bm{\eta}(s),  \bm{\theta}(s))
        \\[1ex]
        & \hspace{30ex}\mbox{for all $ 0 \leq s \leq t \leq T $.}
    \end{align}

        \end{description}
\end{mTh}
\section{Key-Lemmas}
In this Section, we prove several Key-Lemmas that are vital for our Main Theorems.
    We begin by setting a class of relaxed convex functions. For every $\delta>0$, and $ 0 \leq \beta \in L^2(\Omega)  $, let us define:
%\begin{subequations}\label{rx-Phi}
\begin{align}\label{DefOfPhi_delta}
    \Phi_{\delta} & : \KS{\bm{\theta}} = [\theta,\theta_\Gamma] \in H \mapsto \Phi_{\delta}(\beta; \KS{\bm{\theta}}) = \Phi_{\delta}(\beta; \theta,\theta_\Gamma)
\\[2ex]
    &:= \left\{\begin{array}{cl}
    \multicolumn{2}{l}{\ds  \int_\Omega \left(\beta f_\delta(\nabla\theta)+\frac{\delta^2}{2}|\nabla\theta|^2\right)\,dx +\frac{\kappa_\Gamma^2}{2}\int_\Gamma |\nabla_\Gamma \theta_\Gamma|^2\,d\Gamma,}
    \\[2.5ex]
        & \mbox{ if $ [\theta, \theta_\Gamma] \in V_\delta$,}
\\[1.5ex]
\infty, & \mbox{ otherwise,}
\end{array}\right.
\end{align}
with the use of the following relaxation of Euclidean norm:
\begin{equation}\label{f_d}
f_\delta: \omega \in \R^N \mapsto \gamma_\delta(\omega):=\sqrt{\delta^2 + |\omega|^2} -\delta \in [0,\infty).
\end{equation}
%\end{subequations}
\begin{rem}\label{Rem.rx-subDiff}
    For every $ \delta > 0 $ and $ 0 \leq \beta \in L^2(\Omega) $, it is easily checked that the functional $ \Phi_\delta(\beta;{}\cdot{}) $ is proper l.s.c. and convex on $ H $. Additionally, we note that $ \Phi_\delta(\beta;{}\cdot{}) $ is a modified version of the convex energy, dealt with \cite{MR3661429}. So, referring to \cite[Key-Lemma 1]{MR3661429}, we can also verify that the subdifferential $ \partial \Phi_{\delta}(\beta;{}\cdot{})$ coincides with a single-valued operator $ \mathcal{A}_\delta \subset H \times H $, defined as follows:
\begin{align}
    \mathcal{A}_\delta &: [\theta, \theta_\Gamma] \in D(\mathcal{A}_\delta) \subset H \mapsto \mathcal{A}_\delta [\theta, \theta_\Gamma]
    \nonumber
    \\[1ex]
    &:=\rule{0pt}{20pt}^\mathrm{t} \hspace{-0.5ex} \left[\begin{array}{c}
    -\mathrm{div}\bigl(\beta \nabla f_\delta(\nabla \theta)+\delta^2\nabla\theta\bigr)
    \\[1ex]
        -\kappa_\Gamma^2 \varDelta_\Gamma \theta_\Gamma + \bigl[ (\beta \nabla f_\delta(\nabla \theta) + \delta^2 \nabla\theta)\cdot n_\Gamma\bigr] 
    \end{array}\right]\in H,
    \label{op.T_e^d}
\end{align}
with the domain:
\begin{equation}\label{eff.T_e^d}
D(\mathcal{A}_\delta):=\left\{\begin{array}{l|c}[\theta,\theta_\Gamma] \in V_\delta& \parbox{8.8cm}{
        ~ \hfill $\alpha(\eta)\nabla f_\delta(\nabla \theta)+ \delta^2\nabla \theta \in \KS{L_{\mathrm{div}}^2(\Omega; \R^N)}$, \hfill ~
\\[1.5ex]
$-\kappa_\Gamma^2 \varDelta_\Gamma \theta_\Gamma + \bigl[(\beta\nabla f_\delta(\nabla \theta)+\delta^2\nabla\theta)\cdot n_\Gamma \bigr] \in L^2(\Gamma)$
}\end{array}\right\}.
\end{equation}
\end{rem}
\medskip

On this basis, we now prove the following key-lemmas and corollaries.
\begin{keyLem}%[Representation of $\partial \Phi_\varepsilon^\delta$]
    \label{KeyLem01}
    Let us fix any $ \beta \in \KS{Y_\mathrm{c}}(\Omega) $. Then, for pairs of functions $ \KS{\bm{\theta} = }[\theta, \theta_\Gamma] \in H $ and $ \KS{\bm{\theta}^* = } [\theta^*, \theta_\Gamma^*] \in H $, the following two items are equivalent.
    \begin{description}
        \item[(C)]\KS{$ [\bm{\theta}, \bm{\theta}^*] \in \partial \Phi_0(\beta; {}\cdot{}) $ in $ H \times H $, i.e. $ \bm{\theta}^* \in \partial \Phi_0(\beta; \bm{\theta}) $ in $ H $, with $ \bm{\theta} \in D(\partial \Phi_0(\beta;{}\cdot{})) $.}
        \item[(D)]$ \KS{\bm{\theta} = } [\theta, \theta_\Gamma] \in W $, $ \theta_\Gamma \in H^2(\Gamma) $, and there exists a vector field $ \omega^* \in \KS{L_\mathrm{div}^2(\Omega; \R^N)} \cap  L^\infty(\Omega; \R^N) $, such that:
        \item[~~~(b0)]$ \KS{^\mathrm{t} \hspace{-0.1ex} \bm{\theta}^* = } \left[ \begin{array}{c}
                \theta^* \\[0.5ex] \theta_\Gamma^*
        \end{array} \right] = \left[ \begin{array}{c} 
            -\mathrm{div} (\beta \omega^*) \\[0.5ex] -\kappa_\Gamma^2 \varDelta_\Gamma \theta_\Gamma +\beta\trace{\Gamma} [\omega^* \cdot n_\Gamma]
        \end{array} \right] $ in $ H $.
        \item[~~~(b1)]$ \frac{(\omega^*, D\theta)}{|D\theta|} = 1 $, $ |D \theta| $-a.e. in $ \Omega $, and in particular, $ \omega^* \in \mathrm{Sgn}^N(\nabla \theta) $ a.e. in $ \Omega $;
        \item[~~~(b2)]$ [\omega^* \cdot n_\Gamma] \in \mathrm{Sgn}^1(\theta_\Gamma -\theta\trace{\Gamma}) $ a.e. on $ \Gamma $, i.e. $ \begin{cases}
                \theta\trace{\Gamma} = \theta_\Gamma, \mbox{ if $ |[\omega^* \cdot n_\Gamma]| < 1 $,}
                \\[0.5ex]
                \theta\trace{\Gamma} \leq \theta_\Gamma, \mbox{ if $ [\omega^* \cdot n_\Gamma] = 1 $,}
                \\[0.5ex]
                \theta\trace{\Gamma} \geq \theta_\Gamma, \mbox{ if $ [\omega^* \cdot n_\Gamma] = -1 $.}
        \end{cases} $
    \end{description}
\end{keyLem}
\begin{proof}
    For the proof of this key-lemma, we define a set-valued operator $ \mathcal{A} \subset H \times H $ by letting:
    \begin{align}\label{defOfOpA}
        \mathcal{A} ~& : \KS{\bm{\theta} = }[\theta, \theta_\Gamma] \in H \mapsto \KS{\mathcal{A}\bm{\theta}} =  \mathcal{A}[\theta, \theta_\Gamma] 
        \nonumber 
        \\
        & := \left\{ \begin{array}{l|l}
            \KS{\bm{\theta}^* =} [\theta^*, \theta_\Gamma^*] \in H & \parbox{5.5cm}{
                $ [\theta^*, \theta_\Gamma^*] $ fulfills (b0) for some vector field $ \omega^* \in \KS{L_\mathrm{div}^2(\Omega; \R^N)} \cap L^\infty(\Omega; \R^N) $ as in (b1) and (b2)
            }
        \end{array} \right\} \subset H.
    \end{align}
    Then, the equivalence of (C) and (D) is rephrased as the following coincidence of operators:
    \begin{align}
        & \mathcal{A} = \partial \Phi(\beta;{}\cdot{}) \mbox{ in $ H \times H $.}
    \end{align}
    Additionally, this coincidence will obtained as a consequence of the following two claims:
    \begin{description}
        \item[{\boldmath$ \sharp 1) $}]$ \mathcal{A} \subset \partial \Phi(\beta;{}\cdot{}) $ in $ H \times H $, so that $ \mathcal{A} $ is a monotone graph in $ H \times H $;
            \vspace{-1ex}
        \item[{\boldmath$ \sharp 2) $}]$ \mathcal{A} $ is maximal in the class of monotone graphs in $ H \times H $.
    \end{description}

    Now, we set the goal to verify the above claims. 
    \bigskip
    
    \noindent
    \textbf{\boldmath Verification of $ \sharp 1) $. } Let us assume that
        $ \bigl[ \bm{\theta}, \bm{\theta}^*] \in \mathcal{A} $ in $ H \times H $. 
    Then, for any \KS{$ \bm{\tilde{\theta}} =  [\tilde{\theta}, \tilde{\theta}_\Gamma] \in W $,} we can see from \eqref{defOfOpA} and Fact \label{Fact1} and Remark \ref{Rem.ex_hm} in Appendix that:
    \begin{align}
        & \KS{\bigl( \bm{\theta^*}, \bm{\tilde{\theta}} -\bm{\theta} \bigr)_H} = (\theta^*, \tilde{\theta} -\theta)_{L^2(\Omega)} +(\theta_\Gamma^*, \tilde{\theta}_\Gamma -\theta_\Gamma)_{L^2(\Gamma)}
        \\
        & = \int_\Omega -\mathrm{div}\,\bigl( \beta \omega^* \bigr)(\tilde{\theta} -\theta) \, dx +\int_\Gamma \bigl( -\kappa_\Gamma^2 \varDelta_\Gamma \theta_\Gamma + \beta\trace{\Gamma} [\omega^* \cdot n_\Gamma] \bigr) (\tilde{\theta}_\Gamma -\theta_\Gamma) \, d \Gamma
        \\
        & = \int_\Omega [\beta]^* d \bigl( \omega^*, D(\tilde{\theta} -\theta) \bigr) -\int_\Gamma \bigl[ (\beta \omega^*) \cdot n_\Gamma \bigr] (\tilde{\theta}\trace{\Gamma} -\theta\trace{\Gamma}) \, d \Gamma
        \\
        & \qquad -\kappa_\Gamma^2 \int_\Gamma \varDelta_\Gamma \theta_\Gamma (\tilde{\theta}_\Gamma -\theta_\Gamma) \, d \Gamma
        +\int_\Gamma \beta\trace{\Gamma} [\omega^* \cdot n_\Gamma] (\tilde{\theta}_\Gamma -\theta_\Gamma) \, d \Gamma
        \\
        & = \int_\Omega [\beta]^* d (\omega^*, D \tilde{\theta}) -\int_\Omega [\beta]^* d|D \theta| +\kappa_\Gamma^2 \int_\Gamma \nabla_\Gamma \theta_\Gamma \cdot \nabla_\Gamma (\tilde{\theta}_\Gamma -\theta_\Gamma) \, d \Gamma
        \\
        & \qquad +\int_\Gamma \beta\trace{\Gamma} [\omega^* \cdot n_\Gamma] \bigl( (\tilde{\theta}_\Gamma -\tilde{\theta}\trace{\Gamma}) -(\theta_\Gamma -\theta\trace{\Gamma}) \bigr) \, d \Gamma
        \\
        & \leq \Phi(\beta; \KS{\bm{\tilde{\theta}}}) -\Phi(\beta; \KS{\bm{\theta}}).
    \end{align}
    Thus, claim $ \sharp 1) $ is verified. 
    \bigskip

    \noindent
    \textbf{\boldmath Verification of $\sharp 2)$. } 
    By Minty's theorem, it is sufficient to show $ (\mathcal{A} +\mathcal{I}_{H}) = H $. Since the inclusion $ (\mathcal{A} +\mathcal{I}_H)H \subset H $ is trivial, our task can be reduced to show only the converse inclusion. 

    Let us fix any $ \KS{\bm{h} = [h, h_\Gamma]} \in H $. Then, applying Minty's theorem and Remark \ref{Rem.rx-subDiff}, we can find a sequence of functional pairs $ \{ \KS{\bm{\theta}_\delta =} [\theta_\delta, \theta_{\Gamma, \delta}] \in V_\delta \}_{\delta > 0} $ such that:
    \begin{align}\label{baseEq00}
        & \KS{\bm{h} -\bm{\theta}_\delta = \mathcal{A}_\delta \bm{\theta}_\delta \mbox{ in $ H $, for all $ \delta > 0 $. }}
    \end{align}
    Here, with \KS{\eqref{op.T_e^d}} in Remark \ref{Rem.rx-subDiff} in mind, we multiply the both sides of \eqref{baseEq00} by $ \KS{\bm{\theta}_\delta =} [\theta_\delta, \theta_{\Gamma, \delta}] $. Then, by using Young's inequality, it immediately follows that:
    \begin{align}\label{baseEq01}
        & \frac{1}{2} \bigl| \KS{\bm{\theta_\delta}} \bigr|_H^2 +\Phi_\delta(\beta; \KS{\bm{\theta}_\delta}) \leq  \frac{1}{2} \bigl| \KS{\bm{h}} \bigr|_H^2, \mbox{ for all $ \delta > 0 $.}
    \end{align}
    Subsequently, owing to \eqref{DefOfPhi_delta}, \eqref{baseEq01}, and the compactness theory of Rellich--Kondrachov type, we can find an approximating limit $ \bm{\theta} = [\theta, \theta_\Gamma] \in W $, together with a sequence:
    \begin{align}\label{seqOfdelta}
        & 1 > \delta_1 > \dots > \delta_n \downarrow 0 \mbox{ as $ n \to \infty $,}
    \end{align}
    and a sequence of functional pairs $ \{ \bm{\theta}_n = [\theta_n, \theta_{\Gamma, n}] \}_{n = 1}^\infty := \{ \bm{\theta}_{\delta_n} %= [\theta_{\delta_n}, \theta_{\Gamma, \delta_n}] 
    \}_{n = 1}^\infty $, and we can see that:
    \begin{align}\label{rep01}
        & \begin{cases}
            \bm{\theta}_n = [\theta_n, \theta_{\Gamma, n}] \to \bm{\theta} = [\theta, \theta_\Gamma] ~\mbox{in $ L^1(\Omega) \times L^2(\Gamma) $, }
            \\
            \qquad \mbox{and weakly in $ L^2(\Omega) \times H^1(\Gamma) $,}
            \\[1ex]
            \delta_n \theta_n \to 0 \mbox{ in $ L^2(\Omega) $, and weakly in $ H^1(\Omega) $,}
        \end{cases}
        \mbox{as $ n \to \infty $,}
    \end{align}
    and
    \begin{align}
        & \int_\Omega \beta \nabla f_\delta (\nabla \theta_n) \cdot \nabla \varphi \, dx +\int_\Omega \nabla (\delta_n \theta_n) \cdot \nabla (\delta_n \varphi) \, dx +\kappa_\Gamma^2 \int_\Gamma \nabla_\Gamma \theta_{\Gamma, n} \cdot \nabla_\Gamma \varphi_\Gamma \, d \Gamma
        \\
        & \hspace{16ex} = \int_\Omega (h -\theta_n) \varphi \, dx +\int_\Gamma (h_\Gamma -\theta_{\Gamma, n}) \varphi_\Gamma \, d \Gamma, 
        \label{rep02}
        \\
        & \hspace{14ex} \mbox{for all $ [\varphi, \varphi_\Gamma] \in V_{\delta_n} (=V_{1}) $, and $ n = 1, 2, 3, \dots $.}
    \end{align}
    Additionally, since
    \begin{align}\label{rep03}
        & |\nabla f_{\delta_n}(\nabla \theta_n)| = \left| \frac{\nabla \theta_n}{\sqrt{\delta_n^2 +|\nabla \theta_n|^2}} \right| \leq 1 \mbox{ a.e. in $ \Omega $, for $ n = 1, 2, 3, \dots $,}
    \end{align}
    we can suppose to have a vector field $ \omega^* \in L^\infty(\Omega; \R^N) $, such that:
    \begin{align}\label{rep04}
        & 
            |\omega^*| \leq 1 \mbox{ a.e. in $ \Omega $, and }
            \nabla f_{\delta_n}(\nabla \theta_n) \to \omega^* ~\mbox{weakly-$*$ in $ L^\infty(\Omega; \R^N) $ as $ n \to \infty $,}
    \end{align}
    by taking a subsequence if necessary. 

    Now, applying the convergences in \eqref{seqOfdelta}, \eqref{rep01}, and \eqref{rep04} to the variational form \eqref{rep02}, we can see that:
    \begin{align}
        & \int_\Omega \beta \omega^* \cdot \nabla \varphi \, dx +\kappa_\Gamma^2 \int_\Gamma \nabla_\Gamma \theta_{\Gamma} \cdot \nabla_\Gamma \varphi_\Gamma \, d \Gamma 
        \\
        & \hspace{-10ex} = \int_\Omega (h -\theta) \varphi \, dx +\int_\Gamma (h_\Gamma -\theta_\Gamma) \varphi_\Gamma \, d \Gamma, \mbox{ for any $ [\varphi, \varphi_\Gamma] \in V_1 $.}
        \label{rep05}
    \end{align}
    In particular, taking any $ \varphi_0 \in H_0^1(\Omega) $ and putting $ [\varphi, \varphi_\Gamma] = [\varphi_0, 0] \in V_1 $ in \eqref{rep05}, we have:
    \begin{align}
        & \int_\Omega \beta \omega^* \cdot \nabla \varphi_0 \, dx = \int_\Omega (h -\theta) \varphi_0 \, dx, \mbox{ for any $ \varphi_0 \in H_0^1(\Omega) $,}
        \\
        & \hspace{7ex} \mbox{i.e. }~ -\mathrm{div} \, (\beta \omega^*) = h -\theta \in L^2(\Omega) \mbox{ in $ H^{-1}(\Omega) $.}
        \label{rep06}
    \end{align}
    Subsequently, taking any $ \psi_\Gamma \in H^{1}(\Gamma) $ with its harmonic extension $ [\psi_\Gamma]^\mathrm{hm} \in H^1(\Omega) $ as in Remark \ref{Rem.ex_hm} in Appendix, one can deduce from \eqref{rep05}, \eqref{rep06}, and Remark \ref{Prop.pairMeas} in Appendix that:
    \begin{align}
        & \bigl< -\kappa_\Gamma^2 \varDelta_\Gamma \theta_\Gamma +[(\beta \omega^*) \cdot n_\Gamma], \psi_\Gamma \bigr>_{H^{1}(\Gamma)} = \kappa_\Gamma^2 \int_\Gamma \nabla_\Gamma \theta_\Gamma \cdot \nabla_\Gamma \psi_\Gamma \, d \Gamma +\int_\Gamma \beta\trace{\Gamma} [\omega^* \cdot n_\Gamma] \psi_\Gamma \, d \Gamma 
        \\
        & 
        = \kappa_\Gamma^2 \int_\Gamma \nabla_\Gamma \theta_\Gamma \cdot \nabla_\Gamma \psi_\Gamma \, d \Gamma -\int_\Omega (h -\theta) [\psi_\Gamma]^\mathrm{hm} \, dx +\int_\Omega \beta \omega^* \cdot \nabla [\psi_\Gamma]^\mathrm{hm} \, dx = \int_\Gamma (h_\Gamma -\theta_\Gamma) \psi_\Gamma \, d \Gamma.
    \end{align}
    This implies that:
    \begin{align}
        & -\kappa_\Gamma^2 \varDelta_\Gamma \theta_\Gamma  = (h_\Gamma -\theta_\Gamma) -\beta\trace{\Gamma}[\omega^* \cdot n_\Gamma] \in L^2(\Gamma) ~\mbox{ in $ H^{-1}(\Gamma $),}
        \label{rep10}
        \\
        & \hspace{10ex} \mbox{so that $ \theta_\Gamma = ( -\kappa_\Gamma^2 \varDelta_\Gamma +\mathcal{I}_{L^2(\Gamma)})^{-1}\bigl( h_\Gamma -\beta\trace{\Gamma} [\omega^* \cdot n_\Gamma] \bigr) \in H^2(\Gamma) $.} 
    \end{align}

    Next, we take any $ [\tilde{\varphi}, \tilde{\varphi}_\Gamma] \in V_1 $, and put $ [\varphi, \varphi_\Gamma] = [\theta_n -\tilde{\varphi}, \theta_{\Gamma, n} -\tilde{\varphi}_\Gamma] \in V_1 $ in \eqref{rep02} to obtain that:
    \begin{align}
        \int_{\overline{\Omega}} d & [\beta |D \theta_n|]_{\theta_{\Gamma, n}} +\frac{\delta_n^2}{2} \int_\Omega |\nabla \theta_n|^2 \, dx +\frac{\kappa_\Gamma^2}{2} \int_\Gamma |\nabla_\Gamma \theta_{\Gamma, n}|^2 \, d \Gamma
        \\
        & \qquad +\int_\Omega (\theta_n -h)(\theta_n -\tilde{\varphi}) \, dx +\int_\Gamma (\theta_{\Gamma, n} -h_\Gamma)(\theta_{\Gamma, n} -\tilde{\varphi}_\Gamma) \, d \Gamma
        \\
        & \leq \int_\Omega \beta \nabla f_{\delta_n}(\nabla \theta_n) \cdot \nabla \tilde{\varphi} \, dx +\kappa_\Gamma^2 \int_\Gamma \nabla_\Gamma \theta_{\Gamma, n} \cdot \nabla_\Gamma \tilde{\varphi}_\Gamma \, d \Gamma
        \label{rep07}
        \\
        & \qquad +\frac{\delta_n^2}{2} \int_\Omega |\nabla \tilde{\varphi}|^2 \, dx +\delta_n \mathcal{L}^N(\Omega), ~\mbox{ for $ n = 1, 2, 3, \dots $.}
    \end{align}
    Here, having in mind \eqref{Phi_ep}, \eqref{DefOfPhi_delta}, \eqref{seqOfdelta}, \eqref{rep01}, \eqref{rep04}, and Fact \ref{Fact_Gm-conv}, let us take the limit-inf of both sides of \eqref{rep07}. Then, one can compute that:
    \begin{align}
        \int_{\overline{\Omega}} d & [\beta |D \theta|]_{\theta_\Gamma} +\frac{\kappa_\Gamma^2}{2} \int_\Gamma |\nabla_\Gamma \theta_\Gamma|^2 \, d \Gamma +\int_\Omega (\theta -h)(\theta -\tilde{\varphi}) \, dx +\int_\Gamma (\theta_\Gamma -h_\Gamma)(\theta_\Gamma -\tilde{\varphi}_\Gamma) \, d \Gamma
        \\
        & \leq \varliminf_{n \to \infty} \int_{\overline{\Omega}} d [\beta |D \theta_n|]_{\theta_{\Gamma, n}} +\frac{1}{2} \lim_{n \to \infty} \int_\Omega |\nabla (\delta_n \theta_n)|^2 \, dx +\frac{\kappa_\Gamma^2}{2} \varliminf_{n \to \infty} \int_\Gamma |\nabla_\Gamma \theta_{\Gamma, n}|^2 \, d \Gamma
        \\
        & \qquad +\varliminf_{n \to \infty} \int_\Omega (\theta_n -h)(\theta_n -\tilde{\varphi}) \, dx  +\varliminf_{n \to \infty} \int_\Gamma (\theta_{\Gamma, n} -h_\Gamma)(\theta_{\Gamma, n} -\tilde{\varphi}_\Gamma) \, d \Gamma
        \\
        & \leq \varliminf_{n \to \infty} \left( \int_{\overline{\Omega}} d [\beta |D \theta_n|]_{\theta_{\Gamma, n}} +\frac{\kappa_\Gamma^2}{2} \int_\Gamma |\nabla_\Gamma \theta_{\Gamma, n}|^2 \, d \Gamma \right.
        \\
        & \qquad \left. +\int_\Omega (\theta_n -h)(\theta_n -\tilde{\varphi}) \, dx +\int_\Gamma (\theta_{\Gamma, n} -h_\Gamma)(\theta_{\Gamma, n} -\tilde{\varphi}_\Gamma) \, d \Gamma \right)
        \\
        & \leq \int_\Omega \beta \omega^* \cdot \nabla \tilde{\varphi} \, dx +\kappa_\Gamma^2 \int_\Gamma \nabla_\Gamma \theta_\Gamma \cdot \nabla_\Gamma \tilde{\varphi}_\Gamma \, d \Gamma.
    \end{align}
    Therefore, from \eqref{Phi_gm(bt.)} and Fact \ref{Fact3} in Appendix, it follows that:
    \begin{align}
        & \int_\Omega [\beta]^* d |D \theta| +\int_\Gamma \beta\trace{\Gamma} |\theta_\Gamma -\theta\trace{\Gamma}| \, d \Gamma
        \leq \int_\Omega \beta \omega^* \cdot \nabla \tilde{\varphi} \, dx +\int_\Omega (h -\theta)(\theta -\tilde{\varphi}) \, dx 
        \\
        & +\int_\Gamma (h_\Gamma -\theta_\Gamma)(\theta_\Gamma -\tilde{\varphi}_\Gamma) \, d \Gamma +\kappa_\Gamma^2 \int_\Gamma \varDelta_\Gamma \theta_\Gamma (\theta_\Gamma -\tilde{\varphi}_\Gamma) \, d \Gamma, \mbox{ for any $ [\tilde{\varphi}, \tilde{\varphi}_\Gamma] \in V_1 $.}
        \label{rep08}
    \end{align}
Additionally, by using \eqref{rep06}, \eqref{rep10}, \eqref{rep08}, Fact \ref{Fact10}, and Remarks \ref{Rem.pairMeas} and \ref{Prop.pairMeas} in Appendix, we can deduce that:
\begin{align}
    & \int_\Omega \beta |\nabla \theta| \, dx +\int_\Omega [\beta]^* d |D^s \theta| +\int_\Gamma \beta\trace{\Gamma} |\theta_\Gamma -\theta\trace{\Gamma}| \, d \Gamma
    \\
    & = \int_\Omega [\beta]^* d |D\theta| +\int_\Gamma \beta\trace{\Gamma} |\theta_\Gamma -\theta\trace{\Gamma}| \, d \Gamma 
    \\
    & \leq -\int_\Omega \mathrm{div} \, (\beta \omega^*) \tilde{\varphi} \, dx +\int_\Gamma [(\beta \omega^*) \cdot n_\Gamma] \tilde{\varphi}\trace{\Gamma} \, d \Gamma
    \\
    & \qquad +\int_\Omega -\mathrm{div} \, (\beta \omega^*) (\theta -\tilde{\varphi}) \, dx +\int_\Gamma \beta\trace{\Gamma}[\omega^* \cdot n_\Gamma] (\theta_\Gamma -\tilde{\varphi}_\Gamma) \, d \Gamma
    \\
    & = -\int_\Omega \mathrm{div} \, (\beta \omega^*) \theta \, dx +\int_\Gamma [(\beta \omega^*) \cdot n_\Gamma] \theta_\Gamma \, d \Gamma 
    \\
    & = \int_\Omega d \bigl( (\beta \omega^*), D \theta \bigr) -\int_\Gamma \bigl[ (\beta \omega^*) \cdot n_\Gamma \bigr] \theta\trace{\Gamma} \, d \Gamma +\int_\Gamma \beta\trace{\Gamma} [\omega^* \cdot n_\Gamma]\theta_\Gamma \, d \Gamma
    \\
    & = \int_\Omega \beta \omega^* \cdot \nabla \theta \, dx +\int_\Omega [\beta]^* \, {\ts \frac{(\omega^*, D\theta)}{|D \theta|}} \, d |D^s \theta| +\int_\Gamma \beta\trace{\Gamma} [\omega^* \cdot n_\Gamma](\theta_\Gamma -\theta\trace{\Gamma}) \, d \Gamma.
    \label{rep20}
\end{align}
Meanwhile, invoking \eqref{rep04}, $ \beta \in W_\mathrm{c}(\Omega) $, Fact \ref{Prop.pairMeas} and Remark \ref{Rem.pairMeas} in Appendix, it is easily checked that:
\begin{align}
    & \begin{cases}
        \bigl| \bigl( [\beta]^* {\ts \frac{(\omega^*, D \theta)}{|D \theta|}} \bigr)(x) \bigr| \leq [\beta]^*(x)|\omega^*|_{L^\infty(\Omega; \R)} = [\beta]^*(x), \mbox{ for $ |D \theta| $-a.e. $ x \in \Omega $,}
        \\
        \qquad \mbox{with $ |\omega^* \cdot \nabla \theta| \leq |\nabla \theta| $ a.e. in $ \Omega $,}
        \\[1ex]
        \bigl| \bigl( \beta\trace{\Gamma}[\omega^* \cdot n_\Gamma] \bigr)(x_\Gamma) \bigr| \leq \beta\trace{\Gamma}(x_\Gamma)|\omega^*|_{L^\infty(\Omega; \R^N)} \leq \beta\trace{\Gamma}(x_\Gamma), \mbox{ for a.e. $ x_\Gamma \in \Gamma $.}
    \end{cases}
    \label{rep21}
\end{align}
As a consequence of \eqref{rep20}, \eqref{rep21}, and  $ \beta \in W_\mathrm{c}(\Omega) $, one can observe that:
\begin{align}
    & \begin{cases}
        {\ts \frac{(\omega^*, D \theta)}{|D \theta|}} = 1, \mbox{ $ |D \theta| $-a.e. in $ \Omega $, and in particular,}
        \\
        \qquad \omega^* \cdot \nabla \theta = |\nabla \theta|, \mbox{ and } \omega^* \in \mathrm{Sgn}^N(\nabla \theta), \mbox{ a.e. in $ \Omega $,}
        \\[1ex]
        [\omega^* \cdot n_\Gamma](\theta_\Gamma -\theta\trace{\Gamma}) = |\theta_\Gamma -\theta\trace{\Gamma}|,
        \\
        \qquad  \mbox{i.e. ~} [\omega^* \cdot n_\Gamma] \in \mathrm{Sgn}^1(\theta_\Gamma -\theta\trace{\Gamma}), \mbox{ a.e. on $ \Gamma $.}
    \end{cases}
    \label{rep22}
\end{align}
Taking into account \eqref{defOfOpA}, \eqref{rep06}, \eqref{rep10}, and \eqref{rep22}, it is inferred that:
\begin{align}
    & [\bm{\theta}, \bm{h} -\bm{\theta}] \in \mathcal{A} \mbox{ in $ H \times H $,}
    \\
    \mbox{i.e. } ~(\mathcal{A} +& \mathcal{I}_H) \bm{\theta} \ni \bm{h} \mbox{ in $ H $ with $ \bm{\theta} \in D(\mathcal{A}) $.}
\end{align}
This implies $ H \subset (\mathcal{A} +\mathcal{I}_H)H ~(\subset H) $, and the first claim $\sharp1$ and Minty's theorem will lead to the maximality of the graph of $ \mathcal{A} \subset H \times H $. 
\medskip

Thus, we finish the proof of this key-lemma. 
\end{proof}

\begin{keyLem}\label{keyLem02}
    Let $ \delta_\Gamma \geq 0 $ be a fixed constant, and let $ I \subset (0, T) $ be a fixed open interval. Let $ \bm{\vartheta} = [\vartheta, \vartheta_\Gamma] \in L^2(I; L^2(\Omega) \times H^{\frac{1}{2}}(\Gamma)) $ be a functional pair, such that:
    \begin{align}\label{kenConv00}
        & [|D \vartheta|(\cdot)]_{\vartheta_\Gamma(\cdot)}(\overline{\Omega}) \in L^1(I), \mbox{ and } \delta_\Gamma \vartheta_\Gamma \in L^2(I; H^1(\Gamma)).
    \end{align}
    Then, there exists a sequence of functional pairs $ \bigl\{ \bm{\widetilde{\vartheta}}_\ell =  [\widetilde{\vartheta}_\ell, \widetilde{\vartheta}_{\Gamma, \ell}] \bigr\}_{\ell = 1}^\infty \subset L^\infty(I; V_{\delta_\Gamma}) $, such that:
    \begin{align}
        \left\{ \rule{0pt}{48pt} \right. &
        \\[-88pt]
        & \bm{\widetilde{\vartheta}}_\ell = [\widetilde{\vartheta}_\ell, \widetilde{\vartheta}_{\Gamma, \ell}] \to \bm{\vartheta} = [\vartheta, \vartheta_\Gamma] \mbox{ in $ L^2(I; L^2(\Omega) \times H^{\frac{1}{2}}(\Gamma)) $,}
        \label{kenConv01a}
        \\[1ex]
        & \delta_\Gamma \widetilde{\vartheta}_{\Gamma, \ell} \to \delta_\Gamma \vartheta_\Gamma \mbox{ in $ L^2(I; H^{1}(\Gamma)) $,}
        \label{kenConv01b}
        \\[1ex]
        & \int_I \left| \int_\Omega |\nabla \widetilde{\vartheta}_\ell(t)| \, dx dt -\int_{\overline{\Omega}} d [|D\vartheta(t)|] _{\vartheta_\Gamma(t)} \right| dt \to 0,
        \label{kenConv01c}
    \end{align}
    and
    \begin{align}
        & \begin{cases}
            \bm{\widetilde{\vartheta}}_\ell(t) \to \bm{\vartheta}(t) \mbox{ in $ L^2(\Omega) \times H^{\frac{1}{2}}(\Gamma) $,}
            \\[1ex]
            \delta_\Gamma \widetilde{\vartheta}_{\Gamma, \ell}(t) \to \delta_\Gamma \vartheta_\Gamma(t) \mbox{ in $ H^{1}(\Gamma) $,}
            \\[1ex]
            \ds \int_\Omega |\nabla \widetilde{\vartheta}_{\ell}(t)| \, dx \to \int_{\overline{\Omega}} d [|D\vartheta(t)|]_{\vartheta_\Gamma(t)},
        \end{cases}
        \mbox{for a.e. $ t \in I $,}
        \label{kenConv01d}
    \end{align}
    as $ \ell \to \infty $.
\end{keyLem}
\begin{proof}
    For any $ m \in \N $, we take a $m$-equal division of the time-interval $ I $:
    \begin{align}
        & \inf I =: t_0^m < t_1^m < \dots < t_i^m < \dots < t_m^m := \sup I,
        \\
        & \hspace{-2ex} \mbox{with the partition points $ t_i^m := {\ts \frac{T}{m}} i $, for $ i = 1, \dots, m $.}
    \end{align}
    Besides, we put:
        \begin{align}
            & \bm{\overline{\vartheta}}_i^m = \bigl[ \, \overline{\vartheta}_i^{\,m}, \overline{\vartheta}_{\Gamma, i}^{\,m} \bigr] := \frac{m}{T} \int_{t_{i -1}^m}^{t_i^m} \bm{\vartheta}(\varsigma) \, d \varsigma \mbox{ in $ L^2(\Omega)  \times H^{\frac{1}{2}}(\Gamma) $, for $ i = 1, \dots, m $,}
            \label{kenConv02a}
        \end{align}
        and define a sequence of time-discrete approximation $ \bigl\{ \bm{\overline{\vartheta}}^{\,m} = [\,\overline{\vartheta}^{\,m}, \overline{\vartheta}_{\Gamma}^{\,m}] \bigr\}_{m = 1}^\infty \subset  L^\infty(I; $ $ L^2(\Omega) \times H^{\frac{1}{2}}(\Gamma)) $, by letting:
        \begin{align}
            \bm{\overline{\vartheta}}^{\,m} = [\,\overline{\vartheta}^{\,m}(t), & \overline{\vartheta}_{\Gamma}^{\,m}(t)] := \bm{\overline{\vartheta}}_i^{\,m} = \bigl[ \, \overline{\vartheta}_i^{\,m}, \overline{\vartheta}_{\Gamma, i}^{\,m} \bigr] \mbox{ in $ L^2(\Omega)  \times H^{\frac{1}{2}}(\Gamma) $,}
            \label{kenConv02b}
            \\
            & \mbox{if $ t \in [t_{i -1}^m, t_i^m) $, for $ i = 1, \dots, m $.}
        \end{align}

    Let us take any $ n \in \N $. Then, from \eqref{kenConv00}, \eqref{kenConv02a}, and \eqref{kenConv02b}, we can find a large number $ m_\ell \in \N $, such that
    \begin{align}
        & 2^\ell < m_{\ell} < m_{\ell +1}, \mbox{ for $ \ell = 1, 2, 3, \dots $,}
    \end{align}
    and the sequence $ \bigl\{ \bm{\overline{\vartheta}}_\ell = [ \, \overline{\vartheta}_\ell, \overline{\vartheta}_{\Gamma, \ell}] \bigr\}_{\ell = 1}^\infty := \bigl\{ \bm{\overline{\vartheta}}^{\,m_\ell} \bigr\}_{\ell = 1}^\infty $ satisfies that:
   \begin{align}
        & 
        \begin{cases}
            \bigl| \,\bm{\overline{\vartheta}}_\ell -\bm{\vartheta} \bigr|_{L^2(I; L^2(\Omega) \times H^{\frac{1}{2}}(\Gamma))} \leq 2^{-\ell},
            \\[1ex]
%        & 
            \bigl| \delta_\Gamma \overline{\vartheta}_{\Gamma, \ell} -\delta_\Gamma \vartheta_{\Gamma}|_{L^2(I; H^1(\Omega))} \leq 2^{-\ell},
        \end{cases}
        \mbox{for $ \ell = 1, 2, 3, \dots $.}
        \label{kenConv06}
    \end{align}
    Also, by taking a subsequence of $ \{m_\ell\}_{\ell = 1}^\infty \subset \N $ if necessary, we can suppose:
    \begin{align}
        [\,\overline{\vartheta}_\ell(t), \overline{\vartheta}_{\Gamma, \ell}( & t)] \to [\vartheta(t), \vartheta_\Gamma(t)] \mbox{ in $ L^2(\Omega) \times H^{\frac{1}{2}}(\Gamma) $,} 
        \label{kenConv05}
        \\
        & \mbox{for a.e. $ t \in I $, as $ \ell \to \infty $,}
    \end{align}
    so that, by the lower semi-continuity of the total variation, 
    \begin{align}
        & \varliminf_{\ell \to \infty} |D \vartheta_\ell(t)|(\Omega) \geq |D \vartheta(t)|(\Omega), \mbox{ for a.e. $ t \in I $.}
        \label{kenConv03}
    \end{align}
    Furthermore, for every $ n \in \N $ and $ t \in I %\setminus \bigcup_{m = 1}^\infty \{t_i^{m}\}_{i = 1}^m 
    $, we can find partition point $ s_t \in \{ t_i^{m_\ell} \}_{i = 0}^{m_\ell} $, such that:
    \begin{align}
        & |D \overline{\vartheta}_\ell(t)|(\Omega) = \sup \left\{ \begin{array}{l|l}
            \ds \int_\Omega \overline{\vartheta}_\ell(t) \, \mathrm{div} \, \varphi \, dx & \parbox{5.25cm}{
                $ \varphi \in C_\mathrm{c}(\Omega; \R^N) $, $ |\varphi| \leq 1 $ on $ \Omega $
            }
        \end{array} \right\}
        \\
        & \leq \frac{m_\ell}{T} \int_{s_t}^{s_t+\frac{T}{m_\ell}}\sup \left\{ \begin{array}{l|l}
            \ds \int_\Omega \vartheta(\varsigma) \, \mathrm{div} \, \varphi \, dx & \parbox{5.25cm}{
                $ \varphi \in C_\mathrm{c}(\Omega; \R^N) $, $ |\varphi| \leq 1 $ on $ \Omega $
            }
        \end{array} \right\} \, d \varsigma
        \\
        & = \frac{m_\ell}{T} \int_{s_t}^{s_t +\frac{T}{m_\ell}} |D \vartheta(\varsigma)|(\Omega)  \, d \varsigma, \mbox{ and $ s_t \leq t < s_t +\frac{T}{m_\ell} $.}
    \end{align}
    This implies that:
    \begin{align}
        & \varlimsup_{\ell \to \infty} |D \overline{\vartheta}_\ell(t)|(\Omega) \leq |D \vartheta(t)|(\Omega), 
        \mbox{ for a.e. $ t \in I $,}
        \label{kenConv04}
        \\
        \mbox{i.e. for } & \mbox{any Lebesgue point of the function $ |D\vartheta(\cdot)|(\Omega) \in L^1(I) $.}
    \end{align}
\noeqref{kenConv05, kenConv03} 
    Due to \eqref{kenConv06}--\eqref{kenConv04}, we can say that:
    \begin{align}
        & \begin{cases}
            \ds \int_I \left| |D \overline{\vartheta}_\ell(t)|(\Omega) -|D \vartheta(t)|(\Omega)  \right| \, dt \leq 2^{-\ell},
            \\[2ex]
            \ds \int_I \bigl|(\overline{\vartheta}_\ell)\trace{\Gamma}(t) -\vartheta\trace{\Gamma}(t) \bigr|_{L^1(\Gamma)} \, dt \leq 2^{-\ell},
        \end{cases}
        \mbox{for $ \ell = 1, 2, 3, \dots $,}
    \end{align}
    by taking more one subsequence if necessary, and hence, we can further compute that:
    \begin{align}
        & \int_I \bigl| [|D \overline{\vartheta}_\ell(t)|]_{\overline{\vartheta}_{\Gamma, \ell}(t)}(\overline{\Omega}) -[|D \vartheta(t)|]_{\vartheta_\Gamma(t)}(\overline{\Omega}) \bigr| \, dt
        \\
        \leq ~& \int_I \left| |D \overline{\vartheta}_\ell(t)|(\Omega) -|D \vartheta(t)|(\Omega)  \right| \, dt 
        \\
        & +\int_I \bigl|(\overline{\vartheta}_\ell)\trace{\Gamma}(t) -\vartheta\trace{\Gamma}(t) \bigr|_{L^1(\Gamma)} \, dt +\int_I \bigl|\overline{\vartheta}_{\Gamma, \ell}(t) -\vartheta_{\Gamma}(t) \bigr|_{L^1(\Gamma)} \, dt
        \label{kenConv10}
        \\
        \leq ~&  \bigl( 3 +T^{\frac{1}{2}} \mathcal{H}^{N -1}(\Gamma)^{\frac{1}{2}} \bigr) \cdot 2^{-\ell}.
    \end{align}

    Next, with the dense embedding $ H^1(\Gamma) \subset H^{\frac{1}{2}}(\Gamma) $ in mind, we take sequences $ \bigl\{ \widetilde{\vartheta}_{\Gamma, i}^{m_\ell} \bigr\}_{\ell = 1}^\infty \subset H^1(\Gamma) $, for $ i = 1, \dots, m_\ell $, such that:
    \begin{align}
        & \begin{cases}
            \ds \bigl| \widetilde{\vartheta}_{\Gamma, i}^{m_\ell} -\overline{\vartheta}_{\Gamma, i}^{\,m_\ell} \bigr|_{H^{\frac{1}{2}}(\Gamma)} \leq \frac{2^{-\ell}}{m_\ell}, %\mbox{ and } 
            \\[2ex]
            \delta_\Gamma \widetilde{\vartheta}_{\Gamma, i}^{m_\ell} = \delta_\Gamma \overline{\vartheta}_{\Gamma, i}^{\,m_\ell} \mbox{ in $ H^1(\Gamma) $,} 
        \end{cases}
        \mbox{for $ i = 1, \dots, m_\ell $,}
        \label{kenConv11}
    \end{align}
    and subsequently, we invoke Fact \ref{Fact4} to obtain a sequences  $ \bigl\{ \widetilde{\vartheta}_{i}^{m_\ell} \bigr\}_{\ell = 1}^\infty \subset H^1(\Omega) $, for $ i = 1, \dots, m_\ell $, such that:
    \begin{align}
        & \begin{cases}
            \widetilde{\vartheta}_i^{m_\ell}\trace{\Gamma} = \widetilde{\vartheta}_{\Gamma, i}^{\,m_\ell}, \mbox{ a.e. on $ \Gamma $,}
            \\[1ex]
            \ds \bigl| \widetilde{\vartheta}_{i}^{m_\ell} -\overline{\vartheta}_{i}^{\,m_\ell} \bigr|_{L^2(\Omega)} \leq \frac{2^{-\ell}}{m_\ell},
            \\[1ex]
            \ds \left| \int_\Omega |\nabla \widetilde{\vartheta}_i^{m_\ell}| \, dx -[|D \overline{\vartheta}_i^{\,m_\ell}|]_{\widetilde{\vartheta}_{\Gamma, i}^{\,m_\ell}}(\overline{\Omega}) \right| \leq \frac{2^{-\ell}}{m_\ell},
        \end{cases}
        \mbox{for $ i = 1, \dots, m_\ell $.}
        \label{kenConv13}
    \end{align}
    Based on these, we define a sequence $ \bigl\{ \bm{\widetilde{\vartheta}}_\ell =  \bigl[ \widetilde{\vartheta}_{\ell}, \widetilde{\vartheta}_{\Gamma, \ell} \bigr] \bigr\}_{\ell = 1}^\infty \subset L^\infty(I; H^1(\Gamma)) $, by letting:
    \begin{align}
        & \bm{\widetilde{\vartheta}}_\ell(t) = \bigl[ \widetilde{\vartheta}_{\ell}(t), \widetilde{\vartheta}_{\Gamma, \ell}(t) \bigr] := \bm{\widetilde{\vartheta}}_i^{m_\ell} \mbox{ in $ V_1 $,} 
        \\
        & ~~ \mbox{if $ t \in [t_{i -1}^{m_\ell}, t_{i}^{m_\ell}) $, for $ i = 1, \dots, m_\ell $.}
        \label{kenConv12}
    \end{align}

\noeqref{kenConv11, kenConv13} 
    On account of \eqref{kenConv06}, and \eqref{kenConv10}--\eqref{kenConv12}, we can infer that:
    \begin{align}
        & \begin{cases}
            \bigl| \bm{\widetilde{\vartheta}}_\ell -\bm{\vartheta} \bigr|_{L^2(I; L^2(\Omega) \times H^{\frac{1}{2}}(\Gamma))} \leq 2 \cdot 2^{-\ell} +2 \cdot 2^{-\ell} = 2^{-\ell +2},
            \\[1.5ex]
            \bigl| \delta_\Gamma \widetilde{\vartheta}_{\Gamma, \ell} -\delta_\Gamma \vartheta_\Gamma \bigr|_{L^2(I; H^{\frac{1}{2}}(\Gamma))} \leq 2^{-\ell}, \mbox{ for $ \ell = 1, 2, 3, \dots $,}
        \end{cases}
    \end{align}
    and 
    \begin{align}
        \int_I \left| \int_\Omega \right. & |\nabla \widetilde{\vartheta}_\ell(t)| \, dx - \left. \int_{\overline{\Omega}} d [|D \vartheta(t)|]_{\vartheta_\Gamma(t)} \right| \, dt 
        \\
        & \leq m_\ell \cdot \frac{2^{-\ell}}{m_\ell} +\bigl| \widetilde{\vartheta}_{\Gamma, \ell} -\overline{\vartheta}_{\Gamma, \ell} \bigr|_{L^1(I; L^1(\Gamma))} +(3 +T^{\frac{1}{2}}) \cdot 2^{-\ell}
        \\
        & \leq \bigl( 4 +(1 +T^{\frac{1}{2}}) \mathcal{H}^{N -1}(\Gamma)^{\frac{1}{2}} \bigr) \cdot 2^{-\ell}, \mbox{ for $ \ell = 1, 2, 3, \dots $.}
    \end{align}
    These estimates lead to the convergence \eqref{kenConv01b}.

    Furthermore, by taking a subsequence if necessary, we can include the pointwise convergence \eqref{kenConv01d} as a property of the sequence $ \bigl\{ [\widetilde{\vartheta}_\ell, \widetilde{\vartheta}_{\Gamma, \ell}] \bigr\}_{\ell = 1}^\infty \subset L^\infty(I; V_{\delta_\Gamma}) $. 
\end{proof}
\begin{keyLem}\label{keyLem03-00}
    Let $ I \subset (0, T) $ be an open interval. Let us assume:
    \begin{align}
        & \begin{cases}
            \beta \in L^\infty(I; H^1(\Omega)), ~ \beta(t) \in Y_0(\Omega) \mbox{ a.e. $ t \in I $,}
            \\[0.5ex]
            \bm{\vartheta} = [\vartheta, \vartheta_\Gamma] \in L^2(I; L^2(\Omega) \times H^{\frac{1}{2}}(\Gamma)), ~ \bigl| D \vartheta(\cdot) \bigr|_{\vartheta_\Gamma(\cdot)}(\overline{\Omega}) \in L^1(I).
        \end{cases}
        \label{core02a-00}
    \end{align}
    Then, the function of time:
    \begin{align*}
        & t \in I \mapsto [\beta(t)|D \vartheta(t)|]_{\vartheta_\Gamma(t)}(\overline{\Omega}) \in [0, \infty]
    \end{align*}
    is measurable. In particular, if:
    \begin{align}
        & \begin{cases}
            \beta \in C(\overline{I}; L^2(\Omega)), ~ \beta(t) \in \KS{Y_\mathrm{c}}(\Omega) \mbox{ a.e. $ t \in I $,}
            \\[0.5ex]
            \bm{\vartheta} = [\vartheta, \vartheta_\Gamma] \in C(\overline{I}; L^2(\Omega) \times H^{\frac{1}{2}}(\Gamma)),
        \end{cases}
        \label{core02a-01}
    \end{align}
    then this function is l.s.c. on $ [0, T] $. 
\end{keyLem}
\begin{proof}
    By the assumption \eqref{core02a-00}, we can apply Key-Lemma \ref{keyLem02}, and we can take a sequence of functions $ \{ \bm{\widetilde{\vartheta}}_\ell = [\widetilde{\vartheta}_\ell, \widetilde{\vartheta}_{\Gamma, \ell}] \} \subset L^\infty(I; V_\varepsilon) $, that achieve the convergences \eqref{kenConv01b} and \eqref{kenConv01d}. 

    On this basis, let us fix $ t \in I $ satisfying \eqref{kenConv01d}. Besides, with Fact \ref{Fact0} in mind, let us put:
    \begin{align}
        & 
        \begin{cases}
            \ds A_1 := [({\ts \frac{1}{2}} +\beta(t))|D \vartheta(t)|]_{\vartheta_\Gamma(t)}(\overline{\Omega}),  
            \\[1ex]
            A_2 := [({\ts \frac{1}{2}} +|\beta(t)|_{L^\infty(\Omega)} -\beta(t))|D \vartheta(t)|]_{\vartheta_\Gamma(t)}(\overline{\Omega}),
        \end{cases}
    \end{align}
    and
    \begin{align}
        & 
        \begin{cases}
            \ds a_\ell^1 := [({\ts \frac{1}{2}} +\beta(t))|D \widetilde{\vartheta}_\ell(t)|]_{\widetilde{\vartheta}_{\Gamma, \ell}(t)}(\overline{\Omega}),  
            \\[1ex]
            a_\ell^2 := [({\ts \frac{1}{2}} +|\beta(t)|_{L^\infty(\Omega)} -\beta(t))|D \widetilde{\vartheta}_\ell(t)|]_{\widetilde{\vartheta}_{\Gamma, \ell}(t)}(\overline{\Omega}),
        \end{cases}
        \mbox{for $ \ell = 1, 2,3 , \dots $.}
    \end{align}
    Let $ \bB_\Omega \subset \R^N $ be a sufficiently large open ball such that $ \bB_\Omega \supset \Omega $. Then with the properties of extension $ [{}\cdot{}]^\mathrm{ex} $ (see Section 7.2 in Appendix, for details), one can see that:
    \begin{align}
        & \left\{ \hspace{-3ex} \parbox{7.7cm}{
            \vspace{-2ex}
            \begin{itemize}
                \item $ [\widetilde{\vartheta}_{\Gamma, \ell}(t)]^\mathrm{ex} \to [\vartheta_\Gamma(t)]^\mathrm{ex} $ in $ H^1(\R^N) $, and in $ W^{1, 1}(\bB_\Omega) $,
                \item $ [\widetilde{\vartheta}_{\ell}(t)]_{\vartheta_{\Gamma, \ell}(t)}^\mathrm{ex} \to [\vartheta(t)]_{\vartheta_\Gamma(t)}^\mathrm{ex} $ in $ L^2(\R^N) $, and weakly-$*$ in $ BV(\bB_\Omega) $,
            \end{itemize}
            \vspace{-2ex}
        } \right.
        ~\mbox{  as $ \ell \to \infty $,}
    \end{align}
    \begin{align}
        A_1 ~& = [({\ts \frac{1}{2}} +\beta(t))|D \vartheta(t)|]_{\vartheta_\Gamma(t)}(\overline{\Omega})
        \\
        & = \sup \left\{ \begin{array}{l|l}
            \ds \int_{\bB_\Omega} [\vartheta(t)]_{\vartheta_\Gamma(t)}^\mathrm{ex} \, \mathrm{div} \, \bigl( ({\ts \frac{1}{2}} +\beta(t)) \varphi \bigr) \, dx & 
            \parbox{3.5cm}{$ \varphi \in C_\mathrm{c}^1(\bB_\Omega; \R^N) $, $ |\varphi| \leq 1 $ on $ \bB_\Omega $}
        \end{array} \right\}
        \\
        & \qquad -\int_{\bB_\Omega \setminus \overline{\Omega}} \bigl( {\ts \frac{1}{2}} +[\beta(t)]^\mathrm{ex} \bigr) |\nabla [\vartheta_\Gamma(t)]^\mathrm{ex}| \, d \Gamma
        \\
        & = \sup \left\{ \begin{array}{l|l}
            \ds \lim_{\ell \to \infty} \int_{\Omega} [\widetilde{\vartheta}_\ell(t)]_{\widetilde{\vartheta}_{\Gamma, \ell}(t)} \, \mathrm{div} \, \bigl( ({\ts \frac{1}{2}} +\beta(t)) \varphi \bigr) \, dx & 
            \parbox{3.5cm}{$ \varphi \in C_\mathrm{c}^1(\bB_\Omega; \R^N) $, $ |\varphi| \leq 1 $ on $ \bB_\Omega $}
        \end{array} \right\}
        \\
        & \qquad -\lim_{\ell \to \infty} \int_{\bB_\Omega \setminus \overline{\Omega}} \bigl( {\ts \frac{1}{2}} +[\beta(t)]^\mathrm{ex} \bigr) |\nabla [\widetilde{\vartheta}_{\Gamma, \ell}(t)]^\mathrm{ex}| \, d \Gamma
        \\
        & \leq \varliminf_{\ell \to \infty} \left( \sup \left\{ \begin{array}{l|l}
            \ds \int_{\bB_\Omega} [\widetilde{\vartheta}_\ell(t)]_{\widetilde{\vartheta}_{\Gamma, \ell}(t)} \, \mathrm{div} \, \bigl( ({\ts \frac{1}{2}} +\beta(t)) \varphi \bigr) \, dx & 
            \parbox{3.5cm}{$ \varphi \in C_\mathrm{c}^1(\bB_\Omega; \R^N) $, $ |\varphi| \leq 1 $ on $ \bB_\Omega $}
        \end{array} \right\} \right)
        \\
        & \qquad -\lim_{\ell \to \infty} \int_{\bB_\Omega \setminus \overline{\Omega}} \bigl( {\ts \frac{1}{2}} +[\beta(t)]^\mathrm{ex} \bigr) |\nabla [\widetilde{\vartheta}_{\Gamma, \ell}(t)]^\mathrm{ex}| \, d \Gamma
        \\
        & = \varliminf_{\ell \to \infty} [({\ts \frac{1}{2}} +\beta(t))|D \widetilde{\vartheta}_\ell(t)|]_{\widetilde{\vartheta}_{\Gamma, \ell}(t)}(\overline{\Omega}) = \varliminf_{\ell \to \infty} a_\ell^1, 
    \end{align}
    as well as, 
    \begin{align}
        A_2 ~& = [({\ts \frac{1}{2}} +|\beta(t)|_{L^\infty(\Omega)} -\beta(t))|D \vartheta(t)|]_{\vartheta_\Gamma(t)}(\overline{\Omega})
        \\
        & \leq \varliminf_{\ell \to \infty} [({\ts \frac{1}{2}} +|\beta(t)|_{L^\infty(\Omega)} -\beta(t))|D \widetilde{\vartheta}_\ell(t)|]_{\widetilde{\vartheta}_{\Gamma, \ell}(t)}(\overline{\Omega}) = \varliminf_{\ell \to \infty} a_\ell^2,
    \end{align}
    and furthermore, 
    \begin{align}
        \lim_{\ell \to \infty} \,& \bigl( a_\ell^1 +a_\ell^2 \bigr)  = (1 +|\beta(t)|_{L^\infty(\Omega)}) \lim_{\ell \to \infty} \int_\Omega |\nabla \widetilde{\vartheta}_\ell| \, dx 
        \\
        & = (1 +|\beta(t)|_{L^\infty(\Omega)}) \int_{\overline{\Omega}} d [|D \vartheta(t)|]_{\vartheta_\Gamma(t)} = A_1 +A_2.
    \end{align}
    So, we can apply Fact \ref{Fact0} to the case when $ m = 2 $, and we can observe that 
    \begin{align}
        & \int_{\overline{\Omega}} \, d[\beta(t) |D \vartheta(t)|]_{\vartheta_\Gamma(t)} = \lim_{\ell \to \infty} \int_\Omega |\nabla \widetilde{\vartheta}_\ell(t)| \, dx, \mbox{ for a.e. $ t \in I $.}
    \end{align}
    It implies the measurability of $ t \in I \mapsto [\beta(t) |D \vartheta(t)|]_{\vartheta_\Gamma(t)} \in [0, \infty)  $. 
    %The measurabilities of functions $ [\beta|D \vartheta_n|]_{\vartheta_{\Gamma,n}}(\overline{\Omega}) $, for $n = 1, 2, 3, \dots$, will be verified, similarly. 

Next, we assume \eqref{core02a-01} to verify the lower semi-continuity of $ [\beta(\cdot)|D \vartheta(\cdot)|]_{\vartheta_\Gamma(\cdot)} $ on $ \overline{I} $.  Under \eqref{core02a-01}, we easily check that:
    \begin{align}
        & \left\{ \hspace{-3ex} \parbox{12.9cm}{
            \vspace{-2ex}
            \begin{itemize}
                \item $ [\beta]^\mathrm{ex} : t \in \overline{I} \mapsto [\beta(t)]^\mathrm{ex} \in H^1(\R^N) $ is weakly continuous, and $ [\beta]^\mathrm{ex} \in C(\overline{I}; L^2(\bB_\Omega)) $
                \vspace{-1ex}
            \item $ [\vartheta_\Gamma]^\mathrm{ex} : t \in \overline{I} \mapsto [\vartheta_\Gamma(t)]^\mathrm{ex} \in H^1(\R^N) $ is continuous, 
                \vspace{-1ex}
            \item $ [\vartheta]_{\vartheta_\Gamma}^\mathrm{ex} : t \in \overline{I} \mapsto [\vartheta(t)]_{\vartheta_\Gamma(t)}^\mathrm{ex} \in L^2(\bB_\Omega) $ is continuous.
            \end{itemize}
            \vspace{-2ex}
        } \right. 
    \end{align}
    In view of this, the lower semi-continuity of $ [\beta|D \vartheta|]_{\vartheta_\Gamma}(\overline{\Omega}) $ in $ I $ can be verified, as follows:
    \begin{align}
        & [\beta|D \vartheta|(t)]_{\vartheta_\Gamma(t)}(\overline{\Omega}) = \int_{\bB_\Omega} d [\beta(t) |D \vartheta(t)|]_{\vartheta_\Gamma(t)} -\int_{\bB_\Omega \setminus \overline{\Omega}} [\beta]^\mathrm{ex}(t) |\nabla [\vartheta_\Gamma]^\mathrm{ex}(t)| \, dx
        \\
        & = \sup \left\{ \begin{array}{l|l}
            \ds \lim_{\varsigma \to t} \int_{\bB_\Omega} [\vartheta(\varsigma)]_{\vartheta_\Gamma(\varsigma)}^\mathrm{ex} \mathrm{div} \, ([\beta]^\mathrm{ex}(\varsigma) \varphi) \, dx  & \parbox{3.25cm}{
                $ \varphi \in C_\mathrm{c}^1(\bB_\Omega; \R^N) $, $ |\varphi| \leq 1 $ on $ \Omega $
            }
        \end{array} \right\} 
        \\
        & \qquad -\lim_{\varsigma \to t} \int_{\bB_\Omega \setminus \overline{\Omega}} [\beta]^\mathrm{ex}(\varsigma) |\nabla [\vartheta_\Gamma]^\mathrm{ex}(\varsigma)| \, dx
        \\
        & \leq \varliminf_{\varsigma \to t} \left( \int_{\bB_\Omega} d [\beta(\varsigma) |D \vartheta(\varsigma)|]_{\vartheta_\Gamma(\varsigma)} -\int_{\bB_\Omega \setminus \overline{\Omega}} [\beta]^\mathrm{ex}(\varsigma) |\nabla [\vartheta_\Gamma]^\mathrm{ex}(\varsigma)| \, dx \right)
        \\
        & = \varliminf_{\varsigma \to t} \, [\beta|D \vartheta|(\varsigma)]_{\vartheta_\Gamma(\varsigma)}(\overline{\Omega}), \mbox{ for any $ t \in \overline{I} $.}
    \end{align}
    Thus, we conclude this key-lemma. 
\end{proof}
\begin{keyLem}\label{keyLem03}
    Let $ I \subset (0, T) $ be a fixed open interval. Let $ \beta \in L^\infty(I; H^1(\Omega)) $ and $ [\vartheta, \vartheta_\Gamma] $ be as in \eqref{core02a-00}. Let $ \{ \beta_n \}_{n = 1}^\infty \subset L^\infty(I; H^1(\Omega)) $ be a sequence of functions, such that:
%\begin{subequations}\label{core01}
    \begin{align}\label{core01}
    & 
        \begin{cases}
            \beta_n \to \beta \mbox{ in $ L^2(I; L^2(\Omega)) $, and weakly-$*$ in $ L^\infty(I; H^1(\Omega)) $, as $ n \to \infty $,}
            \\[1ex]
            \bigl\{ \beta_n(t) \, \bigl| \, n = 1, 2, 3, \dots \bigr\} \subset Y_{0}(\Omega), ~\mbox{ a.e. $ t \in I $.}
%        \\[1ex]
%        \beta_n(t) \to \beta(t) \mbox{ in $ L^2(\Omega) $, and weakly in $ H^1(\Omega) $, }
%        \mbox{a.e. $ t \in I $, as $ n \to \infty $.}
        \end{cases}
    \end{align}
%\end{subequations}
    Let  $ \{[\vartheta_n, \vartheta_{\Gamma, n}]\}_{n = 1}^\infty \subset L^2(I; L^2(\Omega) \times H^{\frac{1}{2}}(\Gamma)) $ be a sequence of functional pairs such that:
\begin{equation}\label{core02a}
\left\{ ~ \parbox{8cm}{
    $ \bigl\{  [|D \vartheta_n({}\cdot{})|]_{\vartheta_\Gamma({}\cdot{})}(\overline{\Omega}) \bigr\}_{n = 1}^\infty \subset L^1(I) $, 
    \\[1ex]
    $ \ds L_* := \sup_{n \in \N} \left\{  \bigl| [|D \vartheta_n({}\cdot{})|]_{\vartheta_\Gamma({}\cdot{})}(\overline{\Omega}) \bigr|_{L^1(I)} \right\} < \infty  $,
} \right.
\end{equation}
and
    \begin{align}\label{core02b}
        & \bm{\vartheta}_n = [\vartheta_n, \vartheta_{\Gamma, n}] \to \bm{\vartheta} = [\vartheta, \vartheta_\Gamma] \mbox{ in $ L^2(I; L^2(\Omega) \times H^{\frac{1}{2}}(\Gamma)) $, as $ n \to \infty $,}
    \end{align}
Then, it holds that: 
\begin{equation}\label{kenCore01}
    \varliminf_{n \to \infty} \int_I [ \beta_n(t) |D \vartheta_n(t)| ]_{\vartheta_{\Gamma, n}(t)}(\overline{\Omega}) \, dt \geq \int_I [ \beta(t) |D \vartheta(t)| ]_{\vartheta_{\Gamma}(t)}(\overline{\Omega}) \, dt.
\end{equation}
\end{keyLem}
\begin{proof}
    For the proof of \eqref{kenCore01}, it is sufficient to consider only the case when:
    \begin{align}
        \Phi_* ~& = \varliminf_{n \to \infty} \int_I[\beta_n(t)|D \vartheta_n(t)|]_{\vartheta_{\Gamma, n}(t)}(\overline{\Omega}) \, dt  < \infty,
    \end{align}
    because the other case is trivial. 
    In this case, we can find a subsequence $ \{n_k\}_{k = 1}^\infty \subset \{n\}_{n = 1}^\infty $, such that:
    \begin{align}
        & \Phi_* = \lim_{k \to \infty} \int_I [\beta_{n_k}(t)|D \vartheta_{n_k}(t)|]_{\vartheta_{\Gamma, n_k}(t)}(\overline{\Omega}) \, dt.
        \label{kenCore02}
    \end{align}
\noeqref{core02a} 
    Then, by taking more one subsequence if necessary, the assumptions \eqref{core01}--\eqref{core02b}, and Remark \ref{Rem.ex_hm} in Appendix enable us to suppose:
    \begin{align}
        & \left\{ \hspace{-3ex} \parbox{9.5cm}{
            \vspace{-2ex}
            \begin{itemize}
                \item $ 0 \leq [\beta]^\mathrm{ex}(t) \in L^\infty(\R^N) \cap H^1(\R^N) $, for a.e. $ t \in I $,
                    \vspace{-3.5ex}
                \item $ \ds L_*^\mathrm{ex} := \sup_{n \in \N} \, \bigl| |D [\vartheta_{n}]_{\vartheta_{\Gamma, n}}^\mathrm{ex}|(\bB_\Omega) \bigr|_{L^1(I)} < \infty $.
            \end{itemize}
            \vspace{-2ex}
        } \right.
        \label{kenCore04}
    \end{align}
    \begin{align}
        & \left\{ \hspace{-3ex} \parbox{7.7cm}{
            \vspace{-2ex}
            \begin{itemize}
                \item $ [\beta_{n_k}]^\mathrm{ex} \to [\beta]^\mathrm{ex} $ in $ L^2(I; L^2(\R^N)) $, and weakly-$*$ in $ L^\infty(I; H^1(\Omega)) $,
                    \vspace{-0.5ex}
                \item $ [\vartheta_{\Gamma, n_k}]^\mathrm{ex} \to [\vartheta_\Gamma]^\mathrm{ex} $ in $ L^2(I; H^1(\R^N)) $, and in $ L^1(I; W^{1, 1}(\bB_\Omega)) $,
                    \vspace{-0.5ex}
                \item $ [\vartheta_{n_k}]_{\vartheta_{\Gamma, n_k}}^\mathrm{ex} \to [\vartheta]_{\vartheta_\Gamma}^\mathrm{ex} $ in $ L^2(I; L^2(\R^N)) $,,
            \end{itemize}
            \vspace{-2ex}
            }
        \right.
        \mbox{as $ k \to \infty $,}
        \label{kenCore05}
    \end{align}
    and
    \begin{align}
        & \left\{ \hspace{-3ex} \parbox{8cm}{
            \vspace{-2ex}
            \begin{itemize}
                \item $ [\beta_{n_k}]^\mathrm{ex}(t) \to [\beta]^\mathrm{ex}(t) $ in $ L^2(\R^N) $, and weakly in $ H^1(\R^N) $,
                    \vspace{-0.5ex}
                \item $ [\vartheta_{\Gamma, n_k}]^\mathrm{ex}(t) \to [\vartheta_\Gamma]^\mathrm{ex}(t) $ in $ H^1(\R^N) $, and in $ W^{1, 1}(\bB_\Omega) $,
                    \vspace{-0.5ex}
                \item $ [\vartheta_{n_k}(t)]_{\vartheta_{\Gamma, n_k}(t)}^\mathrm{ex} \to [\vartheta(t)]_{\vartheta_\Gamma(t)}^\mathrm{ex} $ in $ L^2(\R^N) $,
            \end{itemize}
            \vspace{-2ex}
        } \right.
        ~\mbox{ a.e. $ t \in I $, as $ k \to \infty $.}
        \label{kenCore03}
    \end{align}

\noeqref{kenCore04, kenCore05} 
    Now, let us take any $ r > 0 $. Then, on account of \eqref{01_[bt|Du|]_gm}, \eqref{kenCore02}--\eqref{kenCore03}, and Fatou's lemma,  we can compute that:
    \begin{align}
        & \int_I [\beta(t) |D \vartheta(t)|]_{\vartheta_\Gamma(t)}(\overline{\Omega}) \,dt \leq \int_I [(\beta(t) +r) |D \vartheta(t)|]_{\vartheta_\Gamma(t)}(\overline{\Omega}) \,dt
        \\
        & = \int_I \sup \left\{ \begin{array}{l|l} 
            \ds \int_{\bB_\Omega} [\vartheta(t)]^\mathrm{ex} \, \mathrm{div} \bigl( [\beta(t)]^\mathrm{ex} +r) \varphi \bigr) \, dx & \parbox{3cm}{ 
               $ \varphi \in C_\mathrm{c}^1(\Omega; \R^N) $, $ |\varphi| \leq 1 $ on $ \bB_\Omega $ 
            }
            \end{array} \right\} \, dt
%        \\
%        & = \int_I \int_{\bB_\Omega} \bigl[ [\beta(t)]^\mathrm{ex} +r \bigr]^* \, d |D [\theta(t)]_{\vartheta_\Gamma(t)}| \, dt
            \\
            & \qquad 
            -\int_I \int_{\bB_\Omega \setminus \overline{\Omega}} ([\beta(t)]^\mathrm{ex} +r) |\nabla [\vartheta_\Gamma(t)]^\mathrm{ex}| \, dx dt
            \\
            & \leq \int_I \varliminf_{k \to \infty} \left( \sup \left\{ \begin{array}{l|l} 
                \ds \int_{\bB_\Omega} [\vartheta_{n_k}(t)]^\mathrm{ex} \, \mathrm{div} \bigl( [\beta_{n_k}(t)]^\mathrm{ex} +r) \varphi \bigr) \, dx & \parbox{3cm}{ 
               $ \varphi \in C_\mathrm{c}^1(\Omega; \R^N) $, $ |\varphi| \leq 1 $ on $ \bB_\Omega $ 
            }
            \end{array} \right\} \right)  dt
%        \\
%        & \leq \int_I \int_{\bB_\Omega} \varliminf_{k \to \infty} \bigl[ [\beta_{n_k}(t)]^\mathrm{ex} +r \bigr]^* \, d |D [\theta_{n_k}(t)]_{\vartheta_{\Gamma, n_k}(t)}| \, dt
            \\
            & \qquad -\lim_{k \to \infty} \int_I \int_{\bB_\Omega \setminus \overline{\Omega}} ([\beta_{n_k}(t)]^\mathrm{ex} +r) |\nabla [\vartheta_{\Gamma, n_k}(t)]^\mathrm{ex}| \, dx dt
            \\
            & \leq \varliminf_{k \to \infty} \int_I [(\beta_{n_k}(t) +r) |D \vartheta_{n_k}(t)|]_{\vartheta_{\Gamma, n_k}(t)}(\overline{\Omega}) \,dt
             \leq \Phi_* +L_* r, \mbox{ for any $ r > 0 $.}
    \end{align}
    
    Since $ r > 0 $ is arbitrary, the above estimate finishes the proof of this key-lemma. 
\end{proof}
%    Now, for every $\delta>0$, and $ 0 \leq \beta \in L^2(\Omega)  $, let us define:
%%\begin{subequations}\label{rx-Phi}
%\begin{align}\label{DefOfPhi_delta}
%    \Phi_{\delta} & : \KS{\bm{\theta}} = [\theta,\theta_\Gamma] \in H \mapsto \Phi_{\delta}(\beta; \KS{\bm{\theta}}) = \Phi_{\delta}(\beta; \theta,\theta_\Gamma)
%\\[2ex]
%    &:= \left\{\begin{array}{cl}
%    \multicolumn{2}{l}{\ds  \int_\Omega \left(\beta f_\delta(\nabla\theta)+\frac{\delta^2}{2}|\nabla\theta|^2\right)\,dx +\frac{\kappa_\Gamma^2}{2}\int_\Gamma |\nabla_\Gamma \theta_\Gamma|^2\,d\Gamma,}
%    \\[2.5ex]
%        & \mbox{ if $ [\theta, \theta_\Gamma] \in V_\delta$,}
%\\[1.5ex]
%\infty, & \mbox{ otherwise,}
%\end{array}\right.
%\end{align}
%with the use of the following relaxation of Euclidean norm:
%\begin{equation}\label{f_d}
%f_\delta: \omega \in \R^N \mapsto \gamma_\delta(\omega):=\sqrt{\delta^2 + |\omega|^2} -\delta \in [0,\infty).
%\end{equation}
%    Then, as a straightforward consequence of Key-Lemma \ref{keyLem03}, we will obtain the following corollary. 
%\end{subequations}
    \begin{cor}\label{Rem.Phi_0^I(bt)}
        Let $ I \subset (0, T) $ be a fixed open interval, and let $ \delta \geq 0 $ be a fixed constant. Let $ \beta \in L^\infty(I; H^1(\Omega)) $ be a function such that $ \log \beta \in L^\infty(Q) $. Then, a functional:
        \begin{align}\label{DefOfPhi_0^I(bt)}
            \Phi_\delta^I(\beta;{}\cdot{}) : \bm{\vartheta} & = [\vartheta, \vartheta_\Gamma] \in L^2(I; H) \mapsto \Phi_\delta^I(\beta; \bm{\vartheta}) = \Phi_\delta^I(\beta; \vartheta, \vartheta_\Gamma) 
            \\
            & := \int_I \Phi_\delta(\beta(t); \bm{\vartheta}(t)) \, dt \in [0, \infty], 
        \end{align}
        is well-defined, and it is proper l.s.c. and convex on $ L^2(I; H) $, with the effective domain:
        \begin{align}
            & D(\Phi_\delta^I(\beta;{}\cdot{})) = \begin{cases}
                \left\{ \begin{array}{l|l}
                    \bm{\tilde{\vartheta}} = [\tilde{\vartheta}, \tilde{\vartheta}_\Gamma] \in L^2(I; L^2(\Omega) \times H^1(\Omega)) & [|D\vartheta|]_{\vartheta_\Gamma}(\overline{\Omega}) \in L^1(I)
                \end{array} \right\}, 
                \\
                \qquad \mbox{if $ \delta = 0 $,}
                \\[2ex]
                L^2(I; V_1) ~ \bigl( = L^2(I; V_\delta) \bigr), \mbox{if $ \delta > 0 $.}
            \end{cases}
         \end{align}
    \end{cor}
    \begin{proof}
        When $ \delta = 0 $, This corollary will be a straightforward consequence of \eqref{Phi_gm(bt.)} and Key-Lemma \ref{keyLem03}. Meanwhile, the case when $ \delta > 0 $ can be said trivial. 
    \end{proof}
\noeqref{core02a} 
\begin{keyLem}\label{keyLem04}
    Let $ I \subset (0, T) $ be a fixed open interval. Let $ \beta \in L^\infty(I; H^1(\Omega)) $ and $ \{ \beta_n \}_{n = 1}^\infty \subset L^\infty(I; H^1(\Omega)) $ be as in \eqref{core01}--\eqref{core02b}. Additionally, let us assume that: 
\begin{align}
    \left\{ \rule{0pt}{32pt} \right. & 
    \\[-60pt]
    & \ds \beta \geq \delta_\beta \mbox{ and } \inf_{n \in \N} \beta_n \geq \delta_\beta ~\mbox{ a.e. in $ I \times \Omega $, for some constant $ \delta_\beta > 0 $,}
    \label{kenCore09a}
    \\[1ex]
    & \ds \int_I [\beta_n(t) |D \vartheta_n(t)|]_{\vartheta_{\Gamma,n}(t)}(\overline{\Omega}) \, dt \to \int_I [\beta(t) |D \vartheta(t)|]_{\vartheta_\Gamma(t)}(\overline{\Omega}) \, dt \mbox{ as $ n \to \infty $,} 
    \label{kenCore09b}
\end{align}
$ \varrho \in L^\infty(I; H^1(\Omega)) \cap L^\infty(I \times \Omega) $ and $ \{ \varrho_n \}_{n = 1}^\infty \subset L^\infty(I; H^1(\Omega)) \cap L^\infty(I \times \Omega) $, and:
    \begin{equation}\label{kenCore10}
    M_* := |\varrho|_{L^\infty(I \times \Omega)} \vee \sup_{n \in \N} |\varrho_n|_{L^\infty(I \times \Omega)} < \infty,
\end{equation}
and
    \begin{align}\label{kenCore11}
    \varrho_n(t) \to \varrho(t) ~& \mbox{in $ L^2(\Omega) $, and weakly in $ H^1(\Omega) $, a.e. $ t \in I $, as $ n \to \infty $.}
\end{align}  
Then, it holds that:
    \begin{equation}\label{kenCore20}
    \int_I [\varrho_n(t) |D \vartheta_n(t)|]_{\vartheta_{\Gamma,n}(t)}(\overline{\Omega}) dt \to \int_I [\varrho(t) |D \vartheta(t)|]_{\vartheta_{\Gamma}(t)}(\overline{\Omega}) \, dt \mbox{ as $ n \to \infty $.}
\end{equation}
\end{keyLem}
\begin{proof}
    With Fact \ref{Fact0} in mind, we put:
    \begin{align}
        & \begin{cases}
            \ds A_1 := \int_I [\varrho(t)|D \vartheta(t)|]_{\vartheta_\Gamma(t)}(\overline{\Omega}) \, dt,
            \\[2ex]
            \ds A_2 := \int_I \bigl[ \bigl({\ts \frac{M_*}{\delta_\beta}}\beta(t) -\varrho(t) \bigr)|D \vartheta(t)| \bigr]_{\vartheta_\Gamma(t)}(\overline{\Omega}) \, dt,
        \end{cases}
    \end{align}
    and 
    \begin{align}
        \begin{cases}
            \ds a_n^1 := \int_I [\varrho_n(t)|D \vartheta_n(t)|]_{\vartheta_{\Gamma, n}(t)}(\overline{\Omega}) \, dt,
            \\[2ex]
            \ds a_n^2 := \int_I \bigl[ \bigl({\ts \frac{M_*}{\delta_*}}\beta_n(t) -\varrho_n(t) \bigr)|D \vartheta_n(t)| \bigr]_{\vartheta_{\Gamma, n}(t)}(\overline{\Omega}) \, dt,
        \end{cases}
        \mbox{for $ n = 1, 2, 3, \dots $.}
    \end{align}
    Then, owing to \eqref{core01}--\eqref{core02b}, \eqref{kenCore09a}, \eqref{kenCore10}, and \eqref{kenCore11}, we can apply Key-Lemma \ref{keyLem03}, and can see that:
    \begin{align}\label{kenCore12}
        & \varliminf_{n \to \infty} a_n^k \geq A_k, \mbox{ for $ k = 1, 2 $.}
    \end{align}
    Also, on  account of \eqref{01_[bt|Du|]_gm} and \eqref{kenCore09b}, one can observed that:
    \begin{align}
        \varlimsup_{n \to \infty} (a_n^1 +a_n^2) ~& = \frac{M_*}{\delta_\beta} \lim_{n \to \infty} \int_I [\beta_n(t)|D \vartheta_n(t)|]_{\vartheta_{\Gamma, n}(t)}(\overline{\Omega}) \, dt
        \\
        & = \frac{M_*}{\delta_\beta} \int_I [\beta(t)|D \vartheta(t)|]_{\vartheta_{\Gamma}(t)}(\overline{\Omega}) \, dt = A_1 +A_2.
    \end{align}
    The required convergence \eqref{kenCore20} will be verified  by applying Fact \ref{Fact0} in the case when $ m  =2 $. 
\end{proof}

\begin{keyLem}\label{keyLem05}
    Let $ I \subset (0, T) $ be a fixed open interval. Let $ \delta_0 \geq 0 $ and $ \{ \delta_n \}_{n = 1}^\infty \subset [0, \infty) $ be such that:
    \begin{align}
        & \delta_n \to \delta_0 \mbox{ as $ n \to \infty $.}
    \end{align}
    For a function $ \beta \in L^\infty(I; H^1(\Omega)) $ and a sequence of functions $ \{ \beta_n \}_{n = 1}^\infty \subset L^\infty(I; H^1(\Omega)) $ as in \eqref{core01}, let us assume that:
    \begin{align}\label{Gm-conv01b}
        & \begin{cases}
            \log \beta \in L^\infty(I \times \Omega), ~ \{ \log \beta_n \}_{n = 1}^\infty \subset L^\infty(I \times \Omega),
            \\[0.5ex]
            \ds \sup_{n \in \N} \, \bigl| \log \beta_n \bigr|_{L^\infty(Q)} < \infty.
        \end{cases}
    \end{align}
    Then, a sequence of convex functions $ \{ \Phi_{\delta_n}^I(\beta_n;{}\cdot{}) \}_{n = 1}^\infty $ $ \Gamma $-converges to a convex function $ \Phi_{\delta_0}^I(\beta;{}\cdot{}) $ on $ L^2(I; H) $, as $ n \to \infty $.
\end{keyLem}
\begin{proof}
    First, we consider the positive case of the constant $ \delta_0 $. If $ \delta_0 > 0 $, then the lower-bound condition for $ \{ \Phi_{\delta_n}^I(\beta_n;{}\cdot{}) \}_{n = 1}^\infty $ will be seen from the continuities of $ L^p $-based norms, for $ p \geq 1 $, and the uniform convergence of $ f_{\delta_n} \to f_{\delta_0} $ on $ \R^N $ as $ n \to \infty $. Meanwhile, for any $ \bm{\hat{\vartheta}} = [\hat{\vartheta}, \hat{\vartheta}_\Gamma] \in D(\Phi_{\delta_0}^I(\beta;{}\cdot{})) $, we can take a constant-sequence $ \{ \bm{\hat{\vartheta}}, \bm{\hat{\vartheta}}, \bm{\hat{\vartheta}}, \dots \} $ as the sequence to realize the condition of optimality for  $ \{ \Phi_{\delta_n}^I(\beta_n;{}\cdot{}) \}_{n = 1}^\infty $.

    Next, we consider the case when $ \delta_0 = 0 $. In this case, we can apply Key-Lemma \ref{keyLem03}, and can verify the condition of lower bound for $ \{ \Phi_{\delta_n}^I(\beta_n;{}\cdot{}) \}_{n = 1}^\infty $ as follows:
    \begin{align}
        & \varliminf_{n \to \infty} \Phi_{\delta_n}^I(\beta_n; \bm{\check{\vartheta}}_n) \geq \varliminf_{n \to \infty} \Phi_{0}^I(\beta_n; \bm{\check{\vartheta}}_n) \geq \Phi_0^I(\beta; \bm{\check{\vartheta}}). 
    \end{align}

    Now, our remaining task will be to verify the optimality condition  for $ \{ \Phi_{\delta_n}^I(\beta_n;{}\cdot{}) \}_{n = 1}^\infty $ when $ \delta_0 = 0 $. Let us fix any $ \bm{\hat{\vartheta}} =  [\hat{\vartheta}, \hat{\vartheta}_\Gamma] \in D(\Phi_0^I(\beta;{}\cdot{})) $. Then, invoking Corollary \ref{Rem.Phi_0^I(bt)}, and applying Key-Lemma \ref{keyLem02} to the case when $ \bm{\vartheta} = \bm{\hat{\vartheta}} $ and $ \delta_\Gamma = \kappa_\Gamma > 0 $, we can find a sequence of functions $ \bm{\widetilde{\vartheta}}_\ell = [\widetilde{\vartheta}_\ell, \widetilde{\vartheta}_{\Gamma, \ell}] \in L^\infty(I; V_{\kappa_\Gamma}) $, such that:
    \begin{align}\label{Gm-conv10}
        & \begin{cases}
            \bm{\widetilde{\vartheta}}_\ell \to \bm{\hat{\vartheta}} \mbox{ in $ L^2(I; L^2(\Omega) \times H^1(\Omega)) $,}
            \\[1ex]
            \ds \int_I \left| \int_\Omega |\nabla \widetilde{\vartheta}_\ell(t)| -\int_{\overline{\Omega}} d [|D \hat{\vartheta}_\ell|]_{\hat{\vartheta}_{\Gamma, \ell}} \right| \, dt \to 0,
        \end{cases} 
        \mbox{as $ \ell \to \infty $.}
    \end{align}
    Also, let us take a subsequence $ \{n_\ell\}_{\ell}^\infty \subset \{n\}_{n = 1}^\infty $, such that:
    \begin{align}\label{Gm-conv11}
        & \frac{\delta_{n_\ell}^2}{2} \int_\Omega |\nabla \widetilde{\vartheta}_\ell|^2 \, dx \leq 2^{-\ell}, \mbox{ for $ \ell = 1, 2,3, \dots $.}
    \end{align}
    On this basis, we define a sequence of functional pairs $ \{ \bm{\hat{\vartheta}}_n =  [\hat{\vartheta}_n, \hat{\vartheta}_{\Gamma, n}] \}_{n = 1}^\infty $, by letting:
    \begin{align}\label{Gm-conv13}
        \bm{\hat{\vartheta}}_n = [\hat{\vartheta}_n, \hat{\vartheta}_{\Gamma, n}] ~& := \begin{cases}
            \bm{\widetilde{\vartheta}}_\ell = [\widetilde{\vartheta}_\ell, \widetilde{\vartheta}_{\Gamma, \ell}], \mbox{ if $ n_{\ell} \leq n < n_{\ell +1} $,}
            \\[1ex]
            \bm{\widetilde{\vartheta}}_1 = [\widetilde{\vartheta}_1, \widetilde{\vartheta}_{\Gamma, 1}], \mbox{ if $ 1 \leq n \leq n_1 $.}
        \end{cases}
    \end{align}
\noeqref{Gm-conv10} 
    Here, owing to \eqref{Phi_gm(bt.)}, \eqref{Gm-conv01b}--\eqref{Gm-conv11}, we have:
    \begin{align}\label{Gm-conv12}
        & \begin{cases}
            \bm{\hat{\vartheta}}_n \to \bm{\hat{\vartheta}} \mbox{ in $ L^2(I; L^2(\Omega) \times H^1(\Omega)) $,}
            \\[1ex]
            \delta_n \hat{\vartheta}_{\Gamma, n} \to 0 \mbox{ in $ L^2(I; H^1(\Omega)) $,}
        \end{cases}
        \mbox{as $ n \to \infty $.}
    \end{align}
\noeqref{Gm-conv11, Gm-conv13} 
    Additionally, with \eqref{f_d}, \eqref{Gm-conv10}--\eqref{Gm-conv12} in mind, let us apply Key-Lemma \ref{keyLem04}, by replacing:
    \begin{align}
        & \left\{ \hspace{-3ex} \parbox{9.25cm}{
            \vspace{-2ex}
            \begin{itemize}
                \item $ \{ \beta_n \}_{n = 1}^\infty $ (in Key-Lemma \ref{keyLem04}) by $ \{1, 1, 1, \dots \} $,
                \item $ \{ \varrho_n \}_{n = 1}^\infty $ by $ \{ \beta_n \}_{n = 1}^\infty $ (in this Key-Lemma),
            \end{itemize}
            \vspace{-2ex}
    } \right. \mbox{respectively.}
    \end{align}
    Then, it is observed that:
    \begin{align}
        & \bigl| \Phi_{\delta_n}^I(\beta_n; \bm{\hat{\vartheta}}_n) -\Phi_0^I(\beta; \bm{\hat{\vartheta}}) \bigr|
        \\
        & \leq \int_I \int_\Omega \beta_n(t) \bigl| f_{\delta_n}(\nabla \hat{\vartheta}_n) -|\nabla \hat{\vartheta}_n| \bigr| \, dx dt
        \\
        & \qquad +\left| \int_I \int_\Omega \beta_n(t)|\nabla \hat{\vartheta}_n(t)| \, dt -\int_I \int_{\overline{\Omega}} d [\beta(t)|D \hat{\vartheta}(t)|]_{\hat{\vartheta}_\Gamma(t)} \, dt \right|
        \\
        & \qquad +\frac{\delta_n^2}{2} \int_I \int_\Omega |\nabla \hat{\vartheta}_n|^2 \, dx dt +\frac{\kappa_\Gamma^2}{2} \int_I \left| \int_\Omega |\nabla_\Gamma \bigl( \hat{\vartheta}_{\Gamma, n}|^2 -|\nabla_\Gamma \hat{\vartheta}(t)|^2 \right| d \Gamma
        \\
        & \leq \left| \int_I d [\beta_n(t)|D \hat{\vartheta}_n(t)|]_{\hat{\vartheta}_{\Gamma, n}(t)}(\overline{\Omega}) \, dt -\int_{I} d [\beta(t)|D \hat{\vartheta}(t)|]_{\hat{\vartheta}_\Gamma(t)}(\overline{\Omega}) \, dt \right| 
        \\
        & \qquad +\delta_n \sup_{n \in \N} |\beta_n|_{L^1(I; L^1(\Omega))} +\frac{1}{2} \bigl| \delta_n \hat{\vartheta}_n \bigr|_{L^2(I; H^1(\Omega))}^2 
        \\
        & \qquad +\frac{\kappa_\Gamma^2}{2} |\hat{\vartheta}_{\Gamma, n} -\hat{\vartheta}_\Gamma|_{L^2(I; H^1(\Gamma))} \sup_{n \in \N} \, \bigl( |\hat{\vartheta}_{\Gamma, n}|_{L^2(I; H^1(\Gamma))} +|\hat{\vartheta}_{\Gamma}|_{L^2(I; H^1(\Gamma))}  \bigr)
        \\
        & \to 0, \mbox{ as $ n \to \infty $.}
        \label{Gm-conv20}
    \end{align}
    The convergences \eqref{Gm-conv12} and \eqref{Gm-conv20} suggest that the sequence $ \{ \bm{\hat{\vartheta}}_n = [\hat{\vartheta}_n, \hat{\vartheta}_{\Gamma, n}] \} $ given in \eqref{Gm-conv13} realizes the condition of optimality for $ \{ \Phi_{\delta_n}^I(\beta_n;{}\cdot{}) \}_{n = 1}^\infty $. 

    Thus, we conclude the proof of this Key-Lemma.
\end{proof}
\begin{cor}\label{CorOfGm-conv01}
    Let $ \bm{\eta}_0 = [\eta_0, \eta_{\Gamma, 0}] \in V_\varepsilon \cap \Lambda_0^1 $ and  $ \bm{\theta}_0 = [\theta_0, \theta_{\Gamma, 0}] \in W \cap \Lambda_{r_0}^{r_1} $ be as in the assumption (A5). Then, there exists a sequence of functions $ \{ \bm{\theta}_0^\delta = [\theta_0^{\delta}, \theta_{\Gamma, 0}^\delta] \}_{\delta \in (0, 1]} \subset V_1 $, such that:
    \begin{align}\label{Gm-conv100}
        & \bigl\{ \bm{\theta}_0^\delta = [\theta_0^{\delta}, \theta_{\Gamma, 0}^\delta] \bigr\}_{\delta \in (0, 1]} \subset \Lambda_{r_0}^{r_1}, 
    \end{align}
    and
    \begin{align}\label{Gm-conv101}
        & \begin{cases}
            \bm{\theta}_0^\delta \to \bm{\theta}_0 \mbox{ in $ L^2(\Omega) \times H^1(\Gamma) $,}
            \\[1ex]
            \Phi_{\delta}(\alpha(\eta_0); \bm{\theta}_0^\delta) \to \Phi_0(\alpha(\eta_0); \bm{\theta}_0),
        \end{cases}
        \mbox{as $ \delta \downarrow 0 $.}
    \end{align}
\end{cor}
\begin{proof}
With Fact \ref{Fact4}  in mind, we can find an approximating sequence $ \{ \widetilde{u}_\ell \}_{\ell = 1}^\infty \subset H^1(\Omega) $, % and a sequence $ \{ \widetilde{\delta}_\ell \}_{\ell = 1}^\infty \subset [0, \delta_0 +1] $, 
    such that: 
    \begin{align}\label{Gm-conv31-00}
        & \widetilde{u}_\ell = \theta_{\Gamma, 0}, \mbox{ a.e. on $ \Gamma $, for $ \ell = 1, 2, 3, \dots $,}
    \end{align}
    \vspace{-1ex}
    and
    %\vspace{-1ex}
    \begin{align}
        & \begin{cases}
            \widetilde{u}_\ell \to \theta_0 \mbox{ in $ L^2(\Omega) $,} 
            \\[1ex]
            \ds \int_\Omega \alpha(\eta_0) |\nabla \widetilde{u}_\ell| \, dx = \int_{\overline{\Omega}} d [\alpha(\eta_0) |D \widetilde{u}_\ell|]_{\theta_{\Gamma, 0}} 
            \\[1ex]
            \hspace{4ex}\ds\to \int_{\overline{\Omega}} d [\alpha(\eta_0) |D \theta_0|]_{\theta_{\Gamma, 0}}, \mbox{ ~as $ \ell \to \infty $.}
        \end{cases}
        \label{Gm-conv31}
    \end{align}
    Besides, for any $ \ell \in \N $, we let:
    \begin{align}
        & \widetilde{\theta}_\ell := (\widetilde{u}_\ell \vee r_0) \wedge r_1 \mbox{ a.e. in $ \Omega $,}
        \label{Gm-conv33}
    \end{align}
    and let $ \widetilde{\delta}_\ell \in [0, 1) $ be a constant such that:
    \begin{align}
        & \begin{cases}
            0 < \widetilde{\delta}_{\ell +1} <  \widetilde{\delta}_{\ell}  \leq 2^{-\ell},
            \\[1.5ex]
            \ds \frac{\widetilde{\delta}_\ell^2}{2} \int_\Omega |\nabla \widetilde{\theta}_\ell|^2 \, dx \leq 2^{-\ell}, \mbox{ for $ \ell = 1, 2, 3, \dots $.}
        \end{cases}
        \label{Gm-conv32}
    \end{align}
    Then, from \eqref{f_d}, \eqref{Gm-conv33}, \eqref{Gm-conv32}, and Fact \ref{Fact4}, one can deduce that:
    \begin{align}
        & \Phi_0(\alpha(\eta_0); \bm{\theta}_0) \leq  \varliminf_{\ell \to \infty} \Phi_0(\alpha(\eta_0); \widetilde{\theta}_\ell, \theta_{\Gamma, 0})
        \\
        & \leq  \varliminf_{\ell \to \infty} \Phi_{\widetilde{\delta}_\ell}(\alpha(\eta_0); \widetilde{\theta}_\ell, \theta_{\Gamma, 0}) \leq \varlimsup_{\ell \to \infty} \Phi_{\widetilde{\delta}_\ell}(\alpha(\eta_0); \widetilde{\theta}_\ell, \theta_{\Gamma, 0})
        \\
        & \leq \varlimsup_{\ell \to \infty} \int_\Omega \alpha(\eta_0) f_{\widetilde{\delta}_\ell}(\nabla \widetilde{u}_\ell) \, dx +\lim_{\ell \to \infty} \frac{\widetilde{\delta}_\ell^2}{2} \int_\Omega |\nabla \widetilde{\theta}_\ell|^2 \, dx +\frac{\kappa_\Gamma^2}{2} \int_\Gamma |\nabla_\Gamma \theta_{\Gamma, 0}|^2  \, d \Gamma
        \\
        & \leq \lim_{\ell \to \infty} \left( \int_\Omega \alpha(\eta_0) |\nabla \widetilde{u}_\ell| \, dx +\frac{\kappa_\Gamma^2}{2} \int_\Gamma |\nabla_\Gamma \theta_{\Gamma, 0}|^2  \, d \Gamma \right) +|\alpha(\eta_0)|_{L^1(\Omega)} \lim_{\ell \to \infty} \widetilde{\delta}_\ell 
        \\
        & = \Phi_0(\alpha(\eta_0); \bm{\theta}_0).
        \label{Gm-conv40}
    \end{align}
\noeqref{Gm-conv31, Gm-conv33, Gm-conv32}
    On account of \eqref{Gm-conv31-00}--\eqref{Gm-conv40}, we will conclude that a sequence $ \{ \bm{\theta}_0^\delta = [\theta_0^\delta, \theta_{\Gamma,0}^{\delta}] \}_{\delta \geq 0 } \subset V_1 $, defined as:
    \begin{align}
        \bm{\theta}_0^\delta = [\theta_0^\delta, \theta_{\Gamma, 0}^{\delta}] := & \left\{ \begin{array}{ll}
            [\widetilde{\theta}_\ell, \theta_{\Gamma, 0}], & \mbox{if $ 2^{-\ell -1} < \delta \leq 2^{-\ell}  $, with some $ \ell \in \N $,}
            \\[1ex]
            {[\widetilde{\theta}_1, \theta_{\Gamma, 0}]}, & \mbox{otherwise,}
        \end{array} \right.
        \\[1ex]
        & \hspace{10ex}\mbox{in $ V_1 $, for all $ \delta \geq 0 $,}
    \end{align}
    will be the required sequence that achieves \eqref{Gm-conv100} and \eqref{Gm-conv101}.
\end{proof}

\section{Time-discretization}

In this paper, the solution to (KWC)$_{\varepsilon}$, for any $\varepsilon\ge0$, will be obtained by means of the time-discretization methods. On this basis, we assume (A1)--(A5), and for every $\delta>0$ and $0<\tau \le 1$, we consider the following time-discretization scheme, denoted by (AP)$_{\varepsilon}^{\delta}$.
\smallskip 

(AP)$_{\varepsilon}^{\delta}$:
\begin{align}
  %\MoveEqLeft
 & \begin{aligned}
  &\frac{1}{\tau}\bigl(\bm{\eta}_i-\bm{\eta}_{i-1}+\bm{g}(\bm{\eta}),\bm{\varphi}\bigr)_{H}+\kappa^2\bigl(\nabla\eta_i,\nabla\varphi\bigr)_{[L^2(\Omega)]^N}+\bigl(\alpha'(\eta_i)f_\delta(\nabla\theta_{i}),\varphi\bigr)_{L^2(\Omega)} 
  \\[1ex]
  &\qquad +\bigl(\nabla_\Gamma(\varepsilon\eta_{\Gamma,i}),\nabla_\Gamma(\varepsilon\varphi_\Gamma)\bigr)_{[L^2(\Gamma)]^N}=0, \mbox{ for any $\bm{\varphi}=[\varphi,\varphi_\Gamma]\in V_\varepsilon$,}
 \end{aligned}\label{AP_eta}
\\[1ex]
& \begin{aligned}
  &\frac{1}{\tau}\bigl(A_0(\bm{\eta}_{i-1})(\bm{\theta}_i-\bm{\theta}_{i-1}),\bm{\psi} \bigr)_{H}+\bigl(\delta^2\nabla\theta_i+\alpha(\eta_{i-1})\nabla f_\delta(\nabla\theta_i),\nabla\psi \bigr)_{[L^2(\Omega)]^N}%
    \\[1ex]
  &\qquad +\kappa_\Gamma^2\bigl(\nabla_\Gamma \theta_{\Gamma,i},\nabla_\Gamma \psi_{\Gamma}\bigr)_{[L^2(\Gamma)]^N}= 0, \mbox{ for any $\bm{\psi}=[\psi,\psi_\Gamma]\in V_\delta$,}
  \end{aligned}\label{AP_t^d}
\end{align}

for $i=1,2,3,\ldots$, starting from the initial data: 
\begin{equation}\label{IC.t-Dis}
    \bm{\eta}_0=[\eta_0,\eta_{\Gamma,0}]\in V_\varepsilon\cap \Lambda_0^1, \mbox{ and } \bm{\theta}_0=[\theta_0,\theta_{\Gamma,0}] \in V_\delta \cap\Lambda_{r_0}^{r_1} \bigl( \subset W \cap \Lambda_{r_0}^{r_1} \bigr) \mbox{ in $H$.}
\end{equation}

\begin{defn}[Solution to (AP)$_{\varepsilon}^{\delta}$]\label{Def.t-Dis} Let us fix any $\varepsilon\ge0$ and $\delta>0$. Then, a sequence of quartet $\{[\bm{\eta}_i,\bm{\theta}_i]\}_{i=0}^\infty=\{[\eta_i,\eta_{\Gamma,i}, \theta_{i}, \theta_{\Gamma,i}]\}_{i=0}^\infty \subset [H]^2$, with $\{\bm{\eta}_i\}_{i=0}^\infty=\{[\eta_i,\eta_{\Gamma,i}]\}_{i=0}^\infty\subset H$ and $\{\bm{\theta}_i\}_{i=0}^\infty=\{[\theta_i,\theta_{\Gamma,i}]\}_{i=0}^\infty\subset H$ is called a solution to the time-discretization scheme (AP)$_{\varepsilon}^{\delta}$ iff. $\{[\bm{\eta}_i,\bm{\theta}_i]\}_{i=0}^\infty=\{[\eta_{i},\eta_{\Gamma,i}, \theta_{i}, \theta_{\Gamma,i}]\}_{i=0}^{\infty} \subset V_\varepsilon \times V_\delta$, and this sequence fulfills \eqref{AP_eta} and \eqref{AP_t^d}.
\end{defn}

Now, the objective of this section is to prove the following Theorem.

\begin{thm}[Solvability and energy estimate for (AP)$_{\varepsilon}^{\delta}$]\label{Th.t-Dis}
    There exists a small positive \linebreak constant $0<\tau_*<1$, such that for any $0<\tau<\tau_*$, the time-discretization scheme (AP)$_{\varepsilon}^{\delta}$ admits a unique solution $\{[\bm{\eta}_i,\bm{\theta}_i]\}_{i=0}^\infty=\{[\eta_{i},\eta_{\Gamma,i}, \theta_{i}, \theta_{\Gamma,i}]\}_{i=0}^\infty \subset [H]^2 $, such that:%and the energy-estimate:
    \begin{align}\label{Ener.ineq00}
        & [\bm{\eta}_i,\bm{\theta}_i] \in (V_\varepsilon \cap \Lambda_0^1) \times (V_\delta\cap \Lambda_{r_0}^{r_1}),
    \end{align}
    \vspace{-4ex}

    and
    \begin{align}
  \frac{1}{2\tau}\bigl|\bm{\eta}_i-\bm{\eta}_{i-1}\bigr|_H^2+
        \frac{1}{2\tau}\bigl|A_0(\bm{\eta}_{i -1})^{\frac12}(\bm{\theta}_{i} -\bm{\theta}_{i-1})\bigr|_H^2+\mathscr{F}_{\varepsilon}^{\delta}(\bm{\eta}_i,\bm{\theta}_i) \le \mathscr{F}_{\varepsilon}^{\delta}(\bm{\eta}_{i-1},\bm{\theta}_{i-1}), \label{Ener.ineq}
\end{align}
for $i=1,2,3\dots$, where $\mathscr{F}_{\varepsilon}^{\delta}$ is a relaxed free-energy, defined as:
\begin{equation}\label{reg.Ener}
  \begin{aligned}\relax
[\bm{\eta},\bm{\theta}]&=[\eta,\eta_\Gamma,\theta,\theta_\Gamma] \in [H]^2 \mapsto \mathscr{F}_{\varepsilon}^{\delta}(\bm{\eta},\bm{\theta})=\mathscr{F}_{\varepsilon}^{\delta}(\eta,\eta_\Gamma,\theta,\theta_\Gamma) 
\\
&:=\Psi_{\Omega,\Gamma}(\kappa\eta,\varepsilon\eta_\Gamma)+G(\eta,\eta_\Gamma)+\Phi_\delta(\alpha(\eta);\theta,\theta_\Gamma)\in [0,\infty].
  \end{aligned}
\end{equation}
\end{thm}

For the proof of Theorem \ref{Th.t-Dis}, we prepare some Lemmas.

\begin{lem}\label{Lem04}
  Under the assumptions (A1)--(A5), let us fix a quartet of functions $[\overline{\bm{\eta}}_0,\overline{\bm{\theta}}_0]=[\overline{\eta}_0,\overline{\eta}_{\Gamma,0},\overline{\theta}_0,\overline{\theta}_{\Gamma,0}] \in [H]^2$, with $\overline{\bm{\eta}}_0=[\overline{\eta}_0,\overline{\eta}_{\Gamma,0}]$ and $\overline{\bm{\theta}}_0=[\overline{\theta}_0,\overline{\theta}_{\Gamma,0}]\in H$. Then, the following items hold.
  \begin{description}
      \item[(I)] There exists a small positive constant $ \tau_* \in (0, 1)$, such that if $ \tau \in (0, \tau_*) $, then the following variational identity: 
    \begin{equation}   
    \begin{aligned}
      &\frac{1}{\tau}\bigl(\bm{\eta}-\overline{\bm{\eta}}_0 +\bm{g}(\bm{\eta}),\bm{\varphi}\bigr)_{H}+\kappa^2\bigl(\nabla\eta,\nabla\varphi\bigr)_{[L^2(\Omega)]^N}+\bigl(\alpha'(\eta)f_\delta(\nabla\tilde{\theta}_0),\varphi\bigr)_{L^2(\Omega)}
      \\
      &\qquad +\bigl(\nabla_\Gamma(\varepsilon\eta_{\Gamma}),\nabla_\Gamma(\varepsilon\varphi_\Gamma)\bigr)_{[L^2(\Gamma)]^N} =0, \mbox{ \ for any $\bm{\varphi}=[\varphi,\varphi_\Gamma]\in V_\varepsilon$},
     \end{aligned}\label{Lem04-101}
    \end{equation}
    admits a unique solution $\bm{\eta}=[\eta,\eta_\Gamma] \in V_\varepsilon$.
  
  \item[(II)] The following variational identity:
  \begin{equation}
    \begin{aligned}
    &\frac{1}{\tau}\bigl(A_0(\overline{\bm{\eta}}_0)(\bm{\theta}-\overline{\bm{\theta}}_0),\bm{\psi}\bigr)_{H}+\bigl(\delta^2\nabla\theta+\alpha(\overline{\eta}_{0})\nabla f_\delta(\nabla\theta),\nabla\psi\bigr)_{[L^2(\Omega)]^N}%
      \\
    &\qquad +\kappa_\Gamma^2\bigl(\nabla_\Gamma \theta_{\Gamma},\nabla_\Gamma\psi_{\Gamma}\bigr)_{[L^2(\Gamma)]^N}= 0, \mbox{ \ for any $\bm{\psi}=[\psi,\psi_\Gamma]\in V_\delta$,} \label{Lem04-100}
    \end{aligned}
  \end{equation}
  admits a unique solution $\bm{\theta}=[\theta,\theta_\Gamma]\in V_\delta$.
  \end{description}
  \end{lem}

  \begin{proof}
  First, we show the item (I). Let us note that the variational identity \eqref{Lem04-101} corresponds to the Euler-Lagrange equation for the following proper and l.s.c. potential functional:
  \begin{equation}\label{Lem04-200}
    \begin{aligned}
    \bm{\eta}=[\eta,\eta_\Gamma]& \in H \mapsto \dfrac{1}{2\tau}|\bm{\eta}-\overline{\bm{\eta}}_0|_H^2
+\Psi_{\Omega,\Gamma}(\kappa\eta,\varepsilon\eta_\Gamma) 
    \\
    &+G(\eta,\eta_\varGamma)+\int_\Omega \alpha(\eta)f_\delta(\nabla\overline{\theta}_0)\,dx \in [0,\infty].
  \end{aligned}
\end{equation}
  Here, let us set:
  \begin{equation}
      \tau_* := \frac{1}{2(1+|\bm{g}'|_{[L^\infty(\R)]^2})} = \frac{1}{2(1+|g'|_{L^\infty(\R)} +|g_\Gamma'|_{L^\infty(\R)})} .
      \label{tau_*}
  \end{equation}
  Then, invoking Taylor's theorem, one can see that:
  \begin{align}
      & \frac{1}{2 \tau} |\bm{\eta} -\bm{\overline{\eta}}_0|_H^2 +G(\bm{\eta}) 
      \\
      & \geq \left( \frac{1}{2 \tau} -\frac{1}{2}|\bm{g}'|_{[L^\infty(\R)]^2} \right) |\bm{\eta} -\bm{\overline{\eta}}_0|_H^2 +G(\bm{\overline{\eta}}_0) +\bigl( \bm{g}(\bm{\overline{\eta}}_0), \bm{\eta} -\bm{\overline{\eta}}_0 \bigr)_H
      \\
      & \geq \frac{1}{4 \tau} |\bm{\eta} -\bm{\overline{\eta}}_0|_H^2 +G(\bm{\overline{\eta}}_0) +\bigl( \bm{g}(\bm{\overline{\eta}}_0), \bm{\eta} -\bm{\overline{\eta}}_0 \bigr)_H, \mbox{ for all $ \tau \in (0, \tau_*) $.}
    \label{g-condi}
  \end{align}
      Under $  \tau \in (0, \tau_*) $, the above \eqref{g-condi} implies the coercivity and strict convexity of the potential \eqref{Lem04-200}. Therefore, we can conclude the item (I), by applying the general theory of convex analysis (cf. \cite[Proposition 1.2 in Chapter II]{MR1727362}).

  As well as, the item (II) can be verified, immediately, as a straightforward consequence of the general theory of convex function (cf. \cite[Proposition 1.2 in Chapter II]{MR1727362}), applied to the following proper l.s.c., coercive, and strict convex function: 
  \begin{equation*}
      \bm{\theta} = [\theta,\theta_\Gamma] \in H \mapsto \frac{1}{2\tau}\bigl|A_0(\overline{\bm{\eta}}_0)^{\frac12}(\bm{\theta}-\overline{\bm{\theta}}_0)\bigr|_H^2+ \Phi_{\delta}(\alpha(\overline{\eta}); \bm{\theta}) \in [0,\infty]. 
  \end{equation*}
  Thus, we can conclude the proof of Lemma \ref{Lem04}.
\end{proof}

\begin{lem}\label{Lem02}
    Let $ \tau_* \in (0, 1) $ be the small constant as in \eqref{tau_*}.   For every $\varepsilon\ge0$ and $ \tau \in (0, \tau_*) $, let us assume that $\bm{\eta}^k=[\eta^k,\eta_{\Gamma}^k]\in V_\varepsilon$, $\bm{\eta}_{0}^k=[\eta_{0}^k,\eta_{\Gamma,0}^k]\in V_\varepsilon$, $k=1, 2$, $\tilde{\vartheta} \in V_\delta$, and  
\begin{equation}\label{eta_1}
  \begin{aligned}
      &\frac{1}{\tau}\bigl(\bm{\eta}^{1}-\bm{\eta}_{0}^1+\bm{g}(\bm{\eta}^1),[\bm{\varphi}]^+\bigr)_H+\kappa^2\bigl(\nabla\eta^1,\nabla[\varphi]^+\bigr)_{[L^2(\Omega)]^N}
  \\
      &\qquad +\bigl(\alpha'(\eta^1)f_\delta(\nabla\tilde{\vartheta}),[\varphi]^+\bigr)_{L^2(\Omega)}+\bigl(\nabla_\Gamma(\varepsilon\eta_{\Gamma}^1),\nabla_\Gamma(\varepsilon[\varphi_\Gamma]^+)\bigr)_{[L^2(\Gamma)]^N} \leq 0, 
\end{aligned}
\end{equation}
\begin{equation}\label{eta_2}
  \begin{aligned}
      &\frac{1}{\tau}\bigl(\bm{\eta}^{2}-\bm{\eta}_{0}^2+\bm{g}(\bm{\eta}^2),[\bm{\varphi}]^+ \bigr)_H+\kappa^2\bigl(\nabla\eta^2,\nabla [\varphi]^+ \bigr)_{[L^2(\Omega)]^N}
  \\
      &\qquad +\bigl(\alpha'(\eta^2)f_\delta(\nabla\tilde{\vartheta}), [\varphi]^+ \bigr)_{L^2(\Omega)}+\bigl(\nabla_\Gamma(\varepsilon\eta_{\Gamma}^2),\nabla_\Gamma(\varepsilon [\varphi_\Gamma]^+)\bigr)_{[L^2(\Gamma)]^N}\geq0,
\end{aligned}
\end{equation}
for any $\bm{\varphi}=[\varphi,\varphi_\Gamma]\in V_\varepsilon$. Then, it holds that:
\begin{equation}\label{eta_3}
\bigl|[\bm{\eta}^1-\bm{\eta}^2]^+\bigr|_H^2 \le \bigl| [\bm{\eta}_{0}^1-\bm{\eta}_{0}^2]^+ \bigr|_H^2. %, \mbox{ \ for any $\tau\in(0,\tau_*)$}.
\end{equation}
\end{lem}

\begin{proof}
    Let us take the difference between \eqref{eta_1} and \eqref{eta_2}, and put $ \bm{\varphi} = \bm{\eta}^1-\bm{\eta}^2 $. Then, by (A1), (A3), and Young's inequality, we can observe that:
\begin{align}
&\frac{1}{\tau}\bigl|[\bm{\eta}^1-\bm{\eta}^2]^+\bigr|_H^2+\bigl|\nabla[\eta^1-\eta^2]^+\bigr|_{[L^2(\Omega)]^N}^2+\bigl|\nabla_\Gamma(\varepsilon[\eta_{\Gamma}^1-\eta_{\Gamma}^2]^+)\bigr|_{[L^2(\Gamma)]^N}^2\nonumber
\\
&=\frac{1}{\tau}\bigl(\bm{\eta}_{0}^1-\bm{\eta}_{0}^2,[\bm{\eta}^1-\bm{\eta}^2]^+\bigr)_H-\bigl(\bm{g}(\bm{\eta}^1)-\bm{g}(\bm{\eta}^2),[\bm{\eta}^1-\bm{\eta}^2]^+\bigr)_H\nonumber
\\
&\qquad -\int_\Omega(\alpha'(\eta^1)-\alpha'(\eta^2))[\eta^1-\eta^2]^+f_\delta(\nabla\tilde{\vartheta})dx\nonumber
\\
&\le \left(\frac{1}{4\tau}+|\bm{g}'|_{[L^\infty(\R)]^2}\right)\bigl|[\bm{\eta}^1-\bm{\eta}^2]^+\bigr|_H^2+\frac{1}{\tau}\bigl|[\bm{\eta}_{0}^1-\bm{\eta}_{0}^2]^+\bigr|_H^2.
    \label{Lem02-500}
\end{align}
  Here, noting that:
\begin{align}
    & \frac{3}{4 \tau} -|\bm{g}'|_{[L^\infty(\R)]^2} \geq \frac{1}{4 \tau}, \mbox{ for all $ \tau \in (0, \tau_*) $,}
\end{align}
the inequality \eqref{Lem02-500} is obtained as a consequence of \eqref{eta_3}.
\end{proof}

\begin{lem}\label{Lem03}
Let $\bm{\tilde{\eta}}=[\tilde{\eta},\tilde{\eta}_\Gamma] \in V_\varepsilon$ and $\bm{\vartheta}_{0,k}=[\vartheta_{0}^k,\vartheta_{\Gamma,0}^k]\in V_\delta$, $k=1,2$ be fixed pairs of functions, and let us assume that:
\begin{align}
    & \frac{1}{\tau} \bigl( A_0(\bm{\tilde{\eta}})(\bm{\vartheta}^1 -\bm{\vartheta}_0^1), [\bm{\psi}]^+ \bigr)_H +\bigl( \delta^2 \nabla \vartheta^1 +\alpha(\tilde{\eta}) \nabla f_\delta(\nabla \vartheta^1), \nabla [\psi]^+ \bigr)_{[L^2(\Omega)]^N} 
    \nonumber
    \\
    & \hspace{14ex} +\kappa_\Gamma^2 \bigl( \nabla_\Gamma \vartheta_\Gamma^1, \nabla_\Gamma [\psi_\Gamma]^+ \bigr)_{[L^2(\Gamma)]^N} \leq 0,
    \label{theta_1}
\end{align}
    \vspace{-3ex}
\begin{align}
    & \frac{1}{\tau} \bigl( A_0(\bm{\tilde{\eta}})(\bm{\vartheta}^2 -\bm{\vartheta}_0^2), [\bm{\psi}]^+ \bigr)_H +\bigl( \delta^2 \nabla \vartheta^2 +\alpha(\tilde{\eta}) \nabla f_\delta(\nabla \vartheta^2), \nabla [\psi]^+ \bigr)_{[L^2(\Omega)]^N} 
    \nonumber
    \\
    & \hspace{14ex} +\kappa_\Gamma^2 \bigl( \nabla_\Gamma \vartheta_\Gamma^2, \nabla_\Gamma [\psi_\Gamma]^+ \bigr)_{[L^2(\Gamma)]^N} \geq 0,
    \label{theta_2}
\end{align}
for any $ \bm{\psi} = [\psi, \psi_\Gamma] \in V_\delta $. Then, it holds that:
\begin{equation}\label{theta_3}
\bigl|A_0(\bm{\tilde{\eta}})^{\frac12}[\bm{\vartheta}^1-\bm{\vartheta}^2]^+\bigr|_H^2 \le \bigl|A_0(\bm{\tilde{\eta}})^{\frac12}[\bm{\vartheta}_{0}^1-\bm{\vartheta}_{0}^2]^+\bigr|_H^2.
\end{equation}
\end{lem}
\begin{proof}
    This lemma is immediately verified by taking the difference between \eqref{theta_1} and \eqref{theta_2}, putting $ \bm{\psi} = \bm{\vartheta}^1 -\bm{\vartheta}^2 $, and using \eqref{f_d} and (A3). 
\end{proof}

\noindent
\emph{Proof of Theorem \ref{Th.t-Dis}.} Let $\tau_* \in (0, 1)$ be the small constant as in \eqref{tau_*}. On this basis, let us fix any time-step-size $\tau \in (0,\tau_*)$. Then, since the value of constant $\tau_*$ is independent of the time-index $i \in \N \cup \{0\}$, the solution $\{[\bm{\eta}_i,\bm{\theta}_i]\}_{i=0}^\infty=\{[\eta_{i},\eta_{\Gamma,i},\theta_{i},\theta_{\Gamma,i}]\}_{i=1}^\infty \subset V_\varepsilon \times V_\delta$ to the time-discretization scheme (AP)$_\varepsilon^\delta$ is obtained by applying Lemma \ref{Lem04} to the equations \eqref{AP_eta} and \eqref{AP_t^d}, inductively. 

Next, the properties $ \{ \bm{\eta}_i \}_{i = 1}^\infty \subset \Lambda_0^1 $ and $ \{ \bm{\theta}_i \}_{i = 1}^\infty \subset \Lambda_{r_0}^{r_1} $ will be verified by applying Lemma \ref{Lem02} and \ref{Lem03} to the cases when:
\begin{align}
    & \begin{cases}
        \hspace{-3ex}
        \parbox{10cm}{
            \vspace{-2.5ex}
            \begin{itemize}
            \item $ \bm{\eta}^1 = \bm{\eta}_i $, $ \bm{\eta}_0^1 = \bm{\eta}_{i -1} $, $ \bm{\eta}^2 = \bm{\eta}_0^2 = [1, 1] $, and $ \tilde{\vartheta} = \theta_i $,
            \item $ \bm{\eta}^1 = \bm{\eta}_0^1 = [0, 0] $, $ \bm{\eta}^2 = \bm{\eta}_i $, $ \bm{\eta}_0^2 = \bm{\eta}_{i -1} $, and $ \tilde{\vartheta} = \theta_i $,
        \end{itemize}
        \vspace{-2ex}
    }
    \end{cases}
    i = 1, 2, 3, \dots,
\end{align}
and
\begin{align}
    & \begin{cases}
        \hspace{-3ex}
        \parbox{10.5cm}{
            \vspace{-2.5ex}
            \begin{itemize}
                \item $ \bm{\vartheta}^1 = \bm{\theta}_i $, $ \bm{\vartheta}_0^1 = \bm{\theta}_{i -1} $, $ \bm{\vartheta}^2 = \bm{\theta}_0^2 = [r_1, r_1] $, and $ \bm{\tilde{\eta}} = \bm{\eta}_{i -1} $,
                \item $ \bm{\vartheta}^1 = \bm{\theta}_0^1 = [r_0, r_0] $, $ \bm{\vartheta}^2 = \bm{\theta}_i $, $ \bm{\vartheta}_0^2 = \bm{\theta}_{i -1} $, and $ \bm{\tilde{\eta}} = \bm{\eta}_{i -1} $,
        \end{itemize}
        \vspace{-2ex}
    }
    \end{cases}
    i = 1, 2, 3, \dots,
\end{align}
respectively, and using (A1)--(A5). 

Finally, we verify the energy inequality \eqref{Ener.ineq}. To this end, for every $i \in \N$, let us put $ \bm{\varphi} = \bm{\eta}_{i} -\bm{\eta}_{i -1} $ in \eqref{AP_eta}. Then, invoking (A3), (A4), Taylor's theorem, and
\begin{align}
    & \frac{1}{\tau} -\frac{1}{2} |\bm{g}'|_{[L^\infty(\R)]^2} \geq \frac{3}{4 \tau} \geq \frac{1}{2 \tau}, \mbox{ for all $ \tau \in (0, \tau_*) $,}
\end{align}
we infer that:
\begin{align}
    & \frac{1}{2 \tau} \bigl| \bm{\eta}_{i} -\bm{\eta}_{i -1} \bigr|_H^2 +G(\bm{\eta}_{i}) -G(\bm{\eta}_{i -1})   +\int_\Omega \alpha(\eta_{i}) f_\delta(\nabla \theta_{i}) \, dx
    \\
    & \qquad -\int_\Omega \alpha(\eta_{i -1}) f_{\delta}(\nabla \theta_{i}) \, dx \leq 0, \mbox{ for $ i = 1, 2, 3, \dots $.}
    \label{KS0424-01}
\end{align}
Also, by putting $ \bm{\psi} = \bm{\theta}_{i} -\bm{\theta}_{i -1} $ in \eqref{AP_t^d} and using Young's inequality, one can see from \eqref{DefOfPhi_delta}, \eqref{f_d}, (A2), (A3), and (A4) that:
\begin{align}
    & \frac{1}{\tau} \bigl| A_0(\bm{\eta}_{i -1})^{\frac{1}{2}}(\bm{\theta}_{i} -\bm{\theta}_{i -1}) \bigr|_H^2 
    \\
    & \qquad +\int_\Omega \alpha(\eta_{i -1}) f_\delta(\nabla \theta_i) \, dx -\int_\Omega  \alpha(\eta_{i -1}) f_\delta(\nabla \theta_{i -1}) \, dx 
    \\
    & \qquad +\frac{\delta^2}{2} \int_\Omega |\nabla \theta_{i}|^2 \, dx -\frac{\delta^2}{2} \int_\Omega |\nabla \theta_{i -1}|^2 \, dx 
    \\
    & \qquad +\frac{\kappa_\Gamma}{2} \int_\Gamma |\nabla_\Gamma \theta_{i}|^2 \, d \Gamma -\frac{\kappa_\Gamma}{2} \int_\Gamma |\nabla_\Gamma \theta_{i -1}|^2 \, d \Gamma  \leq 0, \mbox{ for $ i = 1, 2, 3, \dots $.}
    \label{KS0424-02}
\end{align}
The energy inequality \eqref{Ener.ineq} will be obtained by taking the sum of \eqref{KS0424-01} and \eqref{KS0424-02}. 
\bigskip

Thus, we complete the proof of Theorem \ref{Th.t-Dis}.
 \hfill$\Box$

\section{Proof of Main Theorem}
This section is devoted to the proof of Main Theorem \ref{mTh.Sol}. %of this paper. 

%\subsection{Proof of Main Theorem \ref{mTh.Sol}}
Let us fix a constant $ \varepsilon \geq 0 $. Let $0< \tau^* <1$ be the constant given in Theorem \ref{Th.t-Dis}. Let $ \{ \bm{\tilde{\theta}}_0^\delta \}_{\delta \in (0, 1]} = \{ [\tilde{\theta}_0^\delta, \tilde{\theta}_{\Gamma, 0}^\delta] \}_{\delta \in (0, 1]} \subset V_1 $ be as in Corollary \ref{CorOfGm-conv01}.
Additionally, for every $ \delta > 0 $ and $0<\tau < \tau^*$, let us denote by $ \{ [\bm{\eta}_i^\delta, \bm{\theta}_i^\delta] \}_{i = 1}^\infty = \{ [\eta_{i}^\delta, \eta_{\Gamma,\tau,i}^\delta, \theta_{i}^\delta, \theta_{\Gamma,\tau,i}^\delta] \}_{i = 1}^\infty \subset L^2(0, T; [H]^2) $ with $ \{ \bm{\eta}_i^\delta \}_{i = 1}^\infty = \{ [\eta_i^\delta, \eta_{\Gamma, i}^\delta] \}_{i = 1}^\infty \subset L^2(0, T; H) $ and $ \{ \bm{\theta}_i^\delta \}_{i = 1}^\infty = \{ [\theta_i^\delta, \theta_{\Gamma, i}^\delta] \}_{i = 1}^\infty \subset L^2(0, T; H) $ the solution to (AP)$_{\varepsilon, \tau}^\delta$, subject to the following initial condition:
\begin{equation*}
    [\bm{\eta}_0^\delta, \bm{\theta}_0^\delta] = [\eta_0^\delta, \eta_{\Gamma, 0}^\delta, \theta_0^\delta, \theta_{\Gamma, 0}^\delta] = [\bm{\eta}_0, \bm{\tilde{\theta}}_0^\delta] = [\eta_0, \eta_{\Gamma, 0}, \tilde{\theta}_0^\delta, \tilde{\theta}_{\Gamma, 0}^\delta] \mbox{ in $ [H]^2 $.}
\end{equation*}
On this basis, we let:
\begin{equation*}
t_i:=i\tau, \mbox{ for $i=0,1,2,\dots$,}
\end{equation*}
and we define the following time-interpolations:
\begin{align}
    \left\{ \rule{0pt}{64pt} \right. & 
    \\[-128pt]
    & \bigl[ [\bm{\overline{\eta}}]_\tau^\delta(t), [\bm{\overline{\theta}}]_\tau^\delta(t) \bigr] = \bigl[ [\overline{\eta}]_\tau^\delta(t), [\overline{\eta_\Gamma}]_\tau^\delta(t), [\overline{\theta}]_\tau^\delta(t), [\overline{\theta_\Gamma}]_\tau^\delta(t) \bigr] 
    \\
    & \qquad : = [\eta_{i}^\delta, \eta_{\Gamma, i}^\delta, \theta_{i}^\delta, \theta_{\Gamma,i}^\delta] = [\bm{\eta}_i^\delta, \bm{\theta}_i^\delta], \mbox{ if $t_{i-1} \leq t < t_i$,}
    \\
    & \bigl[ [\bm{\underline{\eta}}]_\tau^\delta(t), [\bm{\underline{\theta}}]_\tau^\delta(t) \bigr] = \bigl[ [\underline{\eta}]_\tau^\delta(t), [\underline{\eta_\Gamma}]_\tau^\delta(t), [\underline{\theta}]_\tau^\delta(t), [\underline{\theta_\Gamma}]_\tau^\delta(t) \bigr] 
    \\
    & \qquad : = [\eta_{i -1}^\delta, \eta_{\Gamma, i -1}^\delta, \theta_{i -1}^\delta, \theta_{\Gamma, i -1}^\delta] = [\bm{\eta}_{i -1}^\delta, \bm{\theta}_{i -1}^\delta], \mbox{ if $t_{i-1} \leq t < t_i$,}
    \label{mTh.sol00a}
\\
    & \bigl[ [\bm{\eta}]_\tau^\delta(t), [\bm{\theta}]_\tau^\delta(t) \bigr] = \bigl[ [\eta]_\tau^\delta(t), [\eta_\Gamma]_\tau^\delta(t), [\theta]_\tau^\delta(t), [\theta_\Gamma]_\tau^\delta(t)]
    \\
    & \qquad : =\frac{t-t_{i-1}}{\tau} \bigl[ [\bm{\overline{\eta}}]_\tau^\delta(t), [\bm{\overline{\theta}}]_\tau^\delta(t) \bigr] + \frac{t_i-t}{\tau}\bigl[ [\bm{\underline{\eta}}]_\tau^\delta(t), [\bm{\underline{\theta}}]_\tau^\delta(t) \bigr], \mbox{ if $t_{i-1} \le t <t_i$,}
    \\
    & \hspace{20ex} \mbox{in $ V_\varepsilon \times V_\delta $, for $ i = 1, 2, 3, \dots $.}
\end{align}
Then, from \eqref{Ener.ineq00} and \eqref{Ener.ineq} in Theorem \ref{Th.t-Dis}, we can see that:
\begin{equation}\label{mTh.sol01}
    \begin{array}{c}
        \bigl\{ \bigl[ [\bm{\overline{\eta}}]_\tau^\delta(t), [\bm{\overline{\theta}}]_\tau^\delta(t) \bigr], \bigl[ [\bm{\underline{\eta}}]_\tau^\delta(t), [\bm{\underline{\theta}}]_\tau^\delta(t) \bigr], \bigl[ [\bm{\eta}]_\tau^\delta(t), [\bm{\theta}]_\tau^\delta(t) \bigr] \bigr\} \subset \Lambda_0^1 \times \Lambda_{r_0}^{r_1}, 
        \\[1ex]
        \mbox{ a.e. $ t \in (0, T) $,}
    \end{array}
\end{equation}
and
\begin{align}
    \frac12 & \int_s^t \bigl|\partial_t [\bm{\eta}]_\tau^\delta(\varsigma)\bigr|_H^2\,d \varsigma +\frac12\int_s^t \bigl|A_0([\,\underline{\eta}\,]_\tau^\delta(\varsigma))^{\frac12} \partial_t [\bm{\theta}]_\tau^\delta(\varsigma) \bigr|_H^2\,d \varsigma 
    \\[1ex]
    & \qquad +\mathscr{F}_\varepsilon^\delta([\bm{\overline{\eta}}]_\tau^\delta(t), [\bm{\overline{\theta}}]_\tau^\delta(t)) \le \mathscr{F}_\varepsilon^\delta([\bm{\underline{\eta}}]_\tau^\delta(s), [\bm{\underline{\theta}}]_\tau^\delta(s)) 
    \\[1ex]
    & \leq \mathscr{F}_\varepsilon^\delta(\bm{\eta}_0^\delta, \bm{\theta}_0^\delta) \le \Psi_{\Omega, \Gamma}(\bm{\eta}_0) +G(\bm{\eta}_0) +\sup_{\delta \in (0, 1]} \Phi_\varepsilon^\delta(\alpha(\eta_0); \bm{\tilde{\theta}}_0^\delta) 
    \\
    & =: F_0^* <  \infty, \mbox{ for all $ 0 \leq s \leq t \leq T $.}
    \label{mTh.sol02}
\end{align}

Now, from \eqref{mTh.sol01} and \eqref{mTh.sol02}, we can see the following properties:
\begin{description}
    \item[{\boldmath ($\flat1$-a)}]the class $ \left\{ \begin{array}{l|l}
                [\bm{\overline{\eta}}]_\tau^\delta, [\bm{\underline{\eta}}]_\tau^\delta  & \parbox{4cm}{$0 < \delta \le 1$, $\ 0<\tau < \tau^*$}
        \end{array} \right\}$ is bounded in $L^\infty(0, T; H)$,  
    \item[{\boldmath ($\flat1$-b)}]the class $ \left\{ \begin{array}{l|l}
            [\bm{\eta}]_\tau^\delta, [\bm{\theta}]_\tau^\delta  & \parbox{4cm}{$0 < \delta \le 1$, $\ 0<\tau < \tau^*$}
    \end{array} \right\}$ is bounded in $W_\mathrm{loc}^{1, 2}([0, T); H)$,
\item[{\boldmath ($\flat1$-c)}]the function of time $ t \in [0, T] \mapsto \mathscr{F}_\varepsilon^\delta([\bm{\overline{\eta}}]_\tau^\delta(t), [\bm{\overline{\theta}}]_\tau^\delta(t)) \in [0, \infty) $ and $ t \in [0, T] \mapsto \mathscr{F}_\varepsilon^\delta([\bm{\underline{\eta}}]_\tau^\delta(t), [\bm{\underline{\theta}}]_\tau^\delta(t)) \in [0, \infty) $ are nonincreasing for every $ 0 < \delta \leq 1 $ and $ 0 < \tau < \tau^* $, and  $ \bigl\{ \begin{array}{l|l}
        \mathscr{F}_\varepsilon^\delta(\bm{{\eta}}_0^\delta, \bm{{\theta}}_0^\delta)  & \parbox{1.75cm}{$0 < \delta \le 1$}
        \end{array} \bigr\}$ is bounded, and hence, the class $ \bigl\{ \mathscr{F}_\varepsilon^\delta([\bm{\overline{\eta}}]_\tau^\delta, [\bm{\overline{\theta}}]_\tau^\delta), $ \linebreak $ \mathscr{F}_\varepsilon^\delta([\bm{\underline{\eta}}]_\tau^\delta, [\bm{\underline{\theta}}]_\tau^\delta) \, \bigl| \, 0 < \delta \le 1, 0 < \tau < \tau^* \bigr\} $ is bounded in $ BV(0, T) $.
\end{description}
By virtue of \eqref{mTh.sol01}, ($\flat1$-a), ($\flat1$-b), and ($\flat1$-c), we can apply the compactness theories of Aubin's type \cite[Corollary 4]{MR0916688}, Arzer\'a--Ascoli \cite[Theorem 1.3.1]{MR932730} and Alaoglu--Bourbaki--Kakutani \cite[Theorem 1.2.5]{MR932730}, and can obtain sequences $\{\delta_n\}_{n=1}^\infty \subset (0,1)$, $\{\tau_n\}_{n=1}^\infty \subset (0,\tau_*)$, and a quartet of functions $ [ \bm{\eta}, \bm{\theta} ] = [\eta, \eta_\Gamma, \theta, \theta_\Gamma] \in L^2(0, T; [H]^2)$, with $ \bm{\eta} = [\eta, \eta_\Gamma] \in L^2(0, T; H) $ and $ \bm{\theta} = [\theta, \theta_\Gamma] \in L^2(0, T; H) $, such that 
\begin{align}
    \delta_n \downarrow 0 \mbox{ and } \tau_n \downarrow 0, \mbox{ as $n \to \infty$,}
    \label{mTh.sol10a}
\end{align}
\begin{align}
    & \begin{cases}
        \bm{\eta} = [\eta, \eta_\Gamma] \in W^{1, 2}(0, T; H) \cap L^\infty(0, T; V_\varepsilon),
        \\[0.5ex]
        \bm{\theta} = [\theta, \theta_\Gamma] \in W^{1, 2}(0, T; H), ~ [|D \theta(\cdot)|]_{\theta_\Gamma(\cdot)} \in L^\infty(0, T),
        \\[0.5ex]
        [\bm{\eta}(t), \bm{\theta}(t)] \subset \Lambda_0^1 \times \Lambda_{r_0}^{r_1}, \mbox{ for any $ t \in (0, T) $,}
    \end{cases}
    \label{mTh.sol10b}
\end{align}
\begin{align}
    & \bm{\eta}_n =[\eta_n, \eta_{\Gamma, n}] := [\bm{\eta}]_{\tau_n}^{\delta_n} = \bigl[ [\eta]_{\tau_n}^{\delta_n}, [\eta_{\Gamma}]_{\tau_n}^{\delta_n} \bigr] \to \bm{\eta} = [\eta,\eta_\Gamma] \mbox{ in $C([0, T]; H)$,}\nonumber
\\
    & \qquad \mbox{weakly in $W^{1,2}(0, T; H)$, weakly-$*$ in $L^\infty(0, T; V_\varepsilon)$,}
    \\
    & \qquad \mbox{and weakly-$ * $ in $ L^\infty(Q) \times L^\infty(\Sigma) $, as $n \to \infty$,}
    \label{mTh.sol10c}
\end{align}
\begin{align}
    & \begin{cases}
        \bm{\overline{\eta}}_n = [\overline{\eta}_n, \overline{\eta}_{\Gamma, n}] := [\bm{\overline{\eta}}]_{\tau_n}^{\delta_n} = \bigl[ [\overline{\eta}]_{\tau_n}^{\delta_n}, [\overline{\eta_{\Gamma}}]_{\tau_n}^{\delta_n} \bigr] \to \bm{\eta} = [\eta,\eta_\Gamma]
        \\
        \bm{\underline{\eta}}_{\,n} =[\underline{\eta}_{\,n}, \underline{\eta}_{\,\Gamma, n}] := [\bm{\underline{\eta}}]_{\tau_n}^{\delta_n} = \bigl[ [\underline{\eta}]_{\tau_n}^{\delta_n}, [\underline{\eta_{\Gamma}}]_{\tau_n}^{\delta_n} \bigr] \to \bm{\eta} = [\eta,\eta_\Gamma]
    \end{cases}
    \\
        & \qquad \mbox{in $ L^\infty(0, T; H) $, weakly-$*$ in $L^\infty(0, T; V_\varepsilon)$,}
    \\
    & \qquad \mbox{and weakly-$ * $ in $ L^\infty(Q) \times L^\infty(\Sigma) $, as $n \to \infty$,}
    \label{mTh.sol10d}
\end{align}
\begin{align}
    & \bm{\theta}_n =[\theta_n, \theta_{\Gamma, n}] := [\bm{\theta}]_{\tau_n}^{\delta_n} = \bigl[ [\theta]_{\tau_n}^{\delta_n}, [\theta_{\Gamma}]_{\tau_n}^{\delta_n} \bigr] \to \bm{\theta} = [\theta, \theta_\Gamma]\nonumber
\\
    & \qquad  \mbox{in $C([0, T]; L^2(\Omega) \times H^{\frac{1}{2}}(\Gamma))$, weakly in $W^{1, 2}(0, T; H)$,}
    \\
    & \qquad \mbox{weakly-$*$ in $ L^\infty(0, T; L^2(\Omega) \times H^1(\Gamma)) $,}
    \\
    & \qquad \mbox{and weakly-$*$ in $ L^\infty(Q) \times L^\infty(\Sigma) $, as $n \to \infty$,}
    \label{mTh.sol10e}
\end{align}
\begin{align}
    & \begin{cases}
        \bm{\overline{\theta}}_n = [\overline{\theta}_n, \overline{\theta}_{\Gamma, n}] := [\bm{\overline{\theta}}]_{\tau_n}^{\delta_n} = \bigl[ [\overline{\theta}]_{\tau_n}^{\delta_n}, [\overline{\theta_{\Gamma}}]_{\tau_n}^{\delta_n} \bigr] \to \bm{\theta} = [\theta,\theta_\Gamma]
        \\
        \bm{\underline{\theta}}_{\,n} =[\underline{\theta}_{\,n}, \underline{\theta}_{\,\Gamma, n}] := [\bm{\underline{\theta}}]_{\tau_n}^{\delta_n} = \bigl[ [\underline{\theta}]_{\tau_n}^{\delta_n}, [\underline{\theta_{\Gamma}}]_{\tau_n}^{\delta_n} \bigr] \to \bm{\theta} = [\theta,\theta_\Gamma]
    \end{cases}
    \\
    & \qquad \mbox{in $ L^\infty(0, T; H) $, weakly-$*$ in $L^\infty(0, T; V_\varepsilon)$,}
    \\
    & \qquad \mbox{and weakly-$ * $ in $ L^\infty(Q) \times L^\infty(\Sigma) $, as $n \to \infty$,}
    \label{mTh.sol10f}
\end{align}
%\begin{align}
%    & \begin{cases}
%        \bm{\eta}_n(t) \to \bm{\eta}(t) \mbox{ in $ H $, weakly in $ V_\varepsilon $,}
%        \\
%        \quad \mbox{and weakly-$*$ in $ L^\infty(\Omega) \times L^\infty(\Gamma) $,}
%    \end{cases}
%    \mbox{for any $ t \in [0, T] $, as $ n \to \infty $,}
%    \label{mTh.sol10g}
%\end{align}
%\begin{align}
%    & \begin{cases}
%        \bm{\theta}_n(t) \to \bm{\theta}(t) \mbox{ in $ L^2(\Omega) \times H^{\frac{1}{2}}(\Gamma) $,}
%        \\
%        \theta_n(t) \to \theta(t) \mbox{ weakly-$*$ in $ BV(\Omega) $,}
%        \\
%        \theta_{\Gamma, n}(t) \to \theta_\Gamma(t) \mbox{ weakly in $ H^1(\Gamma) $,}
%    \end{cases}
%    \mbox{for any $ t \in [0, T] $, as $ n \to \infty $,}
%    \label{mTh.sol10h}
%\end{align}
and in particular,
\begin{align}
    & [\bm{\eta}(0), \bm{\theta}(0)]= \lim_{n \to \infty} [\bm{\eta}_n(0), \bm{\theta}_n(0)] = \lim_{n \to \infty} [ \bm{\eta}_0, \bm{\tilde{\theta}}_0^{\delta_n}] = [\bm{\eta}_0, \bm{\theta}_0] \mbox{ in $[H]^2$.}
    \label{mTh.sol10i}
\end{align}
Moreover, on account of ($\flat 1$-c) and Helly's selection theorem, we can find a bounded and nonincreasing function $ \mathcal{J}_* : [0, T] \longrightarrow [0, \infty) $, such that
\begin{align}
    & \begin{cases}
        \mathscr{F}_\varepsilon^{\delta_n}(\bm{\underline{\eta}}_{\,n}, \bm{\underline{\theta}}_{\,n}) \to \mathcal{J}_* \mbox{ weakly-$*$ in $ BV(0, T) $, and}
        \\
        \hspace{19.5ex}\mbox{weakly-$*$ in $ L^\infty(0, T) $,}
        \\[1ex]
        \mathscr{F}_\varepsilon^{\delta_n}(\bm{\underline{\eta}}_{\,n}(t), \bm{\underline{\theta}}_{\,n}(t)) \to \mathcal{J}_*(t), \mbox{ for any $ t \in [0, T] $,} 
    \end{cases}
    \mbox{as $ n \to \infty $,}
    \label{ken0411-20}
\end{align}
under a subsequence if necessary. 

Next, we will show that the limit $ [\bm{\eta}, \bm{\theta}] = [\eta, \eta_\Gamma, \theta, \theta_\Gamma] \in L^2(0, T; [H]^2) $ satisfies the variational inequalities in (S1) and (S2). Let us fix any open interval $ I \subset (0, T) $, and let us set $ L^\infty(I; \Omega, \Gamma) := L^\infty(I \times \Omega) \times L^\infty(I \times \Gamma) $. 
Besides, we take any $ \bm{\widetilde{\psi}} = [\widetilde{\psi}, \widetilde{\psi}_\Gamma] \in L^\infty(I; V_\varepsilon) \cap L^\infty(I; \Omega, \Gamma) $, put $ \bm{\psi} = \bm{\widetilde{\psi}}(t) $ in \eqref{AP_eta}, and integrate the both sides over $ I $. Then, with \eqref{AP_eta} and \eqref{mTh.sol00a} in mind, we can see that:
\begin{align}
    & \int_I \bigl( \partial_t \bm{\eta}_n(t), \bm{\widetilde{\psi}}(t) \bigl)_H \, dt  
    \\
    & \quad +\int_I \bigl( g(\overline{\eta}_n(t)), \widetilde{\psi}(t) \bigr)_{L^2(\Omega)} \, dt +\int_I \bigl( g_\Gamma(\overline{\eta}_{\Gamma, n}(t)), \widetilde{\psi}_\Gamma(t) \bigr)_{L^2(\Gamma)} \, dt
    \\
    & \quad +\kappa^2 \int_I \bigl( \nabla \overline{\eta}_n(t), \nabla \widetilde{\psi}(t) \bigr)_{[L^2(\Omega)]^N} \, dt +\int_I \bigl( \nabla_\Gamma (\varepsilon \overline{\eta}_{\Gamma, n}(t)), \nabla_\Gamma (\varepsilon \widetilde{\psi}_{\Gamma}(t)) \bigr)_{[L^2(\Gamma)]^N} \, dt
    \\
    & \quad +\int_I \widetilde{\psi}(t) \alpha'(\overline{\eta}_n(t)) f_\delta(\nabla \overline{\theta}_n(t)) \, dx dt = 0, 
    \label{ken1stEq}
    \\
    & \mbox{for all $ \bm{\widetilde{\psi}} = [\widetilde{\psi}, \widetilde{\psi}_\Gamma] \in L^\infty(I; V_\varepsilon) \cap L^\infty(I; \Omega, \Gamma) $, and $ n = 1, 2, 3, \dots $.}
\end{align} 
Similarly, from \eqref{AP_t^d} and \eqref{mTh.sol00a}, one can observe that:
\begin{align}
    & \int_I \bigl( A_0(\underline{\eta}_{\,n}(t)) \partial_t \bm{\theta}_n(t), \bm{\overline{\theta}}_n(t) -\bm{\widetilde{\varphi}}(t)  \bigr)_H \, dt +\Phi_{\delta_n}^I(\alpha(\underline{\eta}_{\,n}); \bm{\overline{\theta}}_n) \leq \Phi_{\delta_n}^I(\alpha(\underline{\eta}_{\,n}); \bm{\widetilde{\varphi}}),
    \label{ken2ndEq}
    \\
    & \hspace{14ex}\mbox{for all $ \bm{\widetilde{\varphi}} = [\widetilde{\varphi}, \widetilde{\varphi}_\Gamma] \in L^2(I; V_\varepsilon) $, and $ n = 1, 2, 3, \dots $.}
\end{align} 
Additionally, by \eqref{mTh.sol10a}, \eqref{mTh.sol10b}, and \eqref{mTh.sol10d}, we can apply Key-Lemma \ref{keyLem05}, to see that:
\begin{description}
    \item[{\boldmath ($\flat1$-d)}]$ \{ \Phi_{\delta_n}^I(\alpha(\bm{\underline{\eta}_{\,n}});{}\cdot{}) \}_{n = 1}^\infty $ $ \Gamma $-converges to $ \Phi_0^I(\alpha(\bm{\underline{\eta}}); {}\cdot{}) $ on $ L^2(I; H) $, as $ n \to \infty $.
\end{description}

Now, let us take any $ \bm{\varphi} = [\varphi, \varphi_\Gamma] \in D(\Phi_0^I(\alpha(\bm{\eta});{}\cdot{})) $, with an approximating sequence $ \{ \bm{\widehat{\varphi}}_n \}_{n = 1}^\infty = \{ [\widehat{\varphi}_n, \widehat{\varphi}_{\Gamma, n}] \}_{n = 1}^\infty \subset L^2(I; V_1) $, such that:
\begin{align}
    & \bm{\widehat{\varphi}}_n \to \bm{\varphi} \mbox{ in $ L^2(I; H) $, and } \Phi_{\delta_n}^I(\alpha({\underline{\eta}_{\,n}}); \bm{\widehat{\varphi}}_n) \to \Phi_0^I(\alpha(\eta); \bm{\varphi}), \mbox{ as $ n \to \infty $.}
\end{align}
\noeqref{mTh.sol10b, mTh.sol10c, mTh.sol10d, mTh.sol10e, mTh.sol10f} 
Then, by \eqref{mTh.sol10a}--\eqref{mTh.sol10i}, ($\flat 1$-a), ($\flat 1$-b),  ($\flat 1$-c), and ($\flat 1$-d), putting $ \bm{\widetilde{\varphi}} = \bm{\widehat{\varphi}}_n  $ in \eqref{ken2ndEq}, and  letting $ n \to \infty $ yields that:
\begin{align}
    & \int_I \bigl( A_0(\eta(t)) \partial_t \bm{\theta}(t), \bm{\theta}(t) -\bm{{\varphi}}(t)  \bigr)_H \, dt +\Phi_{0}^I(\alpha(\eta); \bm{\theta}) 
    \\
    & \qquad \leq \lim_{n \to \infty} \int_I \bigl( A_0(\underline{\eta}_{\,n}(t)) \partial_t \bm{\theta}_n(t), \bm{\overline{\theta}}_n(t) -\bm{\widehat{\varphi}}_n(t)  \bigr)_H \, dt +\varliminf_{n \to \infty} \Phi_{\delta_n}^I(\alpha(\underline{\eta}_{\,n}); \bm{\overline{\theta}}_n) 
    \\
    & \qquad \leq \varlimsup_{n \to \infty} \left( \int_I \bigl( A_0(\underline{\eta}_{\,n}(t)) \partial_t \bm{\theta}_n(t), \bm{\overline{\theta}}_n(t) -\bm{\widehat{\varphi}}_n(t)  \bigr)_H \, dt +\Phi_{\delta_n}^I(\alpha(\underline{\eta}_{\,n}); \bm{\overline{\theta}}_n) \right)
    \\
    & \qquad = \int_I \bigl( A_0(\eta(t)) \partial_t \bm{\theta}(t), \bm{\theta}(t) -\bm{\varphi}(t)  \bigr)_H \, dt +\varlimsup_{n \to \infty} \Phi_{\delta_n}^I(\alpha(\underline{\eta}_{\,n}); \bm{\overline{\theta}}_n) 
    \\
    & \qquad \leq \lim_{n \to \infty} \Phi_{\delta_n}^I(\alpha(\underline{\eta}_{\,n}); \bm{\widehat{\varphi}}_n) = \Phi_0^I(\alpha(\eta); \bm{\varphi}), 
    \label{varIneq-S2}
    \\
    & 
    \hspace{16ex} \mbox{for any $ \bm{\varphi} = [\varphi, \varphi_\Gamma] \in D(\Phi_0^I(\alpha(\eta));{}\cdot{}) $.}
\end{align}
Since the choices of open interval $ I \subset (0, T) $ and test function $ \bm{\varphi} = [\varphi, \varphi_\Gamma] \in D(\Phi_0^I(\alpha(\eta);{}\cdot{})) $ are arbitrary, it is deduced from \eqref{varIneq-S2} that $ [\bm{\eta}, \bm{\theta}] = [\eta, \eta_\Gamma, \theta, \theta_\Gamma] $ solves the variational inequality in (S2). 
\bigskip

Next, letting $ \bm{\varphi} = \bm{\theta} $ in \eqref{varIneq-S2}, we have:
\begin{align}
    & \lim_{n \to \infty} \Phi_{\delta_n}^I(\alpha({\underline{\eta}}_n); \bm{\overline{\theta}}_n) = \Phi_0^I(\alpha(\eta); \bm{\theta}). 
    \label{ken0410-01}
\end{align}
Meanwhile, by \eqref{mTh.sol10a}--\eqref{mTh.sol10i}, and Key-Lemma \ref{keyLem03}, it is easily seen from that:
\begin{align}
    \left\{ \rule{0pt}{70pt} \right. & 
    \\[-140pt]
    & \varliminf_{n \to \infty} \int_I \int_\Omega \alpha(\underline{\eta}_{\,n}(t)) f_{\delta_n}(\nabla \overline{\theta}_n(t)) \, dx dt 
    %\\
    %& \qquad 
    \geq \varliminf_{n \to \infty} \int_I \int_{\overline{\Omega}} d \bigl[ \alpha(\underline{\eta}_{\,n}(t)) |D \overline{\theta}_n(t)| \bigr]_{\overline{\theta}_{\Gamma, n}(t)} dt
    \\
    & \quad \geq \int_I \int_{\overline{\Omega}}  d \bigl[ \alpha(\eta(t)) |D \theta(t)| \bigr]_{\theta_{\Gamma}(t)} dt
    \label{ken0410-02}
    \\[1ex]
    & \varliminf_{n \to \infty} \frac{\kappa_\Gamma^2}{2} \int_I \int_\Gamma |\nabla_\Gamma \overline{\theta}_{\Gamma, n}(t)|^2 \, d \Gamma dt \geq \frac{\kappa_\Gamma^2}{2} \int_I \int_\Gamma |\nabla_\Gamma \theta_\Gamma(t)|^2 \, d \Gamma dt,
    \\[1ex]
    & \varliminf_{n \to \infty} \frac{1}{2} \int_I \int_\Omega |\nabla (\delta_n \overline{\theta})(t)|^2 \, dx dt  \geq 0.
\end{align}
Taking into account \eqref{ken0410-01}, \eqref{ken0410-02}, Fact \ref{Fact0}, and the uniform convexity of $L^2$-based norm, one can see that:
\begin{align}
    & \left| \int_I \int_{\overline{\Omega}}  d \bigl[ \alpha(\underline{\eta}_n(t)) |D \overline{\theta}_n(t)| \bigr] dt -\int_I \int_{\overline{\Omega}} d \bigl[ \alpha(\underline{\eta}_n(t)) |D \overline{\theta}_n(t)| \bigr] dt \right|
    \\[1ex]
    & \qquad  = \left| \int_I \int_\Omega \alpha(\underline{\eta}_{\,n}) f_{\delta_n}(\nabla \overline{\theta}_n) \, dx dt -\int_I \int_{\overline{\Omega}} d \bigl[ \alpha(\underline{\eta}_n(t)) |D \overline{\theta}_n(t)| \bigr] dt \right|
    %\\
    %& \qquad \qquad 
    +\delta_n \bigl| \alpha(\underline{\eta}_{\,n}) \bigr|_{L^1(I \times \Omega)} 
    \label{ken0410-03}
    \\[1ex]
    & \qquad \to 0, \mbox{ as $ n \to \infty $,}
\end{align}
and 
\begin{align}
    & \begin{cases}
        \delta_n \overline{\theta}_n \to 0 \mbox{ in $ L^2(I; H^1(\Omega)) $, }
        \\[0.5ex]
        \overline{\theta}_{\Gamma, n} \to \theta_\Gamma \mbox{ in $ L^2(I; H^1(\Gamma)) $,}
    \end{cases}
    \mbox{as $ n \to \infty $.}
    \label{ken0410-04}
\end{align}
In addition, owing to (A4), we can apply Key-Lemma \ref{keyLem04} by replacing:
\begin{align}
    & \left\{ \hspace{-3ex} \parbox{10.25cm}{
        \vspace{-2ex}
        \begin{itemize}
            \item $ \beta $ by $ \alpha(\eta) $, $ \{ \beta_n \}_{n = 1}^\infty $ by $ \bigl\{ \alpha(\underline{\eta}_{\,n}) \bigr\}_{n = 1}^\infty $,
                \vspace{-1ex}
            \item $ \varrho $ by $ \widetilde{\psi} \alpha'(\eta) $, $ \{ \varrho_n \}_{n = 1}^\infty $ by $ \bigl\{ \widetilde{\psi} \alpha'(\overline{\eta}_{n}) \bigr\}_{n = 1}^\infty $, \\
                with $ \widetilde{\psi} \in L^\infty(I; H^1(\Omega)) \cap L^\infty(I \times \Omega) $ as in \eqref{ken1stEq},
                \vspace{-1ex}
            \item $ \{ [\theta_n, \theta_{\Gamma, n}] \}_{n = 1}^\infty $ (in Key-Lemma \ref{keyLem04}) by $ \bigl\{ [\overline{\theta}_n, \overline{\theta}_{\Gamma, n}] \bigr\}_{n = 1}^\infty $,
        \end{itemize}
        \vspace{-2ex}
    } \right.
    ~~~~\mbox{respectively,}
\end{align}
and we can compute that:
\begin{align}
    & \left| \int_I \int_\Omega \widetilde{\psi}(t) \alpha'(\overline{\eta}_n(t)) f_{\delta_n}(\nabla \overline{\theta}_n(t)) \, dx dt -\int_I \int_{\overline{\Omega}} d \bigl[ \widetilde{\psi}(t) \alpha'({\eta}(t)) |D {\theta}(t)| \bigr]_{{\theta}_{\Gamma}(t)} dt \right|
    \\
    & \leq \left| \int_I \int_{\overline{\Omega}} d \bigl[ \widetilde{\psi}(t) \alpha'(\overline{\eta}_n(t)) |D \overline{\theta}_n(t)| \bigr]_{\overline{\theta}_{\Gamma, n}(t)} dt -\int_I \int_{\overline{\Omega}} d \bigl[ \widetilde{\psi}(t) \alpha'({\eta}(t)) |D {\theta}(t)| \bigr]_{{\theta}_{\Gamma}(t)} dt \right|
    \\
    & \quad +\delta_n \bigl| \widetilde{\psi} \alpha'(\overline{\eta}_n) \bigr|_{L^\infty(I \times \Omega)} \to 0, \mbox{ as $ n \to \infty $,}
    \\
    & \mbox{for any $ \widetilde{\psi} \in L^\infty(I; H^1(\Omega)) \cap L^\infty(I \times \Omega) $.}
    \label{ken0410-05}
\end{align}
Invoking \eqref{mTh.sol10a}--\eqref{mTh.sol10f}, and \eqref{ken0410-05}, and letting $ n \to \infty $ in \eqref{ken1stEq}, it is inferred that:
\begin{align}
    & \int_I \bigl( \partial_t \bm{\eta}(t), \bm{\widetilde{\psi}}(t) \bigr)_H \, dt
     +\int_I \bigl( g(\eta(t)), \widetilde{\psi}(t) \bigr)_{L^2(\Omega)} \, dt +\int_I \bigl( g_\Gamma(\eta_\Gamma(t)), \widetilde{\psi}_\Gamma(t) \bigr)_{L^2(\Gamma)} \, dt
    \\
    & \quad +\kappa^2 \int_I \bigl( \nabla \eta(t), \nabla \widetilde{\psi}(t) \bigr)_{[L^2(\Omega)]^N} \, dt +\int_I \bigl( \nabla_\Gamma (\varepsilon \eta_\Gamma(t)), \nabla_\Gamma (\varepsilon \widetilde{\psi}_\Gamma(t)) \bigr)_{[L^2(\Gamma)]^N} \, dt
    \\
    & \quad +\int_I \int_{\overline{\Omega}} d \bigl[ \widetilde{\psi}(t) \alpha'(\eta(t)) |D \theta| \bigr]_{\theta_\Gamma(t)} \, dt = 0,
    \label{ken0410-10}
    \\
    & \mbox{for any $ \bm{\widetilde{\psi}} = [\widetilde{\psi}, \widetilde{\psi}_{\Gamma}] \in L^\infty(I; V_\varepsilon) \cap \Lambda_I^\infty $.}
\end{align}
Since the choices of $ I \subset (0, T) $ and $ \bm{\widetilde{\psi}} = [\widetilde{\psi}, \widetilde{\psi}_\Gamma] \in L^\infty(I; V_\varepsilon) $ are arbitrary, the computation \eqref{ken0410-10} leads to the variational inequality in (S1). 
\bigskip

Finally, we verify the items (A) and (B).  To this end, we first put $ \bm{\widetilde{\psi}} = \bm{\overline{\eta}}_n -\bm{\eta} $ in \eqref{ken1stEq}, and let $ n \to \infty $ with use of \eqref{f_d}, \eqref{mTh.sol10a}--\eqref{mTh.sol10f}, \eqref{ken0410-05}, Key-Lemma \ref{keyLem03} and Young's inequality. Then, one can observe that:
\begin{align}
    & \frac{\kappa^2}{2} \int_I \int_\Omega |\nabla \eta(t)| \, dx dt +\frac{1}{2} \int_I \int_\Gamma |\nabla_\Gamma (\varepsilon \eta_\Gamma(t))|^2 \, d \Gamma dt
    \\
    & \quad \leq  \varliminf_{n \to \infty} \left( \frac{\kappa^2}{2} \int_I \int_\Omega |\nabla \overline{\eta}_n(t)| \, dx dt +\frac{1}{2} \int_I \int_\Gamma |\nabla_{\Gamma} (\varepsilon \overline{\eta}_{\Gamma, n}(t))|^2 \, d \Gamma dt \right)
    \\
    & \quad \leq  \varlimsup_{n \to \infty} \left( \frac{\kappa^2}{2} \int_I \int_\Omega |\nabla \overline{\eta}_n(t)| \, dx dt +\frac{1}{2} \int_I \int_\Gamma |\nabla_{\Gamma} (\varepsilon \overline{\eta}_{\Gamma, n}(t))|^2 \, d \Gamma dt \right)
\end{align}
\begin{align}
%    \\
    & \quad \leq \frac{\kappa^2}{2} \int_I \int_\Omega |\nabla \eta(t)|^2 \, dx dt +\frac{1}{2} \int_I \int_\Gamma |\nabla_\Gamma \eta_\Gamma(t)|^2 \, d\Gamma dt
    \\
    & \qquad  -\lim_{n \to \infty} \int_I \bigl( \partial_t \bm{\eta}_n(t), (\bm{\overline{\eta}}_n -\bm{\eta})(t) \bigr)_H \, dt 
    \\
    & \qquad -\lim_{n \to \infty} \int_I \bigl( \bm{g}(\bm{\overline{\eta}}_n(t)), (\bm{\overline{\eta}}_n -\bm{\eta})(t) \bigr)_{H} \, dt 
    \\
    & \qquad -\varliminf_{n \to \infty} \int_I \int_\Omega \overline{\eta}_n(t) \alpha'({\overline{\eta}}_n(t)) |\nabla \overline{\theta}_n(t)| \, dx dt 
    \\
    & \qquad +\lim_{n \to \infty} \int_I \int_\Omega \eta(t) \alpha'(\overline{\eta}_n(t)) f_{\delta_n}(\nabla \overline{\theta}_n(t)) \, dx dt
    \\
    & \quad \leq \frac{\kappa^2}{2} \int_I \int_\Omega |\nabla \eta(t)|^2 \, dx dt +\frac{1}{2} \int_I \int_\Gamma |\nabla_\Gamma (\varepsilon \eta_\Gamma(t))|^2 \, d \Gamma dt.
    \label{ken0411-01}
\end{align}
From \eqref{mTh.sol02}, \eqref{mTh.sol10d}, \eqref{ken0410-01}, \eqref{ken0411-01}, and the uniform convexity of $ L^2 $-based norm, it follows that:
\begin{align}
    & \begin{cases}
        \bm{\overline{\eta}}_n \to \bm{\eta} \mbox{ in $ L^2(I; V_\varepsilon) $,}
        \\[1ex]
        \ds \int_I \mathscr{F}_\varepsilon^{\delta_n}(\bm{\overline{\eta}}_n(t), \bm{\overline{\theta}}_n(t)) \, dt \to \int_I \mathscr{F}_\varepsilon(\bm{\eta}(t), \bm{\theta}(t)) \, dt,
    \end{cases}
    \mbox{as $ n \to \infty $,}
\end{align}
and
\begin{align}
    & \left| \int_I \mathscr{F}_\varepsilon^{\delta_n}(\bm{\underline{\eta}}_n(t), \bm{\underline{\theta}}_n(t)) \, dt -\int_I \mathscr{F}_\varepsilon(\bm{\eta}(t), \bm{\theta}(t)) \, dt \right|
    \\
    & \leq \tau_n %\mathscr{F}_\varepsilon^{\delta_n}(\bm{\eta}_0^{\delta_n}, \bm{\theta}_0^{\delta_n}) 
    F_0^* +\left| \int_I \mathscr{F}_\varepsilon^{\delta_n}\bm{\overline{\eta}}_n(t), \bm{\overline{\theta}}_n(t)) \, dt -\int_I \mathscr{F}_\varepsilon(\bm{\eta}(t), \bm{\theta}(t)) \, dt \right|
    \\
    & \to 0, \mbox{ as $ n \to \infty $.}
    \label{ken0411-10}
\end{align}
Meanwhile, for any open set $ A \subset I $, it is deduced from \eqref{f_d}, \eqref{mTh.sol02}--\eqref{mTh.sol10f}, \eqref{ken0410-05}, and Key-Lemma \ref{keyLem03} that:
\begin{align}
    \int_A & \mathscr{F}_\varepsilon(\bm{\eta}(t), \bm{\theta}(t)) \, dt = \sum_{\tilde{I} \in \mathfrak{I}_A} \int_{\tilde{I}} \mathscr{F}_\varepsilon(\bm{\eta}(t), \bm{\theta}(t)) \, dt \leq \sum_{\tilde{I} \in \mathfrak{I}_A} \varliminf_{n \to \infty} \int_{\tilde{I}} \mathscr{F}_\varepsilon(\bm{\underline{\eta}}_n(t), \bm{\underline{\theta}}_n(t)) \, dt
    \\
    & \leq \sum_{\tilde{I} \in \mathfrak{I}_A} \varliminf_{n \to \infty} \int_{\tilde{I}} \mathscr{F}_\varepsilon^{\delta_n}(\bm{\underline{\eta}}_n(t), \bm{\underline{\theta}}_n(t)) \, dt
    \leq \varliminf_{n  \to \infty} \sum_{\tilde{I} \in \mathfrak{I}_A} \int_{\tilde{I}} \mathscr{F}_\varepsilon^{\delta_n}(\bm{\underline{\eta}}_n(t), \bm{\underline{\theta}}_n(t)) \, dt 
    \\
    & = \varliminf_{n \to \infty} \int_{A} \mathscr{F}_\varepsilon^{\delta_n}(\bm{\underline{\eta}}_n(t), \bm{\underline{\theta}}_n(t)) \, dt,
    \label{ken0411-11}
\end{align}
where $ \mathfrak{I}_A $ is the (at most countable) class of pairwise-disjoint open intervals satisfying $ A = \bigcup_{\tilde{I} \in \mathfrak{I}_A} \tilde{I} $. So, as an application of \cite[Proposition 1.80]{MR1857292} to \eqref{ken0411-10} and \eqref{ken0411-11}, we have:
\begin{align}
    & \mathscr{F}_\varepsilon^{\delta_n}(\bm{\underline{\eta}}_n, \bm{\underline{\theta}}_n) \to \mathscr{F}_\varepsilon(\bm{\eta}, \bm{\theta}) \mbox{ weakly-$*$ in $ \mathscr{M}_\mathrm{loc}(I) $, as $ n \to \infty $.} 
    \label{ken0411-21}
\end{align}

Now, by virtue of \eqref{mTh.sol01}--\eqref{mTh.sol10f}, \eqref{ken0411-20}, and \eqref{ken0411-21}, we will obtain that:
\begin{align}
    & \frac{1}{2} \int_s^t |\partial_t \bm{\eta}(\varsigma)|_H^2 \, d \varsigma +\frac{1}{2} \int_s^t |A_0(\eta(\varsigma))\partial_t \bm{\theta}(\varsigma)|_H^2 \, d \varsigma +\mathcal{J}_*(t) \leq \mathcal{J}_*(s),
    \\
    & \hspace{22ex}\mbox{for all $ 0 \leq s \leq t \leq T $,}
    \label{ken0411-30}
\end{align}
and 
\begin{align}
    & \mathcal{J}_*(t) = \mathscr{F}_\varepsilon(\bm{\eta}(t), \bm{\theta}(t)), \mbox{ for a.e. $ t \in (0, T) $.}
    \label{ken0411-31}
\end{align}
In addition, due to \eqref{ken0411-31}, Key-Lemma \ref{keyLem03-00} and the lower semi-continuity of $ L^2 $-based norm, we can say that:
\begin{description}
    \item[{\boldmath ($\flat1$-e)}]the function $ t \in [0, T] \mapsto \mathscr{F}_\varepsilon(\bm{\eta}(t), \bm{\theta}(t)) \in [0, \infty) $ is nonincreasing and lower semi-continuous. 
\end{description}
The items (A) and (B) will be concluded as consequences of \eqref{ken0411-30}, \eqref{ken0411-31}, and ($\flat1$-e). 
\bigskip

Thus, we complete the proof of Main Theorem \ref{mTh.Sol}. \hfill \qed 

\section{Appendix}
\subsection{Notations in BV-theory}
Let $ d \in \N $ be a fixed dimension, and let $U\subset \R^d$ be an open set. We denote by $ \mathscr{M}(U) $ (resp. $ \mathscr{M}_{\rm loc}(U) $) the space of all finite Radon measures (resp. the space of all Radon measures) on $ U $. Recall that the space $ \mathscr{M}(U) $ (resp. $ \mathscr{M}_{\rm loc}(U) $) is the dual of the Banach space $ C_0(U) $ (resp. dual of the locally convex space $ C_{\rm c}(U) $), for any open set $ U \subset \mathbb{R}^d $.

For a function $ u \in L_{\rm loc}^1(U) $ and a point $ x_0 \in U $, we denote by $  [\,\widetilde{u}\,](x_0) $ the \emph{approximate limit} of $ u $ at $ x_0 $, and as well as, we denote by $ [\widetilde{u_+}](x_0) $ (resp. $ [\widetilde{u_-}](x_0) $) the \emph{approximate limit-sup (resp. approximate limit-inf)} of $ u $ at $ x_0 $ (cf. \cite[Section 3.6]{MR1857292}). 
Besides, we set
$$
S_u := \left\{ \begin{array}{l|l}
x \in U &
[\widetilde{u_+}](x) \ne [\widetilde{u_-}](x), ~ \mbox{i.e. $ [\widetilde{u_+}](x) > [\widetilde{u_-}](x) $}
\end{array} \right\}.
$$
%\end{notn}
   
    A function $ u \in L^1(U) $ (resp. $ u \in L_{\rm loc}^1(U) $)  is called a function of bounded variation, or a BV-function, (resp. a function of locally bounded variation or a BV$\empty_{\rm loc}$-function) on $ U $, iff. its distributional differential $ D u $ is a finite Radon measure on $ U $ (resp. a Radon measure on $ U $), namely $ D u \in \mathscr[{M}(U)]^d $ (resp. $ D u \in [\mathscr{M}_{\rm loc}(U)]^d $).

We denote by $ BV(U) $ (resp. $ BV_{\rm loc}(U) $) the space of all BV-functions (resp. all BV$\empty_{\rm loc}$-functions) on $ U $. For any $ u \in BV(U) $, the Radon measure $ D u $ is called the variation measure of $ u $, and its  total variation $ |Du| $ is called the total variation measure of $ u $. Additionally, the value $|Du|(U)$, for any $u \in BV(U)$, can be calculated as follows:
\begin{equation*}
|Du|(U) = \sup \left\{ \begin{array}{l|l}
    \ds \int_{U} u \ {\rm div} \,\varphi \, dy & \varphi \in [C_{\rm c}^{1}(U)]^d \ \ \mbox{and}\ \ |\varphi| \le 1\ \mbox{on}\ U
\end{array}
\right\}.
\end{equation*}
The space $BV(U)$ is a Banach space, endowed with the following norm:
\begin{equation*}
|u|_{BV(U)} := |u|_{L^{1}(U)} + |D u|(U),\ \ \mbox{for any}\ u\in BV(U).
\end{equation*}
Also, $ BV(U) $ is a complete metric space, endowed with the following distance:
$$
[u, v] \in BV(U)^2 \mapsto |u -v|_{L^1(U)} +\left| \int_U |Du| -\int_U |Dv| \right|.
$$
The topology provided by this distance is called the \em strict topology \em of $ BV(U) $.

Now, by Radon-Nikod\'{y}m's theorem \cite[Theorem 1.28]{MR1857292}, the measure $ Du $ for $ u \in BV(U) $ is decomposed in the absolutely continuous part $ D^a u $ for $ \mathcal{L}^{N} $ and the singular part $ D^s u $, i.e.
$$
Du = D^{a}u +D^{s}u \mbox{ \ in $ \mathscr{M}(U) $.}
$$
Furthermore, (cf. \cite{MR3288271,MR1857292}), this decomposition is precisely expressed as follows:
$$
D^{a} u = \nabla u \, \mathcal{L}^d \, \lfloor \, (U \setminus S_u) \mbox{ and } D^s u = ([\widetilde{u_+}] -[\widetilde{u_-}]) \nu_u \mathcal{H}^{d -1} \, \lfloor \, S_u +D^{s} u \, \lfloor \, (U \setminus S_u).
$$
In this context, $ \nabla u \in L^1(U) $ denotes the Radon-Nikod\'{y}m density of $ D^a u  $ for $ \mathcal{L}^{d} $. The part $ ([\widetilde{u_+}] -[\widetilde{u_-}]) \nu_u \mathcal{H}^{{d} -1} \, \lfloor \, S_u $ is called the \em jump part \em of $ Du $, and it provides an exact expression of the variation at the discontinuities of $ u $. The part $ D^s u \, \lfloor \, (U \setminus S_u) $ is called the \em Cantor part \em of $ Du $ and it is denoted by $D^c u$. We denote by SBV the space of special functions of bounded variation; i.e. those BV functions such that $D^c u=0$. We note that a function defined as
$$
x \in U \mapsto [u]^*(x) := \left\{ \begin{array}{ll}
[\,\widetilde{u}\,](x), & \mbox{if $ x \in U \setminus S_u $,}
\\[1ex]
\ds \frac{[\widetilde{u_+}](x) +[\widetilde{u_-}](x)}{2}, & \mbox{if $ x \in S_u $, and $  [\widetilde{u_+}](x) $ and $ [\widetilde{u_-}](x) $ exist,}
\\[2ex]
0, & \mbox{otherwise,}
\end{array} \right.
$$
provides a \em precise representative \em of $ u \in BV(U) $, $ \mathcal{H}^{d -1} $-a.e. on $ U $.

    In particular, if $ U $ is bounded and the boundary  $\partial U$ is Lipschitz, then the space $BV(U)$ is continuously embedded into $L^{d/(d -1)}(U)$ and compactly embedded into $L^{q}(U)$ for any $1 \le q < N/(N-1)$ (cf. \cite[Corollary 3.49]{MR1857292} or \cite[Theorems 10.1.3 and 10.1.4]{MR3288271}). Also, there exists a bounded linear operator $ \mathfrak{tr}_{{\partial U}} : BV(U) \mapsto L^1(\partial U) $, called the \em trace operator, \em such that $ \mathfrak{tr}_{{\partial U}} \varphi = \varphi|_{\partial U} $ on $ \partial U $ for any $ \varphi \in C^1(\overline{U}) $. The value of trace $ \mathfrak{tr}_{{\partial U}} u \in L^1(\partial U) $ is often abbreviated by $ u \trace{\partial U} $. The trace operator $ \mathfrak{tr}_{{\partial U}} : BV(U) \mapsto L^1(\partial U) $ is continuous with respect to the strict topology of $ BV(U) $. Namely,  $ \mathfrak{tr}_{{\partial U}} u_n \to \mathfrak{tr}_{{\partial U}} u $ in $ L^1(\partial U) $ as $ n \to \infty $, if $ u \in BV(U) $, $ \{ u_n \}_{n = 1}^\infty \subset BV(U) $ and $ u_n \to u $ in the strict topology of $ BV(U) $ as $ n \to \infty $. %, by identifying $ \mathfrak{tr}_{{\partial U}} u $ as the extension of $ u \in BV(U) $ onto $ \partial U $.
Additionally, if $1 \le r < \infty$, then the space $C^{\infty}(\overline{U})$ is dense in $BV(U) \cap L^{r}(U)$ for the {\em intermediate convergence} (cf. \cite[Definition 10.1.3. and Theorem 10.1.2]{MR3288271}), i.e. for any $u \in BV(U) \cap L^{r}(U)$, there exists a sequence $\{u_{n} \}_{n = 1}^\infty \subset C^{\infty}(\overline{U})$ such that $u_{n} \to u$ in $L^{r}(U)$ and $\int_{U}|\nabla u_{n}|dx \to |Du|(U)$ as $n \to \infty$.
%\end{notn}
\bigskip

Finally, we recall the notion of \emph{pairing measure,} established by Anzellotti \cite{MR0750538}.
Let $ \Omega \subset \R^N $ with $ N \in \N $ and $ \Gamma := \partial \Omega $ be the bounded domain and its boundary, as in $(\omega1)$ and $(\omega2)$, respectively. Also, we let:
\begin{align*}
    & {L}_\mathrm{div}^q(\Omega; \R^N) := \left\{ \begin{array}{l|l}
        \omega \in L^2(\Omega; \R^N)
        & 
        \mathrm{div} \, \omega \in L^q(\Omega)
    \end{array} \right\}, ~\mbox{for any $ 1 \leq q \leq \infty $.}
\end{align*}

For any $ \omega \in {L}_\mathrm{div}^2(\Omega; \R^N) \cap L^\infty(\Omega; \R^N) $ and any $ u \in BV(\Omega) \cap L^2(\Omega) $, we denote by $ (\omega, Du) $ a finite Radon measure on $ \Omega $, called \emph{pairing measure,} which is defined as a unique extension of the following distribution:
\begin{align}\label{ap.pairMeas00}
    & \varphi \in C_\mathrm{c}^\infty(\Omega) \mapsto -\int_\Omega u \, \mathrm{div} \, (\varphi \omega) \, dx \in \R.
\end{align}
\begin{rem}\label{Rem.pairMeas}
It is known (cf. \cite{MR0750538}) that $ (\omega, Du) \in \mathcal{M}(\Omega) $ is absolutely continuous for the total variation measure $ |Du| $ of $ u $, 
\begin{subequations}\label{ap.pairMeas01}
\begin{align}\label{ap.pairMeas01a}
    & \bigl| {\ts \frac{(\omega, Du)}{|Du|}} \bigr| \leq |\omega|_{L^\infty(\Omega; \R^N)}, \mbox{ $ |Du| $-a.e. in $ \Omega $.}
\end{align}
    and
    \begin{align}
        & (\omega, Du)^a = \omega \cdot \nabla u \mathcal{L}^N \mbox{ and } (\omega, Du)^s = {\ts \frac{(\omega, Du)}{|(\omega, Du)|}} |Du^s| \mbox{ in $ \mathcal{M}(\Omega) $,} 
    \end{align}
\end{subequations}
    where $ (\omega, Du)^a $ and $ (\omega, Du)^s $ are the absolutely continuous part of $ (\omega, Du) $ for $ \mathcal{L}^N $ and the singular part $ (\omega, Du)^s $, respectively. 
\end{rem}

In addition, we define a bounded linear operator $ [(\cdot) \cdot n_\Gamma] : {L}_\mathrm{div}^2(\Omega; \R^N) \longrightarrow H^{-\frac{1}{2}}(\Gamma) $, as a unique extension of the mapping $ \omega \in C^1(\overline{\Omega}; \R^N) \mapsto \omega \cdot n_\Gamma \in C(\Gamma) $. 

Now, Referring to \cite[Theorem 1.9]{MR0750538}, we can easily see the following fact and remark.
\begin{fact}%[(Gauss--Green type formula)]
    \label{Fact10}
    The restriction $ [(\cdot) \cdot n_\Gamma]|_{L^\infty(\Omega; \R^N)} $ forms a non-expansive linear operator from $ {L}_\mathrm{div}^2(\Omega; \R^N) \cap L^\infty(\Omega; \R^N) $ into $ L^\infty(\Gamma) $, such that:
    \begin{align}
        & \bigl|[\omega \cdot n_\Gamma]\bigr| \leq |\omega|_{L^\infty(\Omega; \R^N)} \mbox{ a.e. on $ \Gamma $,}
    \end{align}  
    and
    \begin{align}
        & \int_\Gamma [\omega \cdot n_\Gamma] u \trace{\Gamma} \, d \Gamma  = \int_\Omega (\mathrm{div} \, \omega) u \, dx +\int_\Omega d(\omega, Du)
        \\
        \mbox{for}~ & \mbox{all $ \omega \in {L}_\mathrm{div}^2(\Omega; \R^N) \cap L^\infty(\Omega; \R^N) $ and $ u \in BV(\Omega) \cap L^2(\Omega) $.} 
    \end{align}
\end{fact}
\begin{rem}\label{Prop.pairMeas}
    Let us assume:
    \begin{align}\label{ap.KS00}
        \beta \in H^1(\Omega) \cap L^\infty(\Omega), ~ \omega \in {L}_\mathrm{div}^2(\Omega; \R^N), \mbox{ and } u \in BV(\Omega) \cap L^2(\Omega).
    \end{align}
    Then, we can easily seen that:
    \begin{align}
        \left\{ \rule{0pt}{20pt} \right. & 
        \\[-40pt]
        & (\beta \omega, Du) = \beta^* (\omega, Du) \mbox{ in $ \mathcal{M}(\Omega) $,}
        \label{ap.A}
        \\[1ex]
        & [(\beta \omega) \cdot n_\Gamma] = \beta\trace{\Gamma} [\omega \cdot n_\Gamma] \mbox{ in $ L^\infty(\Gamma) $.}
        \label{ap.bB}
    \end{align}
    In fact, on account of \eqref{ap.pairMeas00} and Fact \ref{Fact10}, we can see that the both $ (\beta \omega, Du)|_{C_\mathrm{c}^\infty(\Omega)} $ and $ \beta^*(\omega, Du)|_{C_\mathrm{c}^\infty(\Omega)} $ coincides with the following distribution:
    \begin{align}
        & \varphi \in C_\mathrm{c}^\infty(\Omega) \mapsto -\int_\Omega u \, \mathrm{div} \, (\varphi \beta \omega) \, dx \in \R.
    \end{align}
    This immediately leads to the identity \eqref{ap.A}. Meanwhile, with \eqref{ap.pairMeas01} and Fact \ref{Fact10} in mind, we can compute that:
    \begin{align}\label{apKS02}
        \bigl< [(\beta \omega) \cdot n_\Gamma], \varphi_\Gamma \bigr>_{H^{\frac{1}{2}}(\Gamma)} ~& = \int_\Omega \bigl( \mathrm{div} \, (\beta \omega) \bigr) [\varphi_\Gamma]^\mathrm{ex} \, dx +\int_\Omega (\beta \omega) \cdot \nabla [\varphi_\Gamma]^\mathrm{ex} \, dx
        \nonumber
        \\
        & = \int_\Omega \mathrm{div} \, \omega \bigl( [\alpha\trace{\Gamma} \varphi_\Gamma]^\mathrm{ex} \bigr) \, dx +\int_\Omega \omega \cdot \nabla [\beta\trace{\Gamma} \varphi_\Gamma]^\mathrm{ex} \, dx
        \nonumber
        \\
        & = \int_\Gamma [\omega \cdot n_\Gamma] (\beta\trace{\Gamma} \varphi_\Gamma) \, d \Gamma = \bigl< \beta\trace{\Gamma} [\omega \cdot n_\Gamma], \varphi_\Gamma \bigr>_{H^{\frac{1}{2}}(\Gamma),}
        \\
        & \hspace{-17ex}\mbox{for any $ \varphi_\Gamma \in H^{\frac{1}{2}}(\Gamma) $ with an extension $ [\varphi_\Gamma]^\mathrm{ex} \in H^1(\Omega) $. }
        \nonumber
    \end{align}
    Since $ \beta\trace{\Gamma}[\omega \cdot n_\Gamma] \in L^\infty(\Gamma) = L^1(\Gamma)^* $, and the embedding $ H^{\frac{1}{2}}(\Gamma) $ is dense, the identity \eqref{ap.bB} will be a straightforward consequence of the above \eqref{apKS02}. 
\end{rem}

\subsection{Extensions of functions}
    Let $ d \in \N $ be a fixed dimension. Let $ \mu $ be a positive measure on $ \R^d $, and let $ B \subset \R^d $ be a $ \mu $-measurable Borel set. For a $ \mu $-measurable function $ u : B \rightarrow \R $, we denote by $ [u]^{\rm ex} $ an extension of $ u $ over $ \R^d $, i.e. $ [u]^{\rm ex} : \R^{d} \rightarrow \R $ is a measurable function such that $ [u]^{\rm ex} = u $ $ \mu $-a.e. in $ B $. In general, the choices of extensions are not necessarily unique.

    If $ 1 < d \in \N $, and an open set $ U \subset \R^d $ has a compact $ C^1 $-boundary $ \partial U $, then the following fact is known as a representative theory of extension.
\begin{fact}\label{Fact1}
(cf. \cite[Proposition 3.21]{MR1857292})There exists a bounded linear operator $ \mathfrak{ex}_{U} : BV(U) \rightarrow BV(\R^N) $, such that:
\begin{itemize}
\item $ \mathfrak{ex}_{U} $ maps any function $ u \in BV(U) $ to an extension $ [u]^{\rm ex} \in BV(\R^N) $;
\item for any $ 1 \leq q < \infty $,  $ \mathfrak{ex}_{U} ({W^{1, q}(U)}) \subset W^{1, q}(\R^N) $, and the restriction $ \mathfrak{ex}_{U}  |_{W^{1, q}(U)} : W^{1, q}(U) \rightarrow W^{1, q}(\R^N) $ forms a bounded and linear operator.
\end{itemize}
\end{fact}

\begin{rem}[Harmonic extension]\label{Rem.ex_hm}
    Let $ \Omega \subset \R^N $ be the spatial domain satisfying $(\omega1)$ and $ (\omega2) $. We denote by $ \mathfrak{hm}_{\Gamma} $ the operator of harmonic extension $ \mathfrak{hm}_{{\Gamma}} : H^{\frac{1}{2}}(\Gamma) \to H^1(\Omega) $, which maps any $ \varrho \in H^{\frac{1}{2}}(\Gamma) $ to a (unique) minimizer $ [\varrho]^{\rm hm} \in H^1(\Omega) $ of the following proper l.s.c. and convex function on $ L^2(\Omega) $:
\begin{equation*}
    z \in {L^2(\Omega)} \mapsto \left\{ \begin{array}{ll}
\multicolumn{2}{l}{\ds \frac{1}{2} \int_\Omega |\nabla z|^2 \, dx, \mbox{ if $ z \in H^1(\Omega) $ and $ z = \varrho $ in $ H^{\frac{1}{2}}(\Gamma) $,}}
\\[2ex]
\infty, & \mbox{otherwise.}
\end{array} \right.
\end{equation*}
    Indeed, by the general theories of linear PDEs (cf. \cite[Chapter 3]{MR0241822}, \cite[Chapter 8]{MR0244627} and \cite[Theorem 8.3 in Chapter 2]{MR0350177}), we can check that $ \mathfrak{hm}_{{\Gamma}} H^{\frac{1}{2}}(\Gamma) \subset H^1(\Omega) \cap H_\mathrm{loc}^2(\Omega) $, and:
\begin{equation}\label{harmonic01}
    (\mathit{\Delta} [\varrho]^{\rm hm}, \varphi)_{L^2(\Omega)} = (\nabla [\varrho]^{\rm hm}, \nabla \varphi)_{[L^2(\Omega)]^N} = 0, \mbox{ for all $ \varrho \in H^{\frac{1}{2}}(\Gamma) $ and {$ \varphi \in C_\mathrm{c}^1(\Omega) $},}
\end{equation}
    i.e. the extension $ [\varrho]^{\rm hm} $ solves the harmonic equation $ \mathit{\Delta} [\varrho]^{\rm hm} = 0 $ on $ \Omega $ with the Dirichlet boundary data $ \varrho \in {H^{\frac{1}{2}}(\Gamma)} $. So, by $(\omega1)$, $(\omega2)$, and the general theory of elliptic boundary value problem (cf. \cite{MR0241822,MR0244627}), the harmonic extension $ \mathfrak{hm}_\Gamma $ forms an homeomorphism between $ H^{m -\frac{1}{2}}(\Gamma) \longrightarrow H^{m}(\Omega) $, for $ m = 1, 2, 3, \dots $. 
Additionally, the  variational form \eqref{harmonic01} implies that:
\begin{description}
        \hypertarget{ex0}{}
    \item[\textmd{(ex.0)}]the operator $ \mathfrak{hm}_{\Gamma} : H^{\frac{1}{2}}(\Gamma) \to H^1(\Omega) $ is a bounded linear isomorphism, and in particular, there exists a positive constant $ C_{\Omega}^{\rm hm} $ satisfying
\begin{equation}\label{hm_const}
    |[\varrho]^{\rm hm}|_{H^1(\Omega)} \leq C_{\Omega}^{\rm hm} |\varrho|_{H^{\frac{1}{2}}(\Gamma)}, \mbox{ for any $ \varrho \in H^{\frac{1}{2}}(\Gamma) $;}
\end{equation}
        \vspace{-4ex}
        \hypertarget{ex1}{}
\item[\textmd{(ex.1)}]if $ r \in \R $, and $ r \leq \varrho \in H^{\frac{1}{2}}(\Gamma) $ (resp. $ r \geq \varrho  \in H^{\frac{1}{2}}(\Gamma) $), then $ r \leq [\varrho]^{\rm hm} \in H^1(\Omega) $ (resp. $ r \geq [\varrho]^{\rm hm} \in H^1(\Omega) $).
\end{description}
Moreover, by using the composition $ \mathfrak{hm}_{{\partial (\overline{\Omega}^{\rm C})}} \circ \mathfrak{tr}_{\Gamma} : H^{1}(\Omega) \to H^1(\overline{\Omega}^{\rm C}) $, we may suppose:
\begin{description}
        \hypertarget{ex2}{}
    \item[\textmd{(ex.2)}]if $ r \in \R $, and $ r \leq u \in H^{1}(\Omega) $ (resp. $ \sigma \geq r \in  H^{1}(\Omega) $), then $ r \leq [u]^{\rm ex} \in H^{1}(\R^N) $ (resp. $ r \geq [u]^{\rm ex} \in H^{1}(\R^N) $).
\end{description}
\end{rem}

%\bibliography{sinst43}
\providecommand{\href}[2]{#2}
\providecommand{\arxiv}[1]{\href{http://arxiv.org/abs/#1}{arXiv:#1}}
\providecommand{\url}[1]{\texttt{#1}}
\providecommand{\urlprefix}{URL }

\end{document}